\documentclass{amsart} 

\usepackage{amssymb}
\usepackage{amsmath}
\usepackage{amsthm}
\usepackage{amscd,mdwlist}
\usepackage[margin=1in]{geometry}
\usepackage{hyperref}
\usepackage{color}
\usepackage{cite}
\usepackage{bbm}

\theoremstyle{plain}
\newtheorem{theorem}{Theorem}[section]
\newtheorem{proposition}{Proposition}[section]
\newtheorem{corollary}{Corollary}[section]
\newtheorem{lemma}{Lemma}[section]
\newtheorem{definition}{Definition}[section]

\theoremstyle{remark}
\newtheorem{remark}{Remark}[section]
\newtheorem{examples}{Examples}[section]
\newtheorem{assumption}{Assumption}[section]

\DeclareMathOperator*{\essinf}{ess\,inf}

\DeclareMathOperator{\supp}{supp}
\DeclareMathOperator{\diam}{diam}

\DeclareMathOperator{\diag}{diag}

\DeclareMathOperator{\diverg}{div}
\DeclareMathOperator{\im}{Im}
\DeclareMathOperator{\id}{id}

\DeclareMathOperator{\ext}{Ext}

\DeclareMathOperator{\D}{\mathcal{D}}
\DeclareMathOperator{\E}{\mathcal{E}}

\begin{document}

\title{Approximation of partial differential equations on compact resistance spaces}

\author{Michael Hinz$^1$, Melissa Meinert$^2$}
\thanks{$^1$, $^2$ Research supported in part by the DFG IRTG 2235: 'Searching for the regular in the irregular: Analysis of singular and random systems'.}
\address{$^1$Fakult\"at f\"ur Mathematik, Universit\"at Bielefeld, Postfach 100131, 33501 Bielefeld, Germany}
\email{mhinz@math.uni-bielefeld.de}
\address{$^2$Fakult\"at f\"ur Mathematik, Universit\"at Bielefeld, Postfach 100131, 33501 Bielefeld, Germany}
\email{mmeinert@math.uni-bielefeld.de}

\begin{abstract}
We consider linear partial differential equations on resistance spaces that are uniformly elliptic and parabolic in the sense of quadratic forms and involve abstract gradient and divergence terms. Our main interest is to provide graph and metric graph approximations for their unique solutions. For families of equations with different coefficients on a single compact resistance space we prove that solutions have accumulation points with respect to the uniform convergence in space, provided that the coefficients remain bounded. If in a sequence of equations the coefficients converge suitably, the solutions converge uniformly along a subsequence. For the special case of local resistance forms on finitely ramified sets we also consider sequences of resistance spaces approximating the finitely ramified set from within. Under suitable assumptions on the coefficients (extensions of) linearizations of the solutions of equations on the approximating spaces accumulate or even converge uniformly along a subsequence to the solution of the target equation on the finitely ramified set. The results cover discrete and metric graph approximations, and both are discussed. 
\tableofcontents
\end{abstract}


\maketitle

\section{Introduction}

For several classes of fractal spaces, such as for instance p.c.f. self-similar sets, \cite{BarlowPerkins, Ki89, Ki93, Ki93b, Ki01, Ku89, Ku93}, classical Sierpinski carpets, \cite{BB89, BBKT10}, certain Julia sets, \cite{RogersTeplyaev10}, Laaks\o~ spaces, \cite{Steinhurst}, diamond lattice type fractals, \cite{AkkermansDunneTeplyaev, Alonso, HamblyKumagai}, and certain random fractals, \cite{Hambly92, Hambly97}, the existence of resistance forms in the sense of \cite{Ki03, Ki12} has been proved. This allows to establish a Dirichlet form based analysis, \cite{BH91, FOT94, MaRoeckner92}, with respect to a given volume measure, and in particular, studies of partial differential equations on fractals, \cite{Ba, Ki01, Str06}. These results and many later developments based on them are motivated by a considerable body of modern research in physics suggesting that in specific situations fractal models may be much more adequate than classical ones. The difficulty in this type of analysis comes from the fact that on fractals many tools from traditional calculus (and even many tools used in the modern theory of metric measure spaces, see e.g., \cite{Gigli15, Gigli17}) are not available. 

For fractal counterparts of equations of linear second order equations, \cite{Evans, GT10}, that do not involve first order terms - such as Poisson or heat equations for Laplacians - many results are known, \cite{Ba, Ka12, Ki01, Ki08, Str06}, and there are also studies of related semilinear equations without first order terms, \cite{FaHu99, FaHu01}. More recently fractal counterparts for equations involving first order terms have been suggested, \cite{HRT13, HT13, HTams, HinzKochMeinert}, and a few more specific results have been obtained, see for instance \cite{HinzMeinert19+, LiuQian17}. The discussion of first order terms is of rather abstract nature, because on most fractals there is no obvious candidate for a gradient operator; instead, it has to be constructed from a given bilinear form in a subsequent step, \cite{CS03, CS09, HRT13, IRT12}. (An intuitive argument why this construction cannot be trivial is the fact that for self-similar fractals, endowed with natural Hausdorff type volume measures, volume and energy are typically singular, \cite{BBST99, Hino05,Hino08,Hino10}.) For a study of, say, counterparts of second order equations, \cite[Section 8]{GT10}, involving abstract gradient and divergence terms, it therefore seems desirable to establish results which indicate that the equations have the correct physical meaning. 

In this article we consider analogs of linear elliptic and parabolic equations with first order terms on locally compact separable resistance spaces, \cite{Ki03, Ki12}. We wish to point out that we use the word 'elliptic' in a very broad (quadratic form) sense - the principal parts of our operators should rather be seen as fractal generalizations of hypoelliptic operators.
Under suitable assumptions the equations admit unique weak respectively semigroup solutions, Corollaries \ref{C:solutionelliptic} and \ref{C:solutionparabolic}. We prove that if the resistance space is compact and we are given bounded sequences of coefficients, the corresponding solutions have uniform accumulation points, Corollary \ref{C:KigamiAA}. If the sequences of coefficients converge, then the corresponding solutions converge in the $L^2$-sense and uniformly along subsequences, Theorem \ref{T:approxsame}. For certain local resistance forms on finitely ramified sets, \cite{IRT12, T08}, we introduce an approximation scheme along varying spaces, general enough to accommodate both discrete and metric graph approximations. If the coefficients are bounded in a suitable manner, extensions of linearizations of solutions to the equations on the approximating spaces have uniform accumulation points on the target space, Corollary \ref{C:KSaccu}. If the coefficients are carefully chosen, the solutions converge in an $L^2$-sense and the mentioned extensions converge uniformly along subsequences, Theorem \ref{T:KS_spectral}. Combining these results, we obtain an approximation for more general coefficients, Theorem \ref{T:general}. 

For resistance forms on discrete and metric graphs the abstract gradient operators admit more familiar expressions, Examples \ref{Ex:discrete1} and \ref{Ex:metric1}, and the bilinear forms associated with linear equations can be understood in terms of the well-known analysis on graphs and metric graphs, Examples \ref{Ex:discrete2} and \ref{Ex:metric2}, see for instance \cite{Grigoryan, KLWbook} and \cite{Mugnolo14, Post12}. The approximation scheme itself is of first order in the sense that it relies on the use of piecewise linear respectively piecewise harmonic functions, and it resembles familiar finite element methods. One motivation to use this approach is that pointwise restrictions of piecewise harmonic functions on, say, the Sierpinski gasket, are of finite energy on approximating metric graphs, \cite{HinzMeinert19+}, but for general energy finite functions on the Sierpinski gasket this is not true - the corresponding trace spaces on the metric graphs are fractional Sobolev spaces of order less than one, see for instance \cite{Str16} and the reference cited there (and \cite{HinoKumagai06} for related results). Of course first order approximations have a certain scope and certain limitations. But keeping these in mind, we can certainly view our results as a strong first indication that the abstractly formulated equations on the target space have the desired physical meaning, because their solutions appear as natural limits of solutions to similar equations on more familiar geometries, where they are better understood. The established approximation scheme also provides a computational tool which could be used for numerical simulations. Our results hold under rather minimal assumptions on the volume measure on the target space. For instance, in the situation of p.c.f. self-similar structures it is not necessary to specialize to self-similar Hausdorff measures, \cite{Ba, Ki01}, or to energy dominant Kusuoka type measures, \cite{Ku89, Ku93, Hino08, Hino10}.

In \cite[Section 6]{StrU00} a finite element method for a Poisson type equation on p.c.f. self-similar fractals was discussed, and the use of an equivalent scalar product and a related orthogonal projection made it possible to regard the approximation itself as the solution of a closely related equation. For equations involving divergence and gradient terms one cannot hope for a similarly simple mechanism. On the other hand, the construction of resistance forms itself is based on discrete approximations, \cite{Ki89,Ki93, Ki93b, Ki01}, and in symmetric respectively self-adjoint situations this can be used to obtain approximation results on the level of resistance forms, \cite{Croydon}, or Dirichlet forms, \cite{PostSimmer17, PostSimmer18}. In the latter case the dynamics of a partial differential equation of elliptic or parabolic type for self-adjoint operators comes into play, and it can be captured using spectral convergence results, \cite{Mosco94, RS80}, possibly along varying Hilbert spaces, \cite{KuwaeShioya03, Post12}. The equations we have in mind are governed by operators that are not necessarily symmetric, but under some conditions on the coefficients they are still sectorial, \cite{Kato80, MaRoeckner92}. This leads to the question of how to implement similar types on convergences for sectorial operators, and one arrives to a situation similar to those in \cite{MugnoloNittkaPost13} or \cite{Suzuki18}. The main difficulty is how to correctly implement the convergence of drift and divergence terms. With \cite{HinzMeinert19+, PostSimmer17, PostSimmer18} in mind a first reflex might be to try to verify a type of generalized norm resolvent convergence as in \cite{MugnoloNittkaPost13}, and to do so the first order terms would have to fit the estimate in their Definition 2.3, where particular (2.7e) is critical. For convergent sequences of drift and divergence coefficients, \cite{HRT13}, on a single resistance space one can establish this estimate with trivial identification operators (as adressed in their Example 2.5), but in the case of varying spaces the interaction of identification operators with the first order calculus seems too difficult to handle. The convergence results in \cite[Section 4]{Suzuki18} use the variational convergence studied in \cite{Toelle06, Toelle10}, which generalizes the Mosco type convergence, \cite{Mosco94}, for generalized forms, \cite{Hino98}, to the setup of varying Hilbert spaces, \cite{KuwaeShioya03}, and encodes a generalization of strong resolvent convergence. Also in the present article this variational convergence is used as a key tool: We verify the adequate Mosco type convergence along varying spaces of the bilinear forms associated with the equations, and by \cite{Hino98} and \cite{Toelle06, Toelle10} we can then conclude the $L^2$-type convergence of the solutions, see Theorem \ref{T:convergence of forms}. A significant difference between \cite[Section 4]{Suzuki18} and our results is the way the first order terms are handled. There the approach from \cite{AmbrosioStraTrevisan17} is used, which relies heavily on having a carr\'e du champ operator, \cite{BH91}. But this is an assumption which we wish to avoid, because - as mentioned above - interesting standard examples do not satisfy it. The target spaces for the approximation result along varying spaces that we implement are assumed to be finitely ramified sets, \cite{IRT12, T08}, endowed with local regular 
resistance forms, \cite{Ki03, Ki12}, satisfying certain assumptions. This class of fractals contains many interesting examples, \cite{Hambly92, Hambly97, HamblyNyberg03, Meyers, RogersTeplyaev10}, and in particular, p.c.f. self-similar fractals with  regular harmonic structures, \cite{Ki01}, but it does not contain Sierpinski carpets, \cite{BB89, BBKT10}. The cell structure of a finitely ramified set allows a transparent use of identification operators based on piecewise harmonic functions. The key property of resistance spaces that energy finite functions are continuous compensates the possible energy singularity of a given volume measure to a certain extent, in particular, we can use an inequality originally shown in \cite{HR16} when handling the first order terms in the presence of an energy singular measure. Uniform energy bounds and the compactness of the space then allow to use Arzela-Ascoli type arguments to obtain subsequential limits in the sense of uniform convergence. Together with the $L^2$-type limit statements produced by the variational convergence these limit points are then identified to be the solutions on the target space. 

The use of variational convergence schemes to study dynamical phenomena on certain geometries is a well-established idea, see for instance \cite{Kasue02, Kasue06, KuwaeShioya03}. It was already a guiding theme in \cite{Mosco94}, and related results in different setups have been studied in a number of recent articles, see for instance \cite{AmbrosioHonda17, ExnerPost09, Hinz09, HinzTeplyaev15, Kolesnikov05, Kolesnikov06, MugnoloNittkaPost13, Post06, Post12, PostSimmer17, PostSimmer18, Suzuki18}. For fractal spaces variational schemes can provide a certain counterpart to homogenization: In the latter the effect of a complicated microstructure can be encoded in an equation for an effective material if the problem is viewed at a certain mesoscopic scale. In analysis on fractals it may not be possible to find such a scale, and it is desirable to have a more direct understanding of how the microstructure determines analysis. This
typically leads to non-classical rescalings when passing from discrete to continuous or from smooth to fractal. Although the present study is written specifically for resistance spaces, some aspects of the approximation scheme in Section \ref{S:varying} might also provide some guidance for schemes along varying spaces for non-symmetric local or non-local operators on non-resistance spaces.

In Section \ref{S:resistanceforms} we recall basics from the theory of resistance forms and explain items of the related first order calculus. We discuss bilinear forms including drift and divergence terms in Section \ref{S:linearequations}, and follow standard methods, \cite{Evans, FiKu04, GT10}, to state existence, uniqueness and energy estimates for weak solutions to elliptic equations and (semigroup) solutions to parabolic equations. In Section \ref{S:single} we prove convergence results for equations on a single compact resistance space. We first discuss suitable conditions on the coefficients, then accumulation points and then strong resolvent convergence. Section \ref{S:varying} contains the approximation scheme along varying spaces for finitely ramified sets. We first state the basic assumptions and record some immediate consequences, then survey conditions on the coefficients and finally state the accumulation and convergence results. Section \ref{S:examples} discusses discrete approximations (Subsection \ref{SS:discretegraphs}), including classes of examples, metric graph approximations (Subsection \ref{SS:metricgraphs}), and short remarks on possible generalizations. Section \ref{S:restriction} contains an auxiliary result on the restriction of vector fields for finitely ramified sets.

We follow the habit to write $\mathcal{E}(u)$ for $\mathcal{E}(u,u)$ if $\mathcal{E}$ is a bilinear quantity depending on two arguments and both arguments are the same.

\section*{Acknowledgements}

We thank Olaf Post, Jan Simmer, Alexander Teplyaev and Jonas T\"olle for helpful and inspiring conversations.

\section{Resistance forms and first order calculus}\label{S:resistanceforms}

We recall the definition of resistance form, due to Kigami, see \cite[Definition 2.3.1]{Ki01} or \cite[Definition 2.8]{Ki03}. By $\ell(X)$ we denote the space of real valued functions on a set $X$.
\begin{definition}\label{D:resistanceform}
A resistance form $(\mathcal{E},\mathcal{F})$ on a set $X$ is a pair such that
\begin{enumerate}
\item[(i)] $\mathcal{F}\subset \ell(X)$ is a linear subspace of $\ell(X)$ containing the constants and $\mathcal{E}$ is a non-negative definite symmetric bilinear form on $\mathcal{F}$ with $\mathcal{E}(u)=0$ if and only if $u$ is constant.
\item[(ii)] Let $\sim$ be the equivalence relation on $\mathcal{F}$ defined by $u \sim v$ if and only if
$u-v$ is constant on $X$. Then $(\mathcal{F}/\sim, \mathcal{E})$ is a Hilbert space.
\item[(iii)] If $V\subset X$ is finite and $v\in\ell(V)$ then there is a function $u\in\mathcal{F}$ so that $u\,\bigr|_{V}=v$.
\item[(iv)] For $x,y\in X$
\begin{equation*}
	R(x,y):=\sup\Bigl\{ \frac{(u(x)-u(y))^{2}}{\mathcal{E}(u)}:u\in\mathcal{F}, \mathcal{E}(u)>0\Bigr\}<\infty.
\end{equation*}
\item[(v)] If $u\in\mathcal{F}$ then $\bar{u}:=\max(0,\min(1,u(x)))\in\mathcal{F}$ and $\mathcal{E}(\bar{u})\leq\mathcal{E}(u)$.
\end{enumerate}
\end{definition}
To $R$ one refers as the \emph{resistance metric} associated with $(\mathcal{E},\mathcal{F})$, \cite[Definition 2.11]{Ki03}, and to the pair $(X,R)$, which forms a metric space, \cite[Proposition 2.10]{Ki03}, we refer as \emph{resistance space}. All functions $u\in\mathcal{F}$ are continuous on $X$ with respect to the resistance metric, more precisely, we have 
\begin{equation}\label{E:resistanceest}
|u(x)-u(y)|^2\leq R(x,y)\mathcal{E}(u),\quad u\in\mathcal{F},\quad x,y\in X.
\end{equation}
For any finite subset $V\subset X$ the restriction of $(\mathcal{E},\mathcal{F})$ to $V$ is the resistance form $(\mathcal{E}_{V}, \ell(V)) $ defined by
\begin{equation}\label{E:EV}	
	\mathcal{E}_{V}(v)=\inf\Bigl\{\mathcal{E}(u):u\in\mathcal{F}, u\,\bigr|_{V}=v\Bigr\},
	\end{equation}
where a unique infimum is achieved. If $V_{1}\subset V_{2}$ and both are finite, then $(\mathcal{E}_{V_{2}})_{V_{1}}=\mathcal{E}_{V_{1}}$.

We assume $X$ is a nonempty set and $(\mathcal{E},\mathcal{F})$ is a resistance form on $X$ so that $(X,R)$ is  \emph{separable}. Then  there exists a sequence $(V_m)_m$ of finite subsets $V_m\subset X$ with $V_m\subset V_{m+1}$, $m\geq 1$, and $\bigcup_{m\geq 0} V_m$ dense in $(X,R)$. For any such sequence $(V_m)_m$ we have 
\begin{equation}\label{E:limitform}
\mathcal{E}(u)=\lim_{m}\mathcal{E}_{V_{m}}(u),\quad  u\in\mathcal{F},
\end{equation}
as proved in \cite[Proposition 2.10 and Theorem 2.14]{Ki03}. Note that for any $u\in\mathcal{F}$ the sequence $(\mathcal{E}_{V_{m}}(u))_m$ is non-decreasing. Each $\mathcal{E}_{V_m}$ is of the form 
\begin{equation}\label{E:approxbydiscreteforms}
\mathcal{E}_{V_m}(u)=\frac{1}{2}\sum_{p\in V_m}\sum_{q\in V_m}c(m;p,q)(u(p)-u(q))^2,\quad u\in\mathcal{F},
\end{equation}
with constants $c(m;p,q)\geq 0$ symmetric in $p$ and $q$.
 
We further assume that $(X,R)$ is \emph{locally compact} and that $(\mathcal{E},\mathcal{F})$ is \emph{regular}, i.e., that the space $\mathcal{F}\cap C_c(X)$ is uniformly dense in the space $C_c(X)$ of continuous compactly supported functions on $(X,R)$, see \cite[Definition 6.2]{Ki12}. Definition \ref{D:resistanceform} (v) implies that $\mathcal{F}\cap C_c(X)$ is an algebra under pointwise multiplication and
\begin{equation}\label{E:pointwisemult}
\mathcal{E}(fg)^{1/2}\leq \left\|f\right\|_{\sup}\mathcal{E}(g)^{1/2}+\left\|g\right\|_{\sup}\mathcal{E}(f)^{1/2},\quad f,g\in\mathcal{F}\cap C_c(X),
\end{equation}
see \cite[Lemma 6.5]{Ki12}.

To introduce the first order calculus associated with $(\mathcal{E},\mathcal{F})$, let $\ell_a(X\times X)$ denote the space of all real valued antisymmetric functions on $X\times X$ and write
\begin{equation}\label{E:action}
(g\cdot v)(x,y):=\overline{g}(x,y)v(x,y),\quad x,y\in X, 
\end{equation}
for any $v\in \ell_a(X\times X)$ and $g\in C_c(X)$, where 
\[\overline{g}(x,y):=\frac12(g(x)+g(y)).\] 
Obviously $g\cdot v \in \ell_a(X\times X)$, and (\ref{E:action}) defines an action of $C_c(X)$ on $\ell_a(X\times X)$, turning it into a module. By $d_u:\mathcal{F}\cap C_c(X)\to \ell_a(X\times X)$ we denote the universal derivation, 
\begin{equation}\label{E:universalder}
d_uf(x,y):=f(x)-f(y), \quad x,y\in X,
\end{equation}
and by 
\begin{equation}\label{E:Omega}
\Omega_a^1(X):=\left\lbrace \sum_i g_i \cdot d_u f_i: g_i\in C_c(X), f_i \in \mathcal{F}\cap C_c(X)\right\rbrace,
\end{equation}
deviating slightly from the notation used in \cite{HinzMeinert19+}, the submodule of $\ell_a(X\times X)$ of finite linear combinations of functions of form $g\cdot d_uf$. A quick calculation shows that for $f,g\in \mathcal{F}\cap C_c(X)$ we have $d_u(fg)=f\cdot d_ug +g\cdot d_uf$.

On $\Omega_a^1(X)$ we can introduce a symmetric nonnegative definite bilinear form $\left\langle\cdot,\cdot\right\rangle_{\mathcal{H}}$ by extending 
\begin{equation}\label{E:scalarprod}
\left\langle g_1\cdot d_uf_1, g_2\cdot d_uf_2\right\rangle_{\mathcal{H}}:=\lim_{m\to\infty} \frac{1}{2}\sum_{p\in V_m}\sum_{q\in V_m} c(m;p,q) \overline{g_1}(p,q)\overline{g_2}(p,q)d_uf_1(p,q)d_uf_2(p,q)
\end{equation}
linearly in both arguments, respectively, and we write $\left\|\cdot\right\|_{\mathcal{H}}=\sqrt{\left\langle\cdot,\cdot\right\rangle_{\mathcal{H}}}$ for the associated Hilbert seminorm. In Proposition \ref{P:indep} below we will verify that the definition of $\left\langle\cdot,\cdot\right\rangle_{\mathcal{H}}$ does not depend on the choice of the sequence $(V_m)_m$.

 We factor $\Omega_a^1(X)$ by the elements of zero seminorm and obtain the space $\Omega^1_a(X)/\ker \left\|\cdot\right\|_{\mathcal{H}}$. Given an element $\sum_i g_i \cdot d_uf_i$ of $\Omega_a^1(X)$ we write $\big[\sum_i g_i\cdot d_uf_i\big]_\mathcal{H}$ to denote its equivalence class. Completing $\Omega^1_a(X)/\ker \left\|\cdot\right\|_{\mathcal{H}}$ with respect to $\left\|\cdot\right\|_{\mathcal{H}}$ we obtain a Hilbert space $\mathcal{H}$, we refer to it as the \emph{space of generalized $L^2$-vector fields associated with $(\mathcal{E},\mathcal{F})$}. This is a version of a construction introduced in \cite{CS03, CS09} and studied in \cite{BK16, HR16, HRT13, HT13, HTams, HT-fgs5, IRT12, LiuQian17}, see also the related sources \cite{Eb99, Gigli15, Gigli17, W00}. The basic idea is much older, see for instance \cite[Exercise 5.9]{BH91}, it dates back to ideas of Mokobodzki and LeJan.

The action (\ref{E:action}) induces an action of $C_c(X)$ on $\mathcal{H}$: Given $v\in \mathcal{H}$ and $g\in C_c(X)$, let $(v_n)_n\subset \Omega_a^1(X)$ be such that $\lim_n [v_n]_\mathcal{H}=v$ in $\mathcal{H}$ and define $g\cdot v \in \mathcal{H}$ by 
$g\cdot v:=\lim_n [g\cdot v_n]_\mathcal{H}$. Since (\ref{E:action}) and (\ref{E:scalarprod}) imply 
\begin{equation}\label{E:boundedaction}
\left\|g\cdot v\right\|_{\mathcal{H}}\leq \left\|g\right\|_{\sup}\left\|v\right\|_{\mathcal{H}},
\end{equation}
it follows that the definition of $g\cdot v$ is correct. Given $f\in\mathcal{F}\cap C_c(X)$, we denote the $\mathcal{H}$-equivalence class of the universal derivation $d_uf$ as in (\ref{E:universalder}) by $\partial f$. By the preceding discussion we observe $[g\cdot d_uf]_{\mathcal{H}}=g\cdot\partial f$ for all $f\in \mathcal{F}\cap C_c(X)$ and $g\in C_c(X)$. It also follows that the map $f\mapsto \partial f$ defines a derivation operator 
\[\partial: \mathcal{F}\cap C_c(X)\to\mathcal{H}\] 
which satisfies the identity $\left\|\partial f\right\|_{\mathcal{H}}^2=\mathcal{E}(f)$ for any $f\in \mathcal{F}\cap C_c(X)$ and the Leibniz rule $\partial(fg)=f\cdot \partial g+g\cdot \partial f$ for any $f,g\in \mathcal{F}\cap C_c(X)$.

To show the independence of $\left\langle\cdot,\cdot\right\rangle_{\mathcal{H}}$ of the choice of the sequence $(V_m)_m$ in 
(\ref{E:scalarprod}) and to formulate later statements, we consider energy measures. For $f\in \mathcal{F}\cap C_c(X)$ there is a unique finite Radon measure $\nu_f$ on $X$ satisfying
\begin{equation}\label{E:energymeasure}
\int_Xg\:d\nu_f = \mathcal{E}(fg,f)-\frac12\mathcal{E}(f^2,g),\quad g\in\mathcal{F}\cap C_c(X),
\end{equation}
the \emph{energy measure of $f$}, see for instance \cite{Ka12, Ku93, Ki08, T08} and or \cite{FOT94, Hino03, Hino05, Hino08, Hino10, Hinz16}. It is not difficult to see that for any $f\in \mathcal{F}\cap C_c(X)$ and $g\in C_c(X)$ we have
\begin{equation}\label{E:energymeasureapprox}
\int_Xg\:d\nu_f=\frac12\lim_{m\to\infty}\sum_{p\in V_m}\sum_{q\in V_m}c(m;p,q)g(p)(d_u f(p,q))^2.
\end{equation}
\emph{Mutual energy measures} $\nu_{f_1,f_2}$ for $f_1,f_2\in \mathcal{F}\cap C_c(X)$ are defined using (\ref{E:energymeasure}) and polarization.

According to the Beurling-Deny decomposition of $(\mathcal{E},\mathcal{F})$, see \cite[Th\'eor\`eme 1]{Allain75} (or \cite[Section 3.2]{FOT94} for a different context), there exist a uniquely determined symmetric bilinear form $\mathcal{E}^c$ on $\mathcal{F}\cap C_c(X)$ satisfying $\mathcal{E}^c(f,g)=0$ whenever $f\in \mathcal{F}\cap C_c(X)$ is constant on an open neighborhood of the support of $g\in \mathcal{F}\cap C_c(X)$ and
a uniquely determined symmetric nonnegative Radon measure $J$ on $X\times X\setminus \diag$ such 
\begin{equation}\label{E:BD}
\mathcal{E}(f)=\mathcal{E}^c(f)+\int_X \int_X (d_uf(x,y))^2J(dxdy),\quad f\in \mathcal{F}\cap C_c(X).
\end{equation}
The form $\mathcal{E}^c$ is called the local part of $\mathcal{E}$, and by $\nu_f^c$ we denote the local part of the energy measure of a function $f\in\mathcal{F}\cap C_c(X)$, i.e. the finite Radon measure (uniquely) defined as in (\ref{E:energymeasure}) with $\mathcal{E}^c$ in place of $\mathcal{E}$. From (\ref{E:energymeasure}) and (\ref{E:BD}) it is immediate that 
\begin{equation}\label{E:BDmeas}
\int_X g d\nu_f=\int g d\nu_f^c +\int_X\int_X g(x) (d_uf(x,y))^2J(dxdy),\quad f,g\in \mathcal{F}\cap C_c(X).
\end{equation}

\begin{proposition}\label{P:indep}
Suppose that closed balls in $(X,R)$ are compact. Then for any $f_1,f_2\in\mathcal{F}\cap C_c(X)$ and $g_1,g_2\in C_c(X)$ we have 
\[\left\langle g_1\cdot \partial f_1, g_2\cdot \partial f_2\right\rangle_{\mathcal{H}}=\int_X g_1g_2\:d\nu_{f_1,f_2}^{(c)}+\int_X\int_X\overline{g_1}(x,y)\overline{g_2}(x,y)d_uf_1(x,y)d_uf_2(x,y)\:J(dxdy).\]
In particular, the definition of the bilinear form $\left\langle\cdot,\cdot\right\rangle_{\mathcal{H}}$ is independent of the choice of the sets $V_m$.
\end{proposition}

\begin{proof}
Standard arguments show that for all $v\in C_c(X\times X\setminus \diag)$ we have 
\begin{equation}\label{E:jumppart}
\frac12\lim_{\varepsilon\to 0}\lim_{m\to\infty}\sum_{x\in V_m}\:\sum_{y\in V_m,R(x,y)>\varepsilon}c(m;x,y)v(x,y)=\int_X\int_X v(x,y)J(dxdy),
\end{equation}
see for instance \cite[Section 3.2]{FOT94}. The particular case $v=d_uf$, together with (\ref{E:BD}), then implies that
\begin{equation}\label{E:localform}
\mathcal{E}^c(f)=\frac12\lim_{\varepsilon \to 0}\lim_{m\to \infty} \sum_{x\in V_m}\:\sum_{y\in V_m, R(x,y)\leq \varepsilon}c(m;x,y)(d_uf(x,y))^2
\end{equation}
for any $f\in\mathcal{F}\cap C_c(X)$. We claim that given such $f$ and $g\in C_c(X)$, 
\begin{equation}\label{E:claim}
\int_Xg^2d\nu_f^c=\frac12\lim_{\varepsilon \to 0}\lim_{m\to \infty} \sum_{x\in V_m}\:\sum_{y\in V_m, R(x,y)\leq \varepsilon}c(m;x,y)\overline{g}(x,y)^2(d_uf(x,y))^2.
\end{equation}
This follows from (\ref{E:energymeasure}) and (\ref{E:localform}) and the fact that 
\[\lim_{\varepsilon\to 0}  \lim_{m\to\infty} \sum_{x\in V_m}\:\sum_{y\in V_m, R(x,y)\leq \varepsilon}c(m;x,y) (d_ug(x,y))^2(d_uf(x,y))^2=0,\]
which can be seen following the arguments in the proof of \cite[Lemma 3.1]{HinzMeinert19+}. Combining (\ref{E:jumppart}), applied to $v=g\cdot d_uf$, and (\ref{E:claim}), we obtain the desired result by polarization.
\end{proof}

As a consequence of Proposition \ref{P:indep} and dominated convergence we can define $g\cdot v$ for all $v\in\mathcal{H}$
and $g\in C_b(X)$ and (\ref{E:boundedaction}) remains true for such $v$ and $g$. Note also that if $v_1, v_2\in\mathcal{H}$ and $g\in C_b(X)$ then 
\[\left\langle g\cdot v_1, v_2\right\rangle_{\mathcal{H}}=\left\langle v_1, g\cdot v_2\right\rangle_{\mathcal{H}}.\]

In the special cases of finite graphs, \cite{Grigoryan, KLWbook}, and compact metric graphs, \cite{FKW07, KS99, KS00, Ku04, Ku05, Post12, Mugnolo14}, the space $\mathcal{H}$ and the operator $\partial$ appear in a more familiar form.

\begin{examples}\label{Ex:discrete1}
If $(V,\omega)$ is a finite simple weighted (unoriented) graph, \cite{Grigoryan}, then 
\[\mathcal{E}(u)=\frac12\sum_{p\in V}\sum_{q\in V}\omega(p,q)(u(p)-u(q))^2,\quad u\in\ell(V),\]
is a resistance form on the finite set $V$, and it makes it a compact resistance space.  In this case $\mathcal{H}$ is isometrically isomorphic to the space $\ell^2_a(V\times V\setminus \diag, \omega)$ of real-valued antisymmetric functions on $V\times V\setminus \diag$, endowed with the usual $\ell^2$-scalar product, and for any $f\in \ell(V)$ the gradient $\partial f\in \mathcal{H}$ of $f$ is the image of $d_uf \in \ell^2_a(V\times V\setminus \diag, \omega)$ under this isometric isomorphism, see for instance \cite[Section 3]{Hinz15}.
\end{examples}

\begin{examples}\label{Ex:metric1}
Let $(V,E)$ be a finite simple (unoriented) graph and $(l_e)_{e\in E}$ a finite sequence of positive numbers. Consider 
the metric graph $\Gamma$ obtained by identifying each edge $e\in E$ with an oriented copy of the interval $(0,l_e)$ and considering different copies to be joined at the vertices the respective edges have in common. Then the set $X_\Gamma=V\cup\bigcup_{e\in E} e$, endowed with a natural topology, becomes a compact metric space. For each $u\in C(X_\Gamma)$ let
\[\mathcal{E}_\Gamma(u)=\sum_{e\in E} \mathcal{E}_e (u_e),\quad \text{where }\quad \mathcal{E}_e(u_e)=\int_0^{l_e}(u_e(s))^2ds, \quad e\in E,\]
and $u_e$ is the restriction of $u$ to $e\in E$. If $\dot{W}^{1,2}(X_\Gamma)$ denotes the space of all $u\in C(X_\Gamma)$ such that $\mathcal{E}(u)<+\infty$ then $(\mathcal{E}_\Gamma, \dot{W}^{1,2}(X_\Gamma))$ is a resistance form making $X_\Gamma$ a compact resistance space. The space $\mathcal{H}$ is isometrically isomorphic to $\bigoplus_{e\in E} L^2(0,l_e)$, and for any $f\in \dot{W}^{1,2}(X_\Gamma)$ the gradient $\partial f\in\mathcal{H}$ is the image under this isometric isomorphism of $(f_e')_{e \in E}$, where $f_e'\in L^2(0,l_e)$ denotes the usual first derivative of $f_e$, seen as a function on $(0,l_e)$. For more precise descriptions and further details see \cite{BK16, IRT12}. In Subsection \ref{SS:metricgraphs} we consider a scaled variant of this construction as in \cite{HinzMeinert19+}. 
\end{examples}

\begin{remark}
For convenience the above construction of the space $\mathcal{H}$ and the operator $\partial$ is formulated for resistance spaces. However, we wish to point out that the original construction does not need the specific properties of a resistance space, it can be formulated for Dirichlet forms in very high generality, \cite{CS03}.
\end{remark}

\section{Linear equations of elliptic and parabolic type}\label{S:linearequations}

The considerations in this section are straightforward from standard theory for partial differential equations, \cite[Chapter 8]{GT10}, and Dirichlet forms, \cite{FOT94}, see for instance \cite{FiKu04}. 

Let $(\mathcal{E},\mathcal{F})$ be a resistance form on a nonempty set $X$ so that $(X,R)$ is separable and locally compact and assume that $(\mathcal{E},\mathcal{F})$ is regular. In addition, assume that closed balls are compact. Let $\mu$ be a Borel measure on $(X,R)$ such that for any $x\in X$ and $R>0$ we have $0<\mu(B(x,R))<+\infty$. Then by \cite[Theorem 9.4]{Ki12} the form $(\mathcal{E},\mathcal{F}\cap C_c(X))$ is closable on $L^2(X,\mu)$ and its closure, which we denote by $(\mathcal{E},\mathcal{D}(\mathcal{E}))$, is a regular Dirichlet form. In general we have $\mathcal{D}(\mathcal{E})\subset \mathcal{F}\cap L^2(X,\mu)$, and in the special case that $(X,R)$ is compact, $\mathcal{D}(\mathcal{E})=\mathcal{F}$, \cite[Section 9]{Ki12}. Given $\alpha>0$ we write 
\begin{equation}\label{E:addalpha}
\mathcal{E}_\alpha(f,g):=\mathcal{E}(f,g)+\alpha \left\langle f,g\right\rangle_{L^2(X,\mu)},\quad f,g\in\mathcal{D}(\mathcal{E}),\end{equation}
and we use an analogous notation for other bilinear forms. Recall that we also write $\mathcal{E}(f)$ to denote $\mathcal{E}(f,f)$ and similarly for other bilinear quantities.

By the closedness of $(\mathcal{E},\mathcal{D}(\mathcal{E}))$ the derivation $\partial$, defined as in the preceding section, extends to a closed unbounded linear operator $\partial:L^2(X,\mu)\to\mathcal{H}$ with domain $\mathcal{D}(\mathcal{E})$, we write 
$\im \partial$ for the image of $\mathcal{D}(\mathcal{E})$ under $\partial$. The adjoint operator $(\partial^\ast, \mathcal{D}(\partial^\ast))$ of $(\partial,\mathcal{D}(\mathcal{E}))$ can be interpreted as minus the divergence operator, and for the generator $(\mathcal{L},\mathcal{D}(\mathcal{L}))$ of $(\mathcal{E},\mathcal{D}(\mathcal{E}))$ we have $\partial f\in \mathcal{D}(\partial^\ast)$ whenever $f\in\mathcal{D}(\mathcal{L})$, and in this case, $\mathcal{L}f=-\partial^\ast\partial f$.

\subsection{Closed forms} We call a symmetric bounded linear operator $a:\mathcal{H} \to\mathcal{H}$ a \emph{uniformly elliptic (in the sense of quadratic forms)} if there are universal constants $0<\lambda<\Lambda$ such that 
\begin{equation}\label{E:elliptic}
\lambda \left\|v\right\|_{\mathcal{H}}^2\leq \left\langle a\,v, v\right\rangle_{\mathcal{H}}\leq \Lambda \left\|v\right\|_{\mathcal{H}}^2,\quad v\in \mathcal{H}.
\end{equation}
As mentioned in the introduction, the phrase 'uniformly elliptic' is interpreted in a wide sense, and (\ref{E:elliptic}) rather corresponds to a sort of energy equivalence, see for instance \cite[Definition 2.17]{BBK06}. We follow \cite{FiKu04} and say that an element $b\in \mathcal{H}$ is in the \emph{Hardy class} if there are constants $\delta(b)\in (0,\infty)$ and $\gamma(b)\in [0,\infty)$ such that 
\begin{equation}\label{E:Hardy}
\left\|g\cdot b\right\|_{\mathcal{H}}^2\leq \delta(b)\mathcal{E}(g)+\gamma(b)\left\|g\right\|_{L^2(X,\mu)}^2,\quad g\in\mathcal{F}\cap C_c(X).
\end{equation}

Given uniformly elliptic $a$ as in (\ref{E:elliptic}), $b$, $\hat{b}\in \mathcal{H}$ in the Hardy class and $c\in L^\infty(X,\mu)$ we consider the bilinear form on  $\mathcal{F}\cap C_c(X)$ defined by
\begin{equation}\label{E:Q}
\mathcal{Q}(f,g)= \left\langle a\cdot\partial f,\partial g\right\rangle_\mathcal{H}-\left\langle g\cdot b,\partial f\right\rangle_{\mathcal{H}}-\big\langle f\cdot \hat{b}, \partial g\big\rangle_{\mathcal{H}}-\left\langle cf,g\right\rangle_{L^2(X,\mu)}, \quad f,g\in\mathcal{F}\cap C_c(X).
\end{equation}

We say that a densely defined bilinear form $(Q,\mathcal{D}(Q))$ on $L^2(X,\mu)$ is \emph{semibounded} if there exists some $C\geq 0$ such that $Q(f)\geq -C\left\|f\right\|_{L^2(X,\mu)}^2$, $f\in\mathcal{D}(Q)$. If in addition $(\mathcal{D}(Q),\widetilde{Q}_{C+1})$ is a Hilbert space, where $\widetilde{Q}$ denotes the symmetric part of $Q$, defined by
\begin{equation}\label{E:symmetricpart}
\widetilde{Q}(f,g)=\frac{1}{2} \left(Q(f,g) + Q(f,g) \right), \quad f,g\in \mathcal{D}(Q),
\end{equation}
then we call $(Q,\mathcal{D}(Q))$ a \emph{closed form}. In other words, we call $(Q,\mathcal{D}(Q))$ a closed form if $(\widetilde{Q},\mathcal{D}(Q))$ is a closed quadratic form in the sense of \cite[Section VIII.6]{RS80}. We say that a closed form 
$(Q,\mathcal{D}(Q))$ is \emph{sectorial} if there is a constant $K>0$ such that 
\[|Q_{C+1}(f,g)|\leq K\:Q_{C+1}(f)^{1/2}Q_{C+1}(g)^{1/2},\quad f,g\in\mathcal{D}(Q),\]
where $C$ is as above. In other words, we call a closed form $(Q,\mathcal{D}(Q))$ sectorial if $(Q_C,\mathcal{D}(Q))$ is a coercive closed form in the sense of \cite[Definition 2.4]{MaRoeckner92}.

The following proposition follows from standard estimates and (\ref{E:Hardy}), we omit its proof.

\begin{proposition}\label{P:Qisnice}\mbox{}
\begin{enumerate}
\item[(i)] Assume that $a:\mathcal{H}\to\mathcal{H}$ is symmetric and satisfies (\ref{E:elliptic}),  
 $c\in L^\infty(X,\mu)$ and $b,\hat{b}\in\mathcal{H}$ are in the Hardy class and such that 
\begin{equation}\label{E:smallfields}
\lambda_0:=\frac12\left(\lambda-\sqrt{\delta(b)}-\sqrt{\delta(\hat{b})}\right)>0.
\end{equation} 
Then $(\mathcal{Q},\mathcal{F}\cap C_c(X))$ is closable on $L^2(X,\mu)$, and its closure $(\mathcal{Q},\mathcal{D}(\mathcal{E}))$ is a sectorial closed form.
\item[(ii)] If in addition $c$ is such that
\begin{equation}\label{E:cisnice}
c_0:=\essinf_{x\in X} \left(-c(x)\right)-\frac{\gamma(b)+\gamma(\hat{b})}{2\lambda_0}>0
\end{equation}
then $(\mathcal{Q},\mathcal{D}(\mathcal{E}))$ satisfies the bounds
\begin{equation}\label{E:posdef}
\lambda_0\:\mathcal{E}(f)+c_0\left\|f\right\|_{L^2(X,\mu)}^2\leq \mathcal{Q}(f)\leq \Lambda_\infty\:\mathcal{E}(f)+c_\infty\left\|f\right\|_{L^2(X,\mu)}^2, \quad f\in \mathcal{D}(\mathcal{E}),
\end{equation}
where 
\begin{equation}\label{E:upperconst}
\Lambda_\infty:=\Lambda+\sqrt{\delta(b)}+\sqrt{\delta(\hat{b})}+1\quad \text{and}\quad c_\infty:=\frac{\gamma(b)+\gamma(\hat{b})}{2}+\left\|c\right\|_{L^\infty(X,\mu)}.
\end{equation}
\end{enumerate}
\end{proposition}

\begin{remark}
These conditions are chosen for convenience, we do not claim their optimality. Standard estimates using integrability conditions for vector fields, as for instance used in \cite{Suzuki18}, do not apply unless one assumes that energy measures are absolutely continuous with respect to $\mu$, an assumption we wish to avoid. 
\end{remark}

Suppose that the hypotheses of Proposition \ref{P:Qisnice} (i) are satisfied. Let $(\mathcal{L}^{(\mathcal{Q})}, \mathcal{D}(\mathcal{L}^{(\mathcal{Q})}))$ denote the infinitesimal generator of $(\mathcal{Q},\mathcal{D}(\mathcal{E}))$, that is, the 
unique closed operator on $L^2(X,\mu)$ associated with $(\mathcal{Q},\mathcal{D}(\mathcal{E}))$ by the identity 
\[\mathcal{Q}(f,g)=-\left\langle \mathcal{L}^\mathcal{Q} f,g\right\rangle_{L^2(X,\mu)},\quad f\in\mathcal{D}(\mathcal{L}^\mathcal{Q}),\ g\in\mathcal{D}(\mathcal{E}).\]
A direct calculation shows the following.

\begin{corollary} 
Let the hypotheses of Proposition \ref{P:Qisnice} (i) and (ii) be satisfied, let notation be as there and set 
\begin{equation}\label{E:possibleK}
K=\frac{1}{\lambda}\left(\Lambda +\sqrt{\delta(b)}+\sqrt{\delta(\hat{b})}+\sqrt{\gamma(b)}+\sqrt{\gamma(\hat{b})}\right)+\frac{2\left\|c\right\|_{L^\infty(X,\mu)}}{c_0}+1.
\end{equation}
Then the generator $(\mathcal{L}^{(\mathcal{Q})}, \mathcal{D}(\mathcal{L}^{(\mathcal{Q})}))$ satisfies the sector condition 
\begin{equation}\label{E:sectorcond}
\vert \left\langle (-\mathcal{L}^\mathcal{Q}-\varepsilon)f,g \right\rangle_{L^2(X, \mu)}\vert \leq K  \left\langle (-\mathcal{L}^\mathcal{Q}-\varepsilon)f,f \right\rangle_{L^2(X, \mu)}^{1/2} \left\langle (-\mathcal{L}^\mathcal{Q}-\varepsilon)g,g \right\rangle_{L^2(X, \mu)}^{1/2},
\end{equation}
$f,g\in\mathcal{D}(\mathcal{L}^\mathcal{Q})$, for all $0\leq \varepsilon\leq c_0/2$. 
\end{corollary}
%

\subsection{Linear elliptic and parabolic problems} 
Suppose throughout this subsection that $a$, $b$, $\hat{b}$ and $c$ satisfy the hypotheses of Proposition \ref{P:Qisnice} (i) and (ii). It is straightforward to formulate equations of elliptic type. Given $f\in L^2(X,\mu)$, we say that $u\in L^2(X,\mu)$ is a \emph{weak solution} to 
\begin{equation}\label{E:ellipticeq}
\mathcal{L}^{\mathcal{Q}} u=f
\end{equation}
if $u\in \mathcal{D}(\mathcal{E})$ and $\mathcal{Q}(u,g)=-\left\langle f,g\right\rangle_{L^2(X,\mu)}$ for all $g\in\mathcal{D}(\mathcal{E})$.

\begin{remark}
Formally, the generator $(\mathcal{L}^{\mathcal{Q}},\mathcal{D}(\mathcal{L}^{\mathcal{Q}}))$ of $(\mathcal{Q},\mathcal{D}(\mathcal{E}))$
has the structure
\[\mathcal{L}^{\mathcal{Q}} u=-\partial^\ast (a\,\partial u)+b\cdot \partial u+\partial^\ast(u\cdot \hat{b})+cu,\]
so that equation (\ref{E:ellipticeq}) is seen to be an abstract version of the elliptic equation 
\[\diverg(a \nabla u) + b\cdot \nabla u - \diverg(ub)+cu=f.\]
\end{remark}


It follows from the lower estimate in (\ref{E:posdef}) that the Green operator $G^\mathcal{Q}=(-\mathcal{L}^{\mathcal{Q}})^{-1}$ of $\mathcal{L}^{\mathcal{Q}}$ exists as a bounded linear operator $G^\mathcal{Q}:L^2(X,\mu)\to L^2(X,\mu)$ and satisfies
\begin{equation}\label{E:Green}
\mathcal{Q}(G^\mathcal{Q}f,g)=\left\langle f,g\right\rangle_{L^2(X,\mu)},\quad f\in L^2(X,\mu),\ g\in\mathcal{D}(\mathcal{E}).
\end{equation}

\begin{corollary}\label{C:solutionelliptic}
For any $f\in L^2(X,\mu)$ the function $u=-G^\mathcal{Q}f \in\mathcal{D}(\mathcal{L}^\mathcal{Q})$ is the unique weak solution to (\ref{E:ellipticeq}). It satisfies
\begin{equation}\label{E:energyboundelliptic}
\mathcal{Q}_1(u)\leq \left(\frac{2}{c_0}+\frac{4}{c_0^2}\right)\left\|f\right\|_{L^2(X,\mu)}^2.
\end{equation}
\end{corollary}
\begin{remark}
The constant in (\ref{E:energyboundelliptic}) is chosen just for convenience. The only fact that matters is that it may be chosen in a way that depends monotonically on $c_0$.
\end{remark}

\begin{proof}
The first part is clear, the second follows from (\ref{E:Green}), Cauchy-Schwarz and because for any $0<\varepsilon\leq c_0/2$ with $c_0$ as in (\ref{E:cisnice}) the operator $\mathcal{L}^\mathcal{Q}+\varepsilon$ generates a strongly continuous contraction semigroup, so that 
\[\left\|G^\mathcal{Q}f\right\|_{L^2(X,\mu)}=\left\|\left(\varepsilon + (-\varepsilon-\mathcal{L}^{\mathcal{Q}})\right)^{-1}f\right\|_{L^2(X,\mu)}\leq \frac{1}{\varepsilon}\left\|f\right\|_{L^2(X,\mu)}.\]
\end{proof}

\begin{remark}\label{R:more_general_c_possible}
If $c\in L^\infty(X, \mu)$ does not satisfy \eqref{E:cisnice}, one can at least solve equations 
\begin{equation}
\mathcal{L}^{\mathcal{Q}}u - c_1 u=f, 
\end{equation}
where $c_1>0$ is such that with $c_0$ defined as in (\ref{E:cisnice}) one has $c_0 + c_1>0$. The sectorial closed form 
\begin{equation}\label{E:modQ}
\mathcal{Q}_{c_1}(f,g) = \mathcal{Q}(f,g) + c_1\langle f,g\rangle_{L^2(X,\mu)}, \quad f,g \in \mathcal{D}(\mathcal{E}),
\end{equation}
satisfies \eqref{E:posdef},  \eqref{E:upperconst}, (\ref{E:sectorcond}) and (\ref{E:possibleK}) with $c_0 + c_1$ and $\left\|c\right\|_{L^\infty(X,\mu)} + c_1$ in place of $c_0$ and $\left\|c\right\|_{L^\infty(X,\mu)}$.
\end{remark}

Related parabolic problems can be discussed in a similar manner. Given $\mathring{u}\in L^2(X,\mu)$ we say that a function $u:(0,+\infty)\to L^2(X,\mu)$ is a \emph{solution} to the
Cauchy problem
\begin{equation}\label{E:paraboliceq}
\partial_t u(t)= \mathcal{L}^\mathcal{Q}u(t),\ t>0, \quad u(0)=\mathring{u},
\end{equation}
if $u$ is an element of $C^1((0,+\infty),L^2(X,\mu))\cap C([0,+\infty),L^2(X,\mu))$, we have $u(t)\in\mathcal{D}(\mathcal{L}^ {\mathcal{Q}})$ for any $t>0$ and (\ref{E:paraboliceq}) holds. See \cite[Chapter 4, Section 1]{P83}.

\begin{remark}
Problem (\ref{E:paraboliceq}) is an abstract version of the parabolic problem 
\[\partial_tu(t)=\diverg(a \nabla u(t)) + b\cdot \nabla u(t) - \diverg(u(t)\hat{b})+cu(t), \ t>0, \quad u(0)=\mathring{u}.\]
\end{remark}

Let $(T_t^\mathcal{Q})_{t>0}$ denote the strongly continuous contraction semigroup on $L^2(X,\mu)$ generated by $\mathcal{L}^\mathcal{Q}$.

\begin{corollary}\label{C:solutionparabolic}
For any $\mathring{u}\in L^2(X,\mu)$ the Cauchy problem (\ref{E:paraboliceq}) has the unique solution $u(t)=T_t^{\mathcal{Q}}\mathring{u}$, $t>0$. For any $t>0$ it satisfies $u(t)\in\mathcal{D}(\mathcal{L}^{\mathcal{Q}})$ and
\begin{equation}\label{E:energyboundparabolic}
\mathcal{Q}_1(u(t))\leq \left(\frac{C_K}{t}+1\right)\left\|\mathring{u}\right\|_{L^2(X,\mu)}^2,
\end{equation}
where $C_K>0$ is a constant depending only on the sector constant $K$ in (\ref{E:sectorcond}).
\end{corollary}

\begin{proof}
Again the first part of the Corollary is standard. To see (\ref{E:energyboundparabolic}) recall that the operator $(\mathcal{L}^\mathcal{Q}, \mathcal{D}(\mathcal{L}^\mathcal{Q}))$ satisfies the sector condition (\ref{E:sectorcond}). Consequently the semigroup $(T_t^\mathcal{Q})_{t>0}$ generated by $(\mathcal{L}^\mathcal{Q}+\varepsilon, \mathcal{D}(\mathcal{L}^\mathcal{Q}))$ extends to a holomorphic contraction semigroup on the sector $\{z\in\mathcal{C}: |\im z|\leq K^{-1}\mathrm{Re}\:z\}$, see for instance \cite[Chapter XI, Theorem 1.24]{Kato80}, or \cite[Theorem 2.20 and Corollary 2.21]{MaRoeckner92}. By (\ref{E:posdef}) zero is contained in the resolvent set of $\mathcal{L}^\mathcal{Q}$. This implies that for any $t>0$ we have 
\begin{equation}\label{E:prop_holomorphic_sg}
\Vert \mathcal{L}^\mathcal{Q} T_t^{\mathcal{Q}}f\Vert_{L^2(X, \mu)} \leq \frac{C_K}{t}\Vert f\Vert_{L^2(X,\mu)}, \quad f\in L^2(X,\mu),
\end{equation}
for some $C_K \in (0, \infty)$ depending only on the sector constant $K$, as an inspection of the classical proofs of (\ref{E:prop_holomorphic_sg}) shows, see for instance \cite[Theorem 4.6]{EngelNagel}, \cite[Section 2.5, Theorem 5.2]{P83} or the explanations in \cite[Section 2]{MugnoloNittkaPost13}. Now (\ref{E:energyboundparabolic}) follows using (\ref{E:prop_holomorphic_sg}), Cauchy-Schwarz and contractivity.
\end{proof}

\begin{remark}
Since weak solutions to (\ref{E:ellipticeq}) and solutions to (\ref{E:paraboliceq}) at fixed positive times are elements of $\mathcal{D}(\mathcal{E})\supset \mathcal{D}(\mathcal{L}^{\mathcal{Q}})$, they are H\"older continuous of order $1/2$ on $(X,R)$ by (\ref{E:resistanceest}). 
\end{remark}

It is a trivial observation that if $a\in C(X)$ satisfies
\begin{equation}\label{E:triviala}
\lambda< a(x) <\Lambda, \quad x\in X,
\end{equation}
then $a$, interpreted as a bounded linear map $v \mapsto a\cdot v$ from $\mathcal{H}$ into itself, satisfies (\ref{E:elliptic}). 
Our main interest is to understand the drift terms and therefore we restrict attention to coefficients $a$ of form (\ref{E:triviala}) in the following sections. Note that under condition (\ref{E:triviala}) the function $a$ may also be seen as a conformal factor, \cite{Azzam}.

\begin{remark}
A discussion of more general diffusion coefficients $a$ should involve suitable coordinates, see \cite{Hino08, HT-fgs5, T08}. In view of the fact that natural local energy forms on p.c.f. self-similar sets have pointwise index one, \cite{BK16, Ku93, Hino10}, assumption (\ref{E:triviala}) does not seem to be unreasonably restrictive for this class of fractal spaces.
\end{remark}

On finite graphs, \cite{Grigoryan, KLWbook}, and compact metric graphs, \cite{FKW07, KS99, KS00, Ku04, Ku05, Post12, Mugnolo14}, the forms in (\ref{E:Q}) admit rather familiar expressions. 

\begin{examples}\label{Ex:discrete2}
In the setup of Examples \ref{Ex:discrete1} and with a given volume function $\mu:V\to (0,+\infty)$ we obtain, accepting a slight abuse of notation,
\begin{multline}
\mathcal{Q}(f,g)=\frac12\sum_{p\in V}\sum_{q\in V}\omega(p,q) \overline{a}(p,q)(f(p)-f(q))(g(p)-g(q))-\frac12\sum_{p\in V}\sum_{q\in V}\omega(p,q) \overline{g}(p,q)b(p,q)(f(p)-f(q))\notag\\
- \frac12\sum_{p\in V}\sum_{q\in V}\omega(p,q) \overline{f}(p,q)\hat{b}(p,q)(g(p)-g(q))-\sum_{p\in V}c(p)f(p)g(p)\mu(p)
\end{multline}
for all $f,g\in \ell(V)$ and any given coefficients $a,c\in \ell(V)$ and $b,\hat{b}\in \ell^2(V\times V\setminus \diag, \omega)$.
\end{examples}

\begin{examples}\label{Ex:metric2}
Suppose we are in the same situation as in Examples \ref{Ex:discrete1} and $\mu$ is a finite Borel measure on $X_\Gamma$ that has full support and is equivalent to the Lebesgue-measure on each individual edge. Then, abusing notation slightly, 
\begin{multline}
\mathcal{Q}(f,g)=\sum_{e\in E} \int_0^{l_e} a_e(s)f_e'(s)g_e'(s)ds-\sum_{e\in E} \int_0^{l_e} g_e(s)b_e(s)f_e'(s)ds-\sum_{e\in E} \int_0^{l_e} f_e(s)\hat{b}_e(s)g_e'(s)ds\notag\\
-\sum_{e\in E} \int_0^{l_e} c_e(s)f_e(s)g_e(s) \mu(ds)
\end{multline}
for all $f,g \in \dot{W}^{1,2}(X_\Gamma)$ and all $a\in C(X_\Gamma)$, $c\in L^\infty(X_\Gamma,\mu)$ and $b,\hat{b}\in \bigoplus_{e\in E} L^2(0,l_e)$, where $u_e$ denoted the restriction to $e\in E$ in the a.e. sense of an integrable function on $X_\Gamma$.
\end{examples}

\section{Convergence of solutions on a single space}\label{S:single}

In this section we define bilinear forms $\mathcal{Q}_{(m)}$ on $L^2(X,\mu)$ by replacing $a$, $b$ and $\hat{b}$ in (\ref{E:Q}) by coefficients $a_m$ $b_m$ and $\hat{b}_m$ that may vary with $m$. To keep the exposition more transparent and since it is rather trivial to vary it, we keep $c$ fixed.
We consider the unique weak solutions to elliptic problems (\ref{E:ellipticeq}) and unique solutions at fixed positive times of parabolic problems (\ref{E:paraboliceq}) with these coefficients. For a sequence $(a_m)_m$ satisfying (\ref{E:elliptic}) uniformly in $m$, bounded sequences $(b_m)_m$ and $(\hat{b}_m)_m$ and small enough $c$, we can find accumulation points with respect to the uniform convergence on $X$ of these solutions, and these accumulation points are elements of $\mathcal{F}$, Corollary \ref{C:KigamiAA}. If  coefficients $a$, $b$, $\hat{b}$ and $c$ are given and the sequences $(a_m)_m$, $(b_m)_m$ and $(\hat{b}_m)_m$ converge to $a$, $b$ and $\hat{b}$, respectively, then we can conclude the uniform convergence of the solutions to the respective solutions of the target problem, Theorem \ref{T:approxsame}.

\subsection{Boundedness and convergence of vector fields} As in the preceding section we assume that $(X,R)$ is separable and locally compact, that $(\mathcal{E},\mathcal{F})$ is regular and that $\mu$ is a Borel measure on $(X,R)$ such that for any $x\in X$ and $R>0$ we have $0<\mu(B(x,R))<+\infty$.

Under a mild geometric assumption on $\mu$ any vector field $b\in\mathcal{H}$ satisfies the Hardy condition. We say that $\mu$ \emph{has a uniform lower bound $V$} if $V$ is an non-decreasing function $V:(0,+\infty)\to (0,+\infty)$ so that 
\begin{equation}\label{E:unilowerbound}
\mu(B(x,r))\geq V(r),\quad x\in X,\quad  r>0.
\end{equation}
The following proposition is a partial refinement of \cite[Lemma 4.2]{HR16}. 

\begin{proposition}\label{P:Luke}
Suppose that $\mu$ has the uniform lower bound $V$. Then for any $g\in \mathcal{F}\cap C_c(X)$, any $b\in\mathcal{H}$ and any $M>0$ we have
\begin{equation}\label{E:HR15}
\left\|g\cdot b\right\|_{\mathcal{H}}^2\leq \frac{1}{M}\mathcal{E}(g)+\mathcal{V}(M\left\|b\right\|_{\mathcal{H}}^2)\left\|b\right\|_{\mathcal{H}}^{2}\left\|g\right\|_{L^2(X,\mu)}^2,
\end{equation}
where $\mathcal{V}$ is the non-decreasing function
\[\mathcal{V}(s)=\frac{2}{V\left(\frac{1}{2 s}\right)}, \quad s>0.\]
In particular, any $b\in\mathcal{H}$ is in the Hardy class, and for any $M>0$ it satisfies the estimate (\ref{E:Hardy}) with 
$\delta(b)=\frac{1}{M}$ and  $\gamma(b)=\mathcal{V}(M\left\|b\right\|_{\mathcal{H}}^2)\left\|b\right\|_{\mathcal{H}}^{2}$.
Moreover, for any $\lambda>0$ condition (\ref{E:smallfields}) holds if we choose $M>2/\lambda$ for both $b$ and $\hat{b}$.
\end{proposition}
A proof of an inequality of type (\ref{E:HR15}) had already been given in \cite[Lemma 4.2]{HR16}, but the function $\mathcal{V}$ had not been specified and an unnecessary metric doubling assumption had been made. We comment on the necessary modifications.

\begin{proof}
We may assume $\left\|b\right\|_{\mathcal{H}}>0$. Let $\{B_j\}_j$ be a countable cover of $X$ by open balls $B_j$ of radius 
 $r=(2M\left\|b\right\|_{\mathcal{H}}^2)^{-1}$. As in \cite{HR16} we can use (\ref{E:resistanceest}) to see that for any $j$ and any $x\in B_j$ we have 
\[|g(x)|^2\leq 2|g(x)-(g)_{B_j}|^2+2(g)_{B_j}^2 \leq  2\mathcal{E}(g)r+2(g^2)_{B_j},\]
where we use the shortcut notation $(f)_B=\frac{1}{\mu(B)}\int_B f\:d\mu$. Setting $B_0=\emptyset$ and $C_j=B_j\setminus \bigcup_{i=0}^{j-1} B_i$, $j\geq 1$, we obtain a countable Borel partition $\{C_j\}_j$ of $X$ with $C_j\subset B_j$, $j\geq 1$.
Then for any $x\in X$ we have 
\[|g(x)|^2\leq 2\mathcal{E}(g)r+2\sum_j (g^2)_{B_j}\mathbf{1}_{C_j}(x)\leq 2\mathcal{E}(g)r+\frac{2}{V(r)}\left\|g\right\|_{L^2(X,\mu)}^2,\]
and using (\ref{E:boundedaction}) we arrive at the claimed inequality.
\end{proof}

We record two consequences of Proposition \ref{P:Luke}. The first states that if the norms of vector fields in a sequence are uniformly bounded then we may choose uniform constants in the Hardy condition (\ref{E:Hardy}).
\begin{corollary}\label{C:uniformHardy}
Suppose that $\mu$ has a uniform lower bound. If $(b_m)_m\subset\mathcal{H}$ satisfies $\sup_m \left\|b_m\right\|_{\mathcal{H}}<+\infty$ then for any $M>0$ there is a constant $\gamma_M>0$ independent of $m$ such that (\ref{E:Hardy})
holds for each $b_m$ with constants $\delta(b_m)=\frac{1}{M}$ and $\gamma(b_m)=\gamma_M$.
\end{corollary}

\begin{proof} Let $\mathcal{V}$ be defined as in Proposition \ref{P:Luke}. Since $\mathcal{V}$ is increasing we can take 
\begin{equation}\label{E:gammam}
\gamma_M:=\mathcal{V} (M \sup_m\left\|b_m\right\|_{\mathcal{H}}^2)\:\sup_m \left\|b_m\right\|_{\mathcal{H}}^2.
\end{equation}
\end{proof}

The second consequence is a continuity statement.
\begin{corollary}\label{C:bmtob}
Suppose that $\mu$ has a uniform lower bound. If $b\in\mathcal{H}$ and $(b_m)_m\subset \mathcal{H}$ is a sequence with $\lim_m b_m=b$ in $\mathcal{H}$ then for any $g\in C_c(X)$ we have 
\[\lim_m \left\|g\cdot b_m-g\cdot b\right\|_{\mathcal{H}}^2=0.\]
\end{corollary}
\begin{proof}
This is immediate from the definition of the function $\mathcal{V}$ in Proposition \ref{P:Luke} and the fact that the uniform lower bound $V$ of $\mu$ is strictly positive and increasing.
\end{proof}

\subsection{Accumulation points} For the rest of this section we assume that $(\E, \mathcal{F})$ is a resistance form on a nonempty set $X$ so that $(X, R)$ is compact, and that $\mu$ is a finite Borel measure on $(X,R)$ with a uniform lower bound $V$.
Note that by compactness $(\E, \mathcal{F})$ is regular.

For each $m$ let $a_m\in C(X)$ satisfy (\ref{E:triviala}) with the same constants $0<\lambda<\Lambda$. Suppose $M>0$ is large enough so that $\lambda_0:=\lambda/2-1/M>0$ and that $b_m,\hat{b}_m\in\mathcal{H}$ satisfy
\begin{equation}\label{E:bmssup_single}
\sup_m \left\|b_m\right\|_{\mathcal{H}}^2<+\infty\quad \text{and}\quad  \sup_m \big\|\hat{b}_m\big\|_{\mathcal{H}}^2<+\infty.
\end{equation}
Let $\gamma_M$ be as in (\ref{E:gammam}), let $\hat{\gamma}_M$ defined in the same way with the $\hat{b}_m$ replacing the $b_m$
and suppose that $c\in L^\infty(X,\mu)$ is such that  
\begin{equation}\label{E:conseqforfieldssup}
c_0:=\essinf_{x\in X} \left(-c(x)\right)-\frac{\gamma_M+\hat{\gamma}_M}{\lambda-2/M}>0.
\end{equation}
Then by Proposition \ref{P:Qisnice} and Corollary \ref{C:uniformHardy} the forms 
\begin{equation}\label{E:Qm}
\mathcal{Q}_{(m)}(f,g):= \left\langle a_m\cdot\partial f, \partial g\right\rangle_{\mathcal{H}}-\left\langle g\cdot b_m,\partial f\right\rangle_{\mathcal{H}}-\big\langle f\cdot \hat{b}_m, \partial g\big\rangle_{\mathcal{H}}-\left\langle cf,g\right\rangle_{L^2(X,\mu)}, \quad f,g\in \mathcal{F},
\end{equation}
are sectorial closed forms on $L^2(X,\mu)$. They satisfy (\ref{E:posdef}) with $\delta(b_m)=\delta(\hat{b}_m)=1/M$ 
and $\gamma_M$, $\hat{\gamma}_M$ replacing $\gamma(b)$, $\gamma(\hat{b})$ in (\ref{E:upperconst}). Their generators $(\mathcal{L}^{\mathcal{Q}_{(m)}},\mathcal{D}(\mathcal{L}^{\mathcal{Q}_{(m)}}))$ satisfy the sector conditions (\ref{E:sectorcond}) with the same sector constant $K$. As a consequence we observe uniform energy bounds for the solutions of (\ref{E:ellipticeq}) and (\ref{E:paraboliceq}). We write $\mathcal{Q}_{(m),\alpha}$ for the form defined as $\mathcal{E}_\alpha$ in (\ref{E:addalpha}) but with $\mathcal{Q}_{(m)}$ in place of $\mathcal{E}$.

\begin{proposition}\label{P:accusingle}
Let $a_m$, $b_m$, $\hat{b}_m$ and $c$ be as above such that (\ref{E:bmssup_single}) and (\ref{E:conseqforfieldssup}) hold. 
\begin{enumerate}
\item[(i)] If $f\in L^2(X,\mu)$ and $u_m$ is the unique weak solution to (\ref{E:ellipticeq}) with $\mathcal{L}^{\mathcal{Q}_{(m)}}$ in place of $\mathcal{L}$, then we have $\sup_m \mathcal{Q}_{(m),1}(u_m)<+\infty$. 
\item[(ii)] If $\mathring{u}\in L^2(X,\mu)$ and $u_m$ is the unique solution to (\ref{E:paraboliceq}) with $\mathcal{L}^{\mathcal{Q}_{(m)}}$ in place of $\mathcal{L}$, then for any $t>0$ we have $\sup_m \mathcal{Q}_{(m),1}(u_m(t))<+\infty$.
\end{enumerate}
\end{proposition}

\begin{proof}
Since (\ref{E:conseqforfieldssup}) and (\ref{E:sectorcond}) hold with the same constants $c_0$ and $K$ for all $m$, the statements follow from Corollaries \ref{C:solutionelliptic} and \ref{C:solutionparabolic}.
\end{proof}

The compactness of $X$ implies the existence of accumulation points in $C(X)$.

\begin{corollary}\label{C:KigamiAA}
Let $a_m$, $b_m$, $\hat{b}_m$ and $c$ be as above such that (\ref{E:bmssup_single}) and (\ref{E:conseqforfieldssup}) are satisfied. 
\begin{enumerate}
\item[(i)] If $f\in L^2(X,\mu)$ and $u_m$ is the unique weak solution to (\ref{E:ellipticeq}) with $\mathcal{L}^{\mathcal{Q}_{(m)}}$ in place of $\mathcal{L}^\mathcal{Q}$, then each subsequence of $(u_m)_m$ has a subsequence converging to a limit $\widetilde{u}\in \mathcal{F}$ uniformly on $X$. 
\item[(ii)] If $\mathring{u}\in L^2(X,\mu)$ and $u_m$ is the unique solution to (\ref{E:paraboliceq}) with $\mathcal{L}^{\mathcal{Q}_{(m)}}$ in place of $\mathcal{L}^\mathcal{Q}$, then for each $t>0$ each subsequence of $(u_m(t))_m$ has a further subsequence converging to a limit $\widetilde{u}_t\in \mathcal{F}$ uniformly on $X$. 
\end{enumerate}
\end{corollary}

At this point we can of course not claim that the $C(X)$-valued function $t\mapsto \widetilde{u}_t$ has any good properties or significance.

\begin{proof}
Since all $\mathcal{Q}_{(m)}$ satisfy (\ref{E:posdef}) with the same constants, Proposition \ref{P:accusingle} implies that $\sup_m\mathcal{E}_1(u_m)<+\infty$. By \cite[Lemma 9.7]{Ki12} the embedding $\mathcal{F}\subset C(X)$ is compact, hence
$(u_m)_m$ has a subsequence that converges uniformly on $X$ to a limit $\widetilde{u}$. To see that $\widetilde{u}\in\mathcal{F}$, note that also this subsequence is bounded in $\mathcal{F}$ and therefore has a further subsequence that converges to a limit $\widetilde{w}\in\mathcal{F}$ weakly in $L^2(X,\mu)$, as follows from a Banach-Saks type argument. This forces $\widetilde{w}=\widetilde{u}$. Statement (ii) is proved in the same manner.
\end{proof}

\subsection{Strong resolvent convergence}
Let $(\mathcal{E},\mathcal{F})$ and $\mu$ be as in the preceding subsection. Let $a\in \mathcal{F}$ be such that (\ref{E:triviala}) holds with constants $0<\lambda<\Lambda$ and let $(a_m)_m\subset C(X)$ be such that 
\begin{equation}\label{E:approxa}
\lim_m \left\|a_m - a \right\|_{\sup}=0.
\end{equation}
Without loss of generality we may then assume that also the functions $a_m$ satisfy (\ref{E:triviala}) with the very same constants $0<\lambda<\Lambda$. Suppose $M>0$ is large enough so that $\lambda_0:=\lambda/2-1/M>0$. Let $b,\hat{b}\in\mathcal{H}$ and let $(b_m)_m \subset \mathcal{H}$ and $(\hat{b}_m)_m\subset \mathcal{H}$ be sequences such that 
\begin{equation}\label{E:bmslim}
\lim_m \left\|b_m-b\right\|_{\mathcal{H}}=0 \quad \text{ and } \quad \lim_m \big\|\hat{b}_m -\hat{b}\big\|_{\mathcal{H}}=0.
\end{equation}
Note that this implies (\ref{E:bmssup_single}). Let $\gamma_M$ be as in (\ref{E:gammam}) and $\hat{\gamma}_M$ similarly but with the $\hat{b}_m$, and suppose that $c\in L^\infty(X,\mu)$ satisfies (\ref{E:conseqforfieldssup}). Let $\mathcal{Q}$ be as in (\ref{E:Q}) and $\mathcal{Q}_{(m)}$ as in (\ref{E:Qm}).

The next theorem states that the solutions to (\ref{E:ellipticeq}) and (\ref{E:paraboliceq}) depend continuously on the coefficients $a$, $b$ and $\hat{b}$. It is based on \cite[Theorem 3.1]{Hino98}.

\begin{theorem}\label{T:approxsame}
Let $a$, $a_m$, $b$, $b_m$, $\hat{b}$ and $\hat{b}_m$ be as above such that (\ref{E:approxa}) and (\ref{E:bmslim}) hold.  Then $\lim_m \mathcal{L}^{\mathcal{Q}_{(m)}} = \mathcal{L}^{\mathcal{Q}}$ in the strong resolvent sense, and the following hold.
\begin{enumerate}
\item[(i)] If $f\in L^2(X,\mu)$, $u$ and $u_m$ are the unique weak solutions to (\ref{E:ellipticeq}) and to (\ref{E:ellipticeq}) with $\mathcal{L}^{\mathcal{Q}_{(m)}}$ in place of $\mathcal{L}^\mathcal{Q}$, respectively, then $\lim_m u_m=u$ in $L^2(X,\mu)$. Moreover, there is a sequence $(m_k)_k$ with $m_k\uparrow +\infty$ such that $\lim_k u_{m_k} =u$ uniformly on $X$. 
\item[(ii)] If $\mathring{u}\in L^2(X,\mu)$, and $u$ and $u_m$ are the unique solutions to (\ref{E:paraboliceq}) and to (\ref{E:paraboliceq}) with $\mathcal{L}^{\mathcal{Q}_{(m)}}$ in place of $\mathcal{L}$, then for any $t>0$ we have $\lim_m u_m(t)=u(t)$ in $L^2(X,\mu)$. Moreover, for any $t>0$ there is a sequence $(m_k)_k$ with $m_k\uparrow +\infty$ such that $\lim_k u_{m_k}(t) =u(t)$ uniformly on $X$. 
\end{enumerate}
\end{theorem}

\begin{proof}
By \cite[Theorem 3.1]{Hino98}, the claimed strong resolvent convergence and the stated convergences in $L^2(X,\mu)$ follow once we have verified the conditions in Definition \ref{D:generalized convergence}, see Theorem \ref{T:convergence of forms} and Remark \ref{R:simplify} in Appendix \ref{S:generalized convergence}. The statements on uniform convergence then also follow using Corollary \ref{C:KigamiAA}. 

Without loss of generality we may assume that the function $c\in L^\infty(X, \mu)$ satisfies condition $\eqref{E:conseqforfieldssup}$. If not, proceed similarly as in Remark \ref{R:more_general_c_possible} and replace $c$ by $c-c_1$, where $c_1>0$ is large enough so that with $c_0$ as defined as in (\ref{E:conseqforfieldssup}) we have $c_1+c_0>0$, and consider the forms $(\mathcal{Q}_{(m),c_1}, \mathcal{F})$ with generators $(\mathcal{L}^{\mathcal{Q}_{(m)}} - c_1, \D(\mathcal{L}^{\mathcal{Q}_{(m)}}))$. If $\lim_{m\rightarrow\infty} \mathcal{L}^{\mathcal{Q}_{(m)}}-c_1 = \mathcal{L}^{\mathcal{Q}}-c_1$ in the KS-generalized strong resolvent sense then also $\lim_{m\rightarrow\infty} \mathcal{L}^{\mathcal{Q}_{(m)}}=\mathcal{L}^{\mathcal{Q}}$ in the KS-generalized strong resolvent sense. The statements on uniform convergence then follow using Corollary \ref{C:KigamiAA} and Remark \ref{R:more_general_c_possible}, note that for all $m$ and $u\in \mathcal{F}$ we have $\mathcal{Q}_{(m)}(u)\leq \mathcal{Q}_{(m),c_1}(u)$.

Thanks to (\ref{E:smallfields}), (\ref{E:cisnice}), (\ref{E:posdef}) and (\ref{E:upperconst}) together with Proposition \ref{P:Luke} and Corollaries \ref{C:uniformHardy} and \ref{C:bmtob} we can find  a constant $C>0$ such that for every sufficiently large $m$ we have
\begin{equation}\label{E:compareforms}
C\E_1(f) \leq \mathcal{Q}_{(m),1}(f)\leq C^{-1}\E_1(f), \quad f \in \mathcal{F}.
\end{equation}

Suppose that $\lim_m u_m =u$ weakly in $L^2(X,\mu)$ with $\varliminf_m Q_{(m),1}(u_m, u_m)<+\infty$. Then there is a subsequence $(u_{m_k})_k$ such that $\sup_k \mathcal{Q}_{(m_k),1}(u_{m_k})<+\infty$, and by (\ref{E:compareforms}) we have 
$\sup_k \E_1(u_{m_k}, u_{m_k})< +\infty$. A subsequence of $(u_{m_k})_k$ converges to a limit $u_\mathcal{E}\in\mathcal{F}$ weakly in $(\mathcal{F},\mathcal{E})$ and standard Banach-Saks type argument shows that $u_\mathcal{E}=u$,
what proves condition (i) in Definition \ref{D:generalized convergence}.
 
To verify condition (ii) in Definition \ref{D:generalized convergence} suppose that $(m_k)_k$ be a sequence of natural numbers with $m_k \uparrow \infty$, $\lim_k u_k = u$ weakly in $L^2(X, \mu)$ with $\sup_k Q_{(m_k),1} (u_k, u_k)<\infty$ and $u \in \mathcal{F}$. By (\ref{E:compareforms}) we have $\sup_k \mathcal{E}_1 (u_k)<\infty$, what implies that $(u_k)_k$ has a subsequence $(u_{k_j})_j$ converging to $u\in\mathcal{F}$ weakly in $\mathcal{F}$ and uniformly on $X$, and such that its averages $N^{-1}\sum_{j=1}^N u_{k_j}$ converge to $u$ in $\mathcal{F}$. Here the statement on uniform convergence is again a consequence of the compact embedding $\mathcal{F}\subset C(X)$, \cite[Lemma 9.7]{Ki12}. Combined with the weak convergence in $L^2(X,\mu)$ it follows that $(u_{k_j})_j$ converges weakly to $u$ in $(\mathcal{F}/\sim,\mathcal{E})$. Moreover, using (\ref{E:BD}), the convergence of averages and the linearity of $d_u$ we may assume that $(d_uu_{k_j})_j$ converges to $d_uu$ weakly in $L^2(X\times X\setminus \diag, J)$. As a consequence, we also have 
\begin{align}
\lim_j \mathcal{E}^c(u_{k_j},v)&=\lim_j \mathcal{E}(u_{k_j},v)-\lim_j\int_X\int_X d_uu_{k_j}(x,y)d_uv(x,y) J(dxdy)\notag\\
&= \mathcal{E}(u,v)-\int_X\int_X d_uu(x,y)d_uv(x,y) J(dxdy)\notag\\
&=\mathcal{E}^c(u,v)\notag
\end{align}
for all $v\in\mathcal{F}$. 

Now let $w \in \mathcal{F}$. Then we have
\begin{align}\label{E:splitQs}
\vert Q_{(m_{k_j})}(w, u_{k_j}) - Q(w,u) \vert \leq \vert  Q_{(m_{k_j})}(w, u_{k_j}) - Q(w,u_{k_j})\vert + \vert Q(w, u_{k_j} -u) \vert.
\end{align}  
Since $c$ is kept fixed, the first summand on the right hand side of the inequality (\ref{E:splitQs}) is bounded by 
\begin{multline}
\left\vert \left\langle (a_{m_{k_j}}-a)\cdot \partial w,\partial u_{k_j}\right\rangle_{\mathcal{H}}\right\vert + \left\vert \big\langle u_{k_j}\cdot (b_{m_{k_j}}- b), \partial w\big\rangle_{\mathcal{H}}\right\vert +
\left\vert \left\langle w\cdot (\hat{b}_{m_{k_j}}-\hat{b}),\partial u_{k_j}\right\rangle_{\mathcal{H}}\right\vert \notag\\
\leq \Vert a_{m_{k_j}}-a\Vert_{\sup} \mathcal{E}(w)^{1/2}\mathcal{E}(u_{k_j})^{1/2} + \Vert  u_{k_j} \Vert_{\sup}\Vert b_{m_{k_j}}-b \Vert_{\mathcal{H}} \E(w)^{1/2}+\Vert  w \Vert_{\sup}\Vert \hat{b}_{m_{k_j}}-\hat{b} \Vert_{\mathcal{H}} \E(u_{k_j})^{1/2},
\end{multline}
where we have used Cauchy-Schwarz and (\ref{E:boundedaction}). By the hypotheses on the coefficients and the boundedness of  $(u_{k_j})_j$ in energy and in uniform norm this converges to zero. The second summand on the right hand side of (\ref{E:splitQs}) is bounded by
\[\vert \left\langle \partial w, a\cdot\partial (u_{k_j} - u)\right\rangle_{\mathcal{H}} \vert + \vert \langle (u_{k_j} - u)\cdot b, \partial w\rangle_{\mathcal{H}} \vert + \vert \langle w\cdot \hat{b} , \partial( u_{k_j} - u) \rangle_{\mathcal{H}}\vert + \vert \langle cw, u_{k_j} - u \rangle_{L^2(X,\mu)}\vert. \]
The last summand in this line obviously converges to zero, and also the second does, note that $ \vert \langle (u_{k_j} - u)\cdot b, \partial w\rangle_{\mathcal{H}} \vert \leq  \Vert u_{k_j} - u \Vert_{\sup} \Vert b \Vert_{\mathcal{H}}\E(w)^{1/2} $ by Cauchy-Schwarz and (\ref{E:boundedaction}). By Proposition \ref{P:indep} we have  
\[\left\langle \partial w, a\cdot\partial (u_{k_j} - u)\right\rangle_{\mathcal{H}}=\int_X a\:d\nu^c_{w,u_{k_j} - u}+\int_X\int_X  \overline{a}(x,y)d_uw(x,y)d_u(u_{k_j} - u)(x,y)\:J(dxdy).\]
Since $\left\|\overline{a}d_uw\right\|_{L^2(X\times X\setminus \diag, J)}\leq \left\|a\right\|_{\sup}\mathcal{E}(w)^{1/2}$, the double integral converges to zero by the weak convergence of $(d_uu_{k_j})_j$ to $d_uu$ in $L^2(X\times X\setminus \diag, J)$. By (\ref{E:pointwisemult}) we have $\sup_j \mathcal{E}_1(au_{k_j})^{1/2}<+\infty$ and $\mathcal{E}_1(wu_{k_j})^{1/2}<+\infty$. Thinning out the sequence $(u_{k_j})_j$ once more we may, using the arguments above, assume that $\lim_j \mathcal{E}^c(a u_{k_j},v)=\mathcal{E}^c(au,v)$ and $\lim_j \mathcal{E}^c(w u_{k_j},v)=\mathcal{E}^c(wu,v)$ for all $v\in\mathcal{F}$. Then  also
\[\int_X a\:d\nu^c_{w,u_{k_j} - u}=\frac12\left\lbrace \mathcal{E}^c(aw, u_{k_j}-u) + \mathcal{E}^c(a(u_{k_j}-u),w) - \mathcal{E}^c(w(u_{k_j}-u),a) \right\rbrace\]
converges to zero. Together this implies that $\lim_j \left\langle \partial w, a\cdot\partial (u_{k_j} - u)\right\rangle_{\mathcal{H}}=0$. Finally, note that by the Leibniz rule for $\partial$, 
\[\big\langle \hat{b}, w\cdot \partial(u_{k_j}-u)\big\rangle_\mathcal{H}= \big\langle \hat{b},\partial(w(u_{k_j}-u))\big\rangle_\mathcal{H}-\big\langle\hat{b},(u_{k_j}-u)\cdot\partial w\big\rangle_{\mathcal{H}}.\]
As before we see easily that the second summand on the right hand side goes to zero. For the first, let $\hat{b}=\partial f+\eta$ be the unique decomposition of $\hat{b}\in\mathcal{H}$ into a gradient $\partial f$ of a function $f\in\mathcal{F}$ and a 'divergence free' vector field $\eta\in \ker \partial^\ast$. Then 
\[\big\langle \hat{b},\partial(w(u_{k_j}-u))\big\rangle_\mathcal{H}=\big\langle \partial f,\partial(w(u_{k_j}-u))\big\rangle_\mathcal{H}=\mathcal{E}(f,w(u_{k_j}-u)),\]
which converges to zero by the preceding arguments. Combining, we see that \[\lim_j \mathcal{Q}^{(n_{k_j})}(w,u_{k_j})=\mathcal{Q}(w,u),\] 
and since $w\in \mathcal{F}$ was arbitrary, this implies condition (ii) in Definition \ref{D:generalized convergence}.
\end{proof}

\begin{examples}
The basic requirements for Theorem \ref{T:approxsame} are that the resistance form $(\mathcal{E},\mathcal{F})$ is regular, the space $(X,R)$ is compact, and that the Borel measure $\mu$ on $(X,R)$ has a uniform lower bound. In particular, $\mu$ does not have to satisfy a volume doubling property. Possible examples can for instance be found amongst finite graphs, \cite{Grigoryan, KLWbook}, compact metric graphs, \cite{FKW07, KS99, KS00, Ku04, Ku05, Post12, Mugnolo14}, p.c.f. self-similar sets, \cite{BarlowPerkins, Ki89, Ki93, Ki93b, Ki01, Ku89, Ku93, Meyers}, classical Sierpinski carpets, \cite{BB89, BBKT10}, certain Julia sets, \cite{RogersTeplyaev10}, and certain random fractals, \cite{Hambly92, Hambly97}. 
\end{examples}

\section{Convergence of solutions on varying spaces}\label{S:varying}

In this section we basically repeat the approximation program from Section \ref{S:single}, but now on varying resistance spaces. More precisely, we study the convergence of suitable linearizations of solutions to (\ref{E:ellipticeq}) and (\ref{E:paraboliceq}) on approximating spaces $X^{(m)}$ to solutions to these equations on $X$. We establish these results for the case that $X$ is a finitely ramified set, \cite{IRT12, T08}, endowed with a local resistance form. Possible generalizations are commented on in Section \ref{SS:generalizations}.

\subsection{Setup and basic assumptions}\label{SS:setup}

We describe the setup we consider and the assumptions under which the results of this section are formulated. They are standing assumptions for all results in this section and will not be repeated in the particular statements.

We recall the notion of finitely ramified cell structures as introduced in \cite[Definition 2.1]{T08}. 

\begin{definition}\label{D:finitelyramified}
A \emph{finitely ramified set} $X$ is a compact metric space which has a \emph{cell structure} $\{X_\alpha\}_{\alpha\in\mathcal{A}}$ and a \emph{boundary (vertex) structure} $\{V_\alpha\}_{\alpha\in\mathcal{A}}$ such that the following hold:
\begin{enumerate}
\item[(i)] $\mathcal{A}$ is a countable index set;
\item[(ii)] each $X_\alpha$ is a distinct compact connected subset of $X$;
\item[(iii)] each $V_\alpha$ is a finite subset of $X_\alpha$;
\item[(iv)] if $X_\alpha=\bigcup_{j=1}^k X_{\alpha_j}$ then $V_\alpha\subset \bigcup_{j=1}^k X_{\alpha_j}$;
\item[(v)] there is a filtration $\{\mathcal{A}_n\}_n$ such that 
\begin{enumerate}
\item[(v.a)] each $\mathcal{A}_n$ is a finite subset of $\mathcal{A}$, $\mathcal{A}_0=\{0\}$, and $X_0=X$;
\item[(v.b)] $\mathcal{A}_n\cap \mathcal{A}_m=\emptyset$ if $n\neq m$;
\item[(v.c)] for any $\alpha\in\mathcal{A}_n$ there are $\alpha_1,...,\alpha_k\in\mathcal{A}_{n+1}$ such that $X_\alpha=\bigcup_{j=1}^k X_{\alpha_j}$;
\end{enumerate}
\item[(vi)] $X_{\alpha'}\cap X_\alpha = V_{\alpha'}\cap V_\alpha$ for any two distinct $\alpha, \alpha'\in\mathcal{A}_n$;
\item[(vii)] for any strictly decreasing infinite sequence $X_{\alpha_1} \supsetneq X_{\alpha_2} \supsetneq ...$ there exists $x\in X$ such that $\bigcap_{n\geq 1} X_{\alpha_n}=\{x\}$.
\end{enumerate}
Under these conditions the triple $(X, \{X_\alpha\}_{\alpha\in\mathcal{A}}, \{V_\alpha\}_{\alpha\in\mathcal{A}})$ is called a \emph{finitely ramified cell structure}.
\end{definition}

We write $V_n=\bigcup_{\alpha\in\mathcal{A}_n} V_\alpha$ and $V_\ast =\bigcup_{n\geq 0} V_n$, note that $V_n\subset V_{n+1}$ for all $n$. Suppose that $(\mathcal{E}, \widetilde{\mathcal{F}})$ is a resistance form on $V_\ast$. It can be written in the form (\ref{E:limitform}) with $\widetilde{\mathcal{F}}$ in place of $\mathcal{F}$, where the forms $\mathcal{E}_{V_m}$ are the restrictions of $\mathcal{E}$ to $V_m$ as in (\ref{E:EV}) and (\ref{E:approxbydiscreteforms}). Any function in $\widetilde{\mathcal{F}}$ is continuous in $(\Omega,R)$, where $\Omega$ denotes the $R$-completion of $V_\ast$, and therefore has a unique extension to a continuous function on $\Omega$. Writing $\mathcal{F}$ for the space of all such extensions, we obtain a resistance form $(\mathcal{E},\mathcal{F})$ on $\Omega$. To avoid topological difficulties in this paper, we make the following assumption.

\begin{assumption}\label{A:basic}\mbox{} 
We have $\Omega=X$ and the resistance metric $R$ is compatible with the original topology.
\end{assumption}

In view of \cite[Section 4]{HamblyNyberg03}, \cite[Section 7]{Meyers} and the well-established theory in \cite[Section 3.3]{Ki01} one could rephrase this by saying we consider a regular harmonic structure. As a consequence, $(X,R)$ is a compact and connected metric space and $(\mathcal{E},\mathcal{F})$ is a regular resistance form on $X$, local in the sense that if $f\in \mathcal{F}$ is constant on an open neighborhood of the support of $g\in \mathcal{F}$, then $\mathcal{E}(f,g)=0$, see for instance \cite[Theorem 3 and its proof]{T08}.

Given $m\geq 0$ and a function $v\in \ell(V_m)$ there exists a unique function $h_m(v)\in\mathcal{F}$ such that $h_m(v)|_{V_m}=v$ in $\ell(V_m)$ and 
\[\mathcal{E}(h_m(v))=\mathcal{E}_{V_m}(v)=\min\left\lbrace \mathcal{E}(u):\ u\in\mathcal{F}, u|_{V_m}=v\right\rbrace,\]
see \cite[Proposition 2.15]{Ki03}. This function $h_m(v)$ is called the \emph{harmonic extension} of $v$, and as usual we say that a function $u\in\mathcal{F}$ is \emph{$m$-harmonic} if $u=h_m(u|_{V_m})$. We write $H_m(X)$ to denote the space of $m$-harmonic functions on $X$ and write $H_mu:=h_m(u|_{V_m})$, $u\in\mathcal{F}$, for the projection from $\mathcal{F}$ onto $H_m(X)$. It is well known and can be seen as in \cite[Theorem 1.4.4]{Str06} that
\begin{equation}\label{E:approxbyPH}
\lim_m \mathcal{E}(u-H_mu)=0,\quad u\in\mathcal{F},
\end{equation}
 and using (\ref{E:resistanceest}) it follows that also $\lim_m\left\| u-H_m u\right\|_{\sup}=0$, where $\left\|\cdot\right\|_{\sup}$ denotes the supremum norm. Consequently the space  
 \[H_\ast(X)=\bigcup_{m\geq 0} H_m(X)\] 
is dense in $\mathcal{F}$ w.r.t. the seminorm $\mathcal{E}^{1/2}$ and in $C(X)$ w.r.t. the supremum norm. We write $H_m(X)/\sim$ for the space of $m$-harmonic functions on $X$ modulo constants. For each $m$ the space $H_m(X)/\sim$ is a finite dimensional, hence closed subspace of $(\mathcal{F}/\sim,\mathcal{E})$, and since $H_m\mathbf{1}=\mathbf{1}$, the operator 
$H_m$ is easily seen to induce an orthogonal projection in $(\mathcal{F}/\sim,\mathcal{E})$ onto $H_m(X)/\sim$, which we denote by the same symbol. Clearly $H_\ast(X)/\sim$ is dense in $(\mathcal{F}/\sim,\mathcal{E})$.

We now state the main assumptions under which we implement the approximation scheme. They are formulated in a way that simultaneously covers approximations schemes by discrete graphs and by metric graphs as discussed in Sections \ref{SS:discretegraphs} and \ref{SS:metricgraphs}, respectively.

Let $\diam_R(A)$ denote the diameter of a set $A$ in $(X,R)$. The following assumption requires $\mathcal{E}$ to be compatible with the cell structure in the following 'uniform' metric sense.
\begin{assumption}\label{A:gotozero}
We have $\lim_{n\to\infty}\max_{\alpha\in \mathcal{A}_n} \diam_R(X_\alpha)=0$.
\end{assumption}

We now assume that $(X^{(m)})_m$ is a sequence of subsets $X^{(m)}\subset X$ such that for each $m\geq 0$ we have $X^{(m)}\subset X^{(m+1)}$  and $X^{(m)}=\bigcup_{\alpha\in\mathcal{A}_m}X_\alpha^{(m)}$ where for any $\alpha\in\mathcal{A}_m$ the set $X_\alpha^{(m)}$ satisfies 
\[V_\alpha\subset X_\alpha^{(m)}\subset X_\alpha.\] 
For any $m\geq 0$ let now $(\mathcal{E}^{(m)},\mathcal{F}^{(m)})$ be a resistance form on $X^{(m)}$ so that $(X^{(m)}, R^{(m)})$ is topologically embedded in $(X,R)$. We also assume that the spaces $(X^{(m)}, R^{(m)})$ are compact, this implies that the resistance forms 
$(\mathcal{E}^{(m)},\mathcal{F}^{(m)})$ are regular. By $\nu_f^{(m)}$ we denote the energy measure of a function $f\in\mathcal{F}^{(m)}$, defined as in (\ref{E:energymeasure}) with $(\mathcal{E}^{(m)},\mathcal{F}^{(m)})$ in place of $(\mathcal{E},\mathcal{F})$. The energy measures $\nu_f^{(m)}$ may be interpreted as Borel measures on $X$.

\begin{remark}
For spaces, forms, operators, coefficients and measures indexed by $m$ and connected to $X$ and the form $(\mathcal{E},\mathcal{F})$ we will use a subscript index $m$, similar objects corresponding to the spaces $X^{(m)}$ and the forms $(\mathcal{E}^{(m)},\mathcal{F}^{(m)})$ will be indexed by a superscript $(m)$, unless stated otherwise. For functions we will generally use a subscript index.
\end{remark}

We make some further assumptions. The first expresses a connection between the resistance forms in terms of $m$-harmonic functions. 
\begin{assumption}\label{A:convergence}\mbox{}
\begin{enumerate}
\item[(i)] For each $m$ the pointwise restriction $u\mapsto u|_{X^{(m)}}$ defines a linear map from $H_m(X)$ into $\mathcal{F}^{(m)}$ which is injective and satisfies
\begin{equation}\label{E:basic}
\mathcal{E}^{(m)}(u|_{X^{(m)}})= \mathcal{E}(u),\quad  u\in H_m(X). 
\end{equation}
\item[(ii)] We have
\begin{equation}\label{E:weakconvenergy}
\nu_u=\lim_m \nu^{(m)}_{H_m(u)|_{X^{(m)}}}, \quad u\in\mathcal{F},
\end{equation}
in the sense of weak convergence of measures on $X$.
\end{enumerate}
\end{assumption}

As a trivial consequence of (\ref{E:weakconvenergy}) we have  
\begin{equation}\label{E:formaslimit}
\mathcal{E}(u)=\lim_{m\to\infty} \mathcal{E}^{(m)}(H_m(u)|_{X^{(m)}}), \quad u\in\mathcal{F}.
\end{equation}

\begin{remark}
For approximations by discrete graphs (\ref{E:weakconvenergy}) follows from (\ref{E:formaslimit}) and (\ref{E:energymeasureapprox}). For metric graph approximations (\ref{E:weakconvenergy}) is verified in Subsection \ref{SS:metricgraphs} below, the use of products in (\ref{E:energymeasure}) hinders a direct conclusion of (\ref{E:weakconvenergy}) from (\ref{E:formaslimit}). 
\end{remark}

Now let
\[H_m(X^{(m)}):=\left\lbrace u|_{X^{(m)}}: u \in H_m(X)\right\rbrace\] 
denote the image of $H_m(X)$ under the pointwise restriction $u\mapsto u|_{X^{(m)}}$, which by (\ref{E:basic}) induces an isometry from $(H_m(X)/\sim,\mathcal{E})$ onto the Hilbert space $(H_m(X^{(m)})/\sim,\mathcal{E}^{(m)})$. The space  $H_m(X^{(m)})$ is a closed linear subspace of $\mathcal{F}^{(m)}$ and the space $H_m(X^{(m)})/\sim$ is a closed linear subspace of $\mathcal{F}^{(m)}/\sim$. Let $H_m^{(m)}$ denote the projection from $\mathcal{F}^{(m)}$ onto $H_m(X^{(m)})$. It satisfies $H_m^{(m)}\mathbf{1}=\mathbf{1}$ and induces an orthogonal projection from $(\mathcal{F}^{(m)}/\sim,\mathcal{E}^{(m)})$ onto $H_m(X^{(m)})/\sim$ so that in particular, 
\begin{equation}\label{E:projectmetric}
\mathcal{E}^{(m)}(H_{m}^{(m)} v)\leq \mathcal{E}^{(m)}(v), \quad v\in \mathcal{F}^{(m)}.
\end{equation}
Let $id_{\mathcal{F}^{(m)}}$ denote the identity operator in $\mathcal{F}^{(m)}$. We need an assumption on the decay of the operators $id_{\mathcal{F}^{(m)}}-H_{m}^{(m)}$ as $m$ goes to infinity. By $\left\|\cdot\right\|_{\sup, X^{(m)}}$ we denote the supremum norm on $X^{(m)}$.

\begin{assumption}\label{A:decay}\mbox{}
\begin{enumerate}
\item[(i)] For any sequence $(u_m)_m$ with $u_m\in\mathcal{F}^{(m)}$ such that $\sup_m \mathcal{E}^{(m)}(u_m)<+\infty$ we have
\begin{equation}\label{E:bumpdecay}
\lim_m \big\|u_m-H_{m}^{(m)}u_m\big\|_{\sup,X^{(m)}}=0.
\end{equation}
\item[(ii)] For $u,w\in H_n(X)$ we have 
\begin{equation}\label{E:bumpdecayproducts}
\lim_m \mathcal{E}^{(m)}(u|_{X^{(m)}}w|_{X^{(m)}}-H_m^{(m)}(u|_{X^{(m)}}w|_{X^{(m)}}))=0.
\end{equation}
\end{enumerate}
\end{assumption}

\begin{remark}\label{R:trivial}
For discrete graph approximations as in Subsection \ref{SS:discretegraphs} we have $H_m^{(m)}v=v$, $v\in\mathcal{F}^{(m)}$, so that Assumption \ref{A:decay} is trivial.
\end{remark}

Now let $\mu$ and $\mu^{(m)}$ be a finite Borel measures on $X$ and $X^{(m)}$, respectively, which assign positive mass to each nonempty open subset of the respective space. Then by \cite[Theorem 9.4]{Ki12} and \cite[Theorem 3]{T08} the forms $(\mathcal{E},\mathcal{F})$ and $(\mathcal{E}^{(m)},\mathcal{F}^{(m)})$ are regular Dirichlet forms on $L^2(X,\mu)$ and $L^2(X^{(m)},\mu^{(m)})$, and the Dirichlet form $(\mathcal{E},\mathcal{F})$ is strongly local in the sense of \cite{FOT94}. 

We make an assumption on the connection between the spaces $L^2(X,\mu)$ and $L^2(X^{(m)},\mu^{(m)})$ and its consistency with the projections and pointwise restrictions. By $\ext_m:H_m(X^{(m)})\to  H_m(X)$ we denote the inverse of the bijection $u\mapsto u|_{X^{(m)}}$ from $H_m(X)$ onto $H_m(X^{(m)})$.
 
\begin{assumption}\label{A:connection}\mbox{}
\begin{enumerate}
\item[(i)] The measures $\mu$ and $\mu^{(m)}$ admit a uniform lower bound in the following sense: There is a non-increasing function $V:\mathbb{N}\to (0,+\infty)$ such that for any $m$ we have $\mu(X_\alpha)\geq V(m)$, $\alpha\in\mathcal{A}_m$, and moreover, $\mu^{(m)}(X_\alpha^{(m)})\geq V(m)$, $\alpha\in\mathcal{A}_m$.
\item[(ii)] There are linear operators $\Phi_m:L^2(X,\mu) \to L^2(X^{(m)},\mu^{(m)})$ such that 
\begin{equation}\label{E:uninormbound}
\sup_m \left\|\Phi_m\right\|_{L^2(X,\mu) \to L^2(X^{(m)},\mu^{(m)})}<+\infty,
\end{equation}
\begin{equation}\label{E:KSconvHilbert}
\lim_{m\to\infty}\big\|\Phi_m u\big\|_{L^2(X^{(m)},\mu^{(m)})}=\left\|u\right\|_{L^2(X,\mu)},\quad u\in L^2(X,\mu),
\end{equation}
and for any $n$ and $u\in H_n(X)$ we have 
\begin{equation}\label{E:adjoint}
\lim_m \left\|\Phi_m^\ast\Phi_m u-u\right\|_{L^2(X,\mu)}=0,
\end{equation}
where for any $m$ the symbol $\Phi_m^\ast$ denotes the adjoint of $\Phi_m$.
\item[(iii)] For any sequence $(u_m)_m\subset \mathcal{F}$ with $\sup_m \mathcal{E}(u_m)<+\infty$ we have 
\begin{equation}\label{E:connection}
\lim_{m\to\infty} \big\|\Phi_m u_m-u_m|_{X^{(m)}}\big\|_{L^2(X^{(m)},\mu^{(m)})}=0.
\end{equation}
\item[(iv)] For any sequence $(u_m)_m$ with $u_m\in \mathcal{F}^{(m)}$ such that $\sup_m \mathcal{E}^{(m)}_1(u_m)<+\infty$
we have 
\begin{equation}\label{E:L2consist}
\sup_m \left\|\ext_m H_{m}^{(m)} u_m \right\|_{L^2(X,\mu)}<+\infty.
\end{equation}
\end{enumerate}
\end{assumption}

Let $\mathcal{H}$ and $\mathcal{H}^{(m)}$ denote the spaces of generalized $L^2$-vector fields associated with $(\mathcal{E},\mathcal{F})$ and $(\mathcal{E}^{(m)},\mathcal{F}^{(m)})$, respectively. The corresponding gradient operators we denote by $\partial$ and $\partial^{(m)}$. If $a$, $b$, $\hat{b}$ and $c$ satisfy the hypotheses of Proposition \ref{P:Qisnice} (i) then 
\[\mathcal{Q}(f,g):= \left\langle a\,\partial f,\partial g\right\rangle_\mathcal{H}-\left\langle g\cdot b,\partial f\right\rangle_{\mathcal{H}}-\big\langle f\cdot \hat{b}, \partial g\big\rangle_{\mathcal{H}}-\left\langle cf,g\right\rangle_{L^2(X,\mu)}, \quad f,g\in\mathcal{F},\]
defines a sectorial closed form $(\mathcal{Q},\mathcal{F})$ on $L^2(X,\mu)$. If $a$ and $c$ are suitable continuous functions on $X$ and $b$, $\hat{b}$, $b^{(m)}$ and $\hat{b}^{(m)}$ are vector fields of a suitable form, then we can define sectorial closed forms  $(\mathcal{Q}^{(m)},\mathcal{F}^{(m)})$ on the spaces $L^2(X^{(m)},\mu^{(m)})$, respectively, by 
\begin{multline}\label{E:KSQm}
\mathcal{Q}^{(m)}(f,g):=
\left\langle a|_{X^{(m)}}\cdot \partial f, \partial g\right\rangle_{\mathcal{H}^{(m)}}-\left\langle g\cdot b^{(m)},\partial f\right\rangle_{\mathcal{H}^{(m)}}\\-\big\langle f\cdot \hat{b}^{(m)}, \partial g\big\rangle_{\mathcal{H}^{(m)}}-\left\langle c|_{X^{(m)}}f,g\right\rangle_{L^2(X^{(m)},\mu^{(m)})}, \quad f,g\in \mathcal{F}^{(m)}.
\end{multline}

In Subsection \ref{SS:KSaccu} below we observe that under simple boundedness assumptions the solutions of (\ref{E:ellipticeq}) and (\ref{E:paraboliceq}) (for fixed $t>0$) associated with the forms $\mathcal{Q}^{(m)}$ on the spaces $X^{(m)}$ accumulate in a suitable sense, see Proposition \ref{P:KSaccu}. In Theorem \ref{T:KS_spectral} in Subsection \ref{SS:KS_spectral} we then conclude that they actually converge to the solutions to the respective equation associated with the form $\mathcal{Q}$, as announced in the introduction. In the preparatory Subsections \ref{SS:consequences} and \ref{SS:compatible} we record some consequences of the assumptions and discuss possible choices for $b$, $\hat{b}$, $b^{(m)}$ and $\hat{b}^{(m)}$.

\subsection{Some consequences of the assumptions}\label{SS:consequences}

We begin with some well-known conclusions.
\begin{lemma}\mbox{}\label{L:metricisgood}
\begin{enumerate}
\item[(i)] For any $p,q\in V_m$ we have $R^{(m)}(p,q)=R(p,q)$. In particular, 
$\diam_R(V_\alpha)=\diam_{R^{(m)}}(V_\alpha)$ for any $m\geq n$ and $\alpha\in\mathcal{A}_n$.
\item[(ii)] We have $\diam_R(X_\alpha)=\diam_R(V_\alpha)$ for any $n$ and $\alpha\in\mathcal{A}_n$, and 
$\diam_{R^{(m)}}(X_\alpha^{(n)})=\diam_{R^{(m)}}(V_\alpha)$ if $m\geq n$.
\end{enumerate}
\end{lemma}
\begin{proof}
To see (i) note that for any $p,q\in V_m$ we have, by a standard conclusion and using (\ref{E:basic}) and (\ref{E:projectmetric}),
\begin{align}
R(p,q)^{-1}&=\min\left\lbrace \mathcal{E}(u): u\in H_m(X): u(p)=0, u(q)=1\right\rbrace\notag\\
&=\min\left\lbrace \mathcal{E}^{(m)}(u|_{X^{(m)}}): u\in H_m(X), u(p)=0, u(q)=1\right\rbrace\notag\\
&=R^{(m)}(p,q)^{-1}.\notag
\end{align}
If the first statement (ii) were not true we could find  $p\in X_\alpha\cap V_\ast$ and $q\in (X_\alpha\cap V_\ast)\setminus V_\alpha$ such that $R(p,q)>R(p,q')$ for all $q'\in V_\alpha$. This would imply that there exists some $u\in H_n(X)$ with $u(p)=0$ and $\mathcal{E}(u)=1$ such that $u(q)^2<u(q')^2$ for all $q'\in V_\alpha$. However, this contradicts the maximum principle for harmonic functions on the cell $X_\alpha$. The second statement follows similarly.
\end{proof}

Also the following is due to Assumption \ref{A:convergence}.
\begin{corollary}\label{C:jumpsvanish}
For any $f_1, f_2\in H_n(X)$ and $g_1, g_2\in C(X)$ we have 
\[\lim_m\left\langle g_1|_{X^{(m)}}\cdot \partial^{(m)}(f_1|_{X^{(m)}}), g_2|_{X^{(m)}}\cdot \partial^{(m)}(f_2|_{X^{(m)}})\right\rangle_{\mathcal{H}^{(m)}}=\left\langle g_1\cdot \partial f_1, g_2\cdot\partial f_2\right\rangle_{\mathcal{H}}.\]
\end{corollary} 
\begin{proof}
If all $\mathcal{E}^{(m)}$'s are local then by Proposition \ref{P:indep} we have 
\[\left\|g|_{X^{(m)}}\cdot \partial^{(m)}(f|_{X^{(m)}})\right\|_{\mathcal{H}^{(m)}}^2=\int_{X^{(m)}}(g|_{X^{(m)}})^2d\nu^{(m),c}_{f|_{X^{(m)}}}=\int_X g^2d\nu_{f|_X^{(m)}}^{(m)}\]
for all $f\in H_n(X)$ and $g\in C(X)$, where $\nu_f^{(m),c}$ denotes the local part of the energy measure of $f$ with respect to $(\mathcal{E}^{(m)},\mathcal{F}^{(m)})$, and by (\ref{E:weakconvenergy}) this converges to 
\[\int_X g^2d\nu_f=\left\|g\cdot \partial f\right\|_{\mathcal{H}}^2.\]
Suppose now that the $\mathcal{E}^{(m)}$'s have nontrivial jump measures $J^{(m)}$. If $f\in H_n(X)$ and $g\in H_{n'}(X)$ have disjoint supports then by Proposition \ref{P:indep}, the locality of $\mathcal{E}^{(m),c}$, (\ref{E:basic}) and the locality of $\mathcal{E}$ we have 
\begin{align}\label{E:jumpsvanish}
-2\lim_m \int_X\int_X f(x)g(y)\:J^{(m)}(dxdy)&=\lim_m\int_{X^{(m)}}\int_{X^{(m)}}(f(x)-f(y))(g(x)-g(y))J^{(m)}(dxdy)\notag\\
&=\lim_m \mathcal{E}^{(m)}(f,g)\notag\\
&=\mathcal{E}(f,g)\notag\\
&=0.
\end{align}
Given $f,g\in C(X)$ with disjoint supports, we can, by the proof of \cite[Theorem 3]{T08}, find sequences of functions $(f_j)_j$ and $(g_j)_j$ from $H_\ast(X)$ approximating $f$ and $g$ uniformly and disjoint compact sets $K(f)\subset X$ and $K(g)\subset X$ such that all $f_j$ and $g_j$ are supported in $K(f)$ and $K(g)$, respectively. Therefore (\ref{E:jumpsvanish}) and the arguments used in the proof of \cite[Theorem 3.2.1]{FOT94} imply that $\lim_m J^{(m)}=0$ vaguely on $X\times X\setminus \diag$. For functions  $f\in H_n(X)$ and $g\in C(X)$ we therefore have 
\[\lim_m \int_{X^{(m)}}\int_{X^{(m)}} (d_ug(x,y))^2(d_uf(x,y))^2J^{(m)}(dxdy)=0, \]
as can be seen using the arguments in the proof of \cite[Lemma 3.1]{HinzMeinert19+}. On the other hand we have 
\[\left\|g\cdot \partial f\right\|_{\mathcal{H}}^2=\lim_m\left\lbrace \int_{X^{(m)}}g^2d\nu_f^{(m),c}+\frac12\int_{X^{(m)}}\int_{X^{(m)}}(g^2(x)+g^2(y))(d_uf(x,y))^2 J^{(m)}(dxdy)\right\rbrace\]
for such $f$ and $g$ by (\ref{E:weakconvenergy}) and (\ref{E:BDmeas}).
Combining and taking into account Proposition \ref{P:indep} we can conclude that 
\begin{align}
\left\|g\cdot \partial f\right\|_{\mathcal{H}}^2&=\lim_m\left\lbrace \int_{X^{(m)}}g^2d\nu_f^{(m),c}+\frac12\int_{X^{(m)}}\int_{X^{(m)}}(\overline{g}(x,y))^2(d_uf(x,y))^2 J^{(m)}(dxdy)\right\rbrace\notag\\
& = \lim_m \left\|g|_{X^{(m)}}\cdot \partial^{(m)}(f|_{X^{(m)}})\right\|_{\mathcal{H}^{(m)}}^2,\notag
\end{align}
from which the stated result follows by polarization.
\end{proof}

Another consequence, in particular of Assumption \ref{A:connection}, is the convergence of the $L^2$-spaces and the energy domains in the sense of Definition \ref{D:KS}.

\begin{corollary}\label{C:energyimpliesL2}     \mbox{}
\begin{enumerate}
\item[(i)] We have 
\begin{equation}\label{E:KSL2}
\lim_{m\to\infty} L^2(X^{(m)},\mu^{(m)})=L^2(X,\mu)
\end{equation}
in the KS-sense with identification operators $\Phi_m$ as above.
\item[(ii)] We have 
\begin{equation}\label{E:KSenergy}
\lim_{m\to\infty} \mathcal{F}^{(m)}=\mathcal{F}
\end{equation}
in the KS-sense with identification operators $u\mapsto (H_mu)|_{X^{(m)}}$ mapping from $\mathcal{F}$ into $\mathcal{F}^{(m)}$ respectively.
\item[(iii)] If $f\in\mathcal{F}$ and $(f_m)_m$ is a sequence of functions $f_m\in\mathcal{F}^{(m)}$ such that $\lim_m f_m=f$ KS-strongly w.r.t. (\ref{E:KSenergy}) then we also have  $\lim_m f_m=f$ KS-strongly w.r.t. (\ref{E:KSL2}).
\end{enumerate}
\end{corollary}

\begin{proof}
Statement (i) is immediate from (\ref{E:KSconvHilbert}). To see statement (ii) let $u\in\mathcal{F}$. If $x_0\in V_0$ is fixed, we have $H_mu(x_0)=u(x_0)$ for any $m$ and therefore, by (\ref{E:resistanceest}) and (\ref{E:approxbyPH}),
\[\lim_m \left\|u-H_mu\right\|_{L^2(X,\mu)}^2\leq \mu(X)\lim_m\left\|u-H_mu\right\|_{\sup}^2\leq \mu(X)\diam_R(X)\lim_m\mathcal{E}(u-H_mu)=0.\]
Using (\ref{E:uninormbound}), we obtain $\lim_m\left\|\Phi_m H_mu-\Phi_m u\right\|_{L^2(X,\mu)}=0$, and combining with (\ref{E:connection}) and (\ref{E:KSconvHilbert}), 
\begin{align}
\lim_m &\left\|(H_mu)|_{X^{(m)}}\right\|_{L^2(X^{(m)},\mu^{(m)})}\notag\\
&=\lim_m \left\|(H_mu)|_{X^{(m)}}-\Phi_m H_mu \right\|_{L^2(X^{(m)},\mu^{(m)})}+\lim_m \left\|\Phi_m H_mu \right\|_{L^2(X^{(m)},\mu^{(m)})}\notag\\
&=\lim_m \left\|\Phi_m u \right\|_{L^2(X^{(m)},\mu^{(m)})}\notag\\
&=\left\|u\right\|_{L^2(X,\mu)}.\notag
\end{align}
Together with (\ref{E:formaslimit}) this shows that $\lim_m \mathcal{E}^{(m)}_1((H_m u)|_{X^{(m)}})=\mathcal{E}_1(u)$ for all $u\in\mathcal{F}$. To see (iii) note that according to the hypothesis, there exist $\varphi_n\in \mathcal{F}$ such that 
\[\lim_n \mathcal{E}_1(\varphi_n-f)=0\quad \text{ and }\quad \lim_n\varlimsup_m \mathcal{E}_1^{(m)}((H_m \varphi_n)|_{X^{(m)}}-f_m)=0.\] 
This implies $\lim_n \left\|\varphi_n -f\right\|_{L^2(X,\mu)}=0$ and $\lim_n\varlimsup_m\left\|(H_m \varphi_n)|_{X^{(m)}}-f_m\right\|_{L^2(X^{(m)},\mu^{(m)})}=0$.
Conditions (\ref{E:KSconvHilbert}) and (\ref{E:connection}), applied to the constant function $\mathbf{1}$, yield $\lim_m \mu^{(m)}(X^{(m)})=\mu(X)$, and in particular,
\begin{equation}\label{E:unitotalmass}
\sup_m \mu(X^{(m)})<+\infty.
\end{equation}
We may therefore use (\ref{E:connection}) to conclude $\lim_m\left\|(H_m \varphi_n)|_{X^{(m)}}-\Phi_m H_m \varphi_n \right\|_{L^2(X^{(m)},\mu^{(m)})}=0$ for any $n$, so that 
\begin{equation}\label{E:onepiece}
\lim_n\varlimsup_m\left\|\Phi_m H_m \varphi_n -f_m\right\|_{L^2(X^{(m)},\mu^{(m)})}=0.
\end{equation}
Let $x_0\in V_0$. Then, since $H_m\varphi_n(x_0)=\varphi(x_0)$ for all $m$ and $n$, the resistance estimate (\ref{E:resistanceest}) implies 
\[\lim_m\left\| H_m\varphi_n-\varphi_n\right\|_{L^2(X,\mu)}=0\] 
for all $n$. Together with (\ref{E:uninormbound}) it follows that
\begin{align}
\lim_n\varlimsup_m & \left\|\Phi_m H_m\varphi_n - \Phi_m \varphi_n \right\|_{L^2(X^{(m)},\mu^{(m)})}\notag\\
&\leq \left(\sup_m \left\|\Phi_m\right\|_{L^2(X,\mu) \to L^2(X^{(m)},\mu^{(m)})}\right) \lim_n\varlimsup_m \left\| H_m\varphi_n-\varphi_n\right\|_{L^2(X,\mu)}\notag\\
&=0,\notag
\end{align}
and combining with (\ref{E:onepiece}) we obtain $\lim_n\varlimsup_m\left\|\Phi_m \varphi_n-f_m\right\|_{L^2(X^{(m)},\mu^{(m)})}=0$.
\end{proof}

In the sequel we will say 'KS-weakly' resp. 'KS-strongly' if we refer to the convergence (\ref{E:KSL2}) and say 'KS-weakly w.r.t. (\ref{E:KSenergy})' resp.'KS-strongly w.r.t. (\ref{E:KSenergy})' if we refer to the convergence (\ref{E:KSenergy}). We finally record a property of KS-weak convergence that will be useful later on.

\begin{lemma}\label{L:productsKSweakly}
If $\lim_m f_m=f$ KS-weakly and $w\in\mathcal{F}$ then there is a sequence $(m_k)_k$ with $m_k\uparrow +\infty$ such that $\lim_k w|_{X^{(m_k)}}f_{m_k}=wf$ KS-weakly.
\end{lemma}
\begin{proof} 
For any $w\in\mathcal{F}$ we have $\lim_m w|_{X^{(m)}}=w$  KS-strongly by (\ref{E:connection}). Fix $w\in\mathcal{F}$. Clearly
\[\sup_m \left\|w|_{X^{(m)}}f_m\right\|_{L^2(X^{(m)},\mu^{(m)})}<+\infty\] 
by the boundedness of $w$, hence $\lim_k w|_{X^{(m_k)}}f_{m_k}=\widetilde{w}$ KS-weakly for some $\widetilde{w}\in L^2(X,\mu)$ and some sequence $(m_k)_k$. For any $v\in\mathcal{F}$ we have $vw\in\mathcal{F}$ and trivially $(vw)|_{X^{(m)}}=v|_{X^{(m)}} w|_{X^{(m)}}$, hence 
\begin{multline}
\left\langle \widetilde{w}, v\right\rangle_{L^2(X,\mu)}=\lim_k \left\langle w|_{X^{(m_k)}}f_{m_k}, v|_{X^{(m_k)}}\right\rangle_{L^2(X^{(m_k)},\mu^{(m_k)})}=\lim_k \left\langle f_{m_k}, (vw)|_{X^{(m_k)}}\right\rangle_{L^2(X^{(m_k)},\mu^{(m_k)})}\notag\\
=\left\langle f, wv\right\rangle_{L^2(X,\mu)}=\left\langle wf, v\right\rangle_{L^2(X,\mu)},
\end{multline}
what by the density of $\mathcal{F}$ in $L^2(X,\mu)$ implies $\widetilde{w}=wf$ and therefore the lemma.
\end{proof}

\subsection{Boundedness and convergence of vector fields}\label{SS:compatible}

We provide a version of Proposition \ref{P:Luke} for finitely ramified sets. Recall the notation from Assumption \ref{A:connection}.
\begin{proposition}\label{P:LukeKS}\mbox{}
\begin{enumerate}
\item[(i)] Given $b\in\mathcal{H}$ and $M>0$ let $n_0$ be such that 
\[\max_{\alpha\in\mathcal{A}_{n_0}}\diam_R(X_\alpha)\leq \frac{1}{2M\left\|b\right\|_{\mathcal{H}}^2}.\]
Then for all $g\in\mathcal{F}$ we have 
\[\left\|g\cdot b\right\|_{\mathcal{H}}^2\leq \frac{1}{M}\mathcal{E}(g)+\frac{2}{V(n_0)}\left\|b\right\|_{\mathcal{H}}^2\left\|g\right\|_{L^2(X,\mu)}^2.\]
\item[(ii)] Suppose $b^{(m)}\in\mathcal{H}^{(m)}$, $M>0$ and $n_0\leq m$ are such that 
\[\max_{\alpha\in \mathcal{A}_{n_0}}\diam_{R^{(m)}}(X_\alpha^{(n_0)})\leq \frac{1}{2M\big\|b^{(m)}\big\|_{\mathcal{H}^{(m)}}^2}.\]
Then for all $g\in\mathcal{F}^{(m)}$ we have
\[\big\|g\cdot b^{(m)}\big\|_{\mathcal{H}^{(m)}}^2\leq \frac{1}{M}\mathcal{E}^{(m)}(g)+\frac{2}{V(n_0)} \big\|b^{(m)}\big\|_{\mathcal{H}^{(m)}}^2\big\|g\big\|_{L^2(X^{(m)},\mu^{(m)})}^2.\]
\end{enumerate}
\end{proposition}

\begin{proof}
We use the shortcut $(g)_{X_\alpha}=\frac{1}{\mu(X_\alpha)}\int_{X_\alpha} g\: d\mu$. For any $\alpha\in \mathcal{A}_{n_0}$ and $x\in X_\alpha$ have, by (\ref{E:resistanceest}), $|g(x)-(g)_{X_\alpha}|\leq \mathcal{E}(g)^{1/2}\diam_R(X_\alpha)^{1/2}$ and therefore
\[|g(x)|^2\leq 2\mathcal{E}(g)\diam_R(X_\alpha)+2(g^2)_{X_\alpha}\leq \frac{1}{M\left\|b\right\|_{\mathcal{H}}^2}\mathcal{E}(g)+\frac{2}{V(n_0)}\left\|g\right\|_{L^2(X,\mu)}^2.\]
Creating a finite partition of $X$ from the cells $X_\alpha$, $\alpha\in \mathcal{A}_{n_0}$, we see that the preceding estimate holds for all $x\in X$, and using (\ref{E:boundedaction}) we obtain (i). Statement (ii) is similar.
\end{proof}

Similarly as in Corollary \ref{C:uniformHardy}, uniform norm bounds on the vector fields allow to choose uniform constants in the Hardy condition (\ref{E:Hardy}).

\begin{corollary}\label{C:KSuniformHardy}
Suppose $b^{(m)}\in \mathcal{H}^{(m)}$ are such that $\sup_m \left\|b^{(m)}\right\|_{\mathcal{H}^{(m)}}<+\infty$. Then for any $M>0$ there exist $\gamma_M>0$ and $n_0$ such that for each $m\geq n_0$ the coefficient $b^{(m)}$ satisfies (\ref{E:Hardy}) with 
$\delta(b^{(m)})=\frac{1}{M}$  and $\gamma(b^{(m)})=\gamma_M$.
\end{corollary}

\begin{proof}
By Lemma \ref{L:metricisgood} and Assumption \ref{A:gotozero} we can find $n_0$ such that
\[\sup_{m\geq n_0}\max_{\alpha\in\mathcal{A}_{n_0}}\diam_{R^{(m)}}(X_\alpha^{(m)})\leq \frac{1}{4M\sup_m\big\|b^{(m)}\big\|_{\mathcal{H}^{(m)}}^2}.\]
Setting 
\begin{equation}\label{E:KSgammam}
\gamma_M:=\frac{2}{V(n_0)}\sup_{m\geq n_0} \big\|b^{(m)}\big\|_{\mathcal{H}^{(m)}}^2 
\end{equation}
we obtain the desired result by Proposition \ref{P:LukeKS}.
\end{proof}

To formulate an analog of Corollary \ref{C:bmtob} for varying spaces we need a certain compatibility of the vector fields involved. One rather easy way to ensure the latter is to focus on suitable elements of the module $\Omega^1_a(X)$ and their equivalence classes in $\mathcal{H}$ and $\mathcal{H}^{(m)}$ which then define vector fields $b$ on $X$ and $b^{(m)}$ on $X^{(m)}$ suitable to allow an approximation procedure. Given an element of $\Omega^1_a(X)$ of the special form $\sum_i g_i\cdot d_uf_i$ with $g_i\in C(X)$ and $f_i\in H_n(X)$, let $b$ defined as its $\mathcal{H}$-equivalence class $\big[\sum_i g_i\cdot d_uf_i\big]_\mathcal{H}$ as in Section \ref{S:resistanceforms}, that is,
\begin{equation}\label{E:equivclass}
b:=\sum_i g_i\cdot \partial f_i.
\end{equation}
By Assumption \ref{A:convergence} we have $f_i|_{X^{(m)}}\in\mathcal{F}^{(m)}$ for all $i$ and $m$, so that $\sum_i g_i|_{X^{(m)}}\cdot d_u(f_i|_{X^{(m)}})$ is an element of $\Omega^1_a(X^{(m)})$. We define $b^{(m)}$ to be its $\mathcal{H}^{(m)}$-equivalence class $\big[\sum_i g_i|_{X^{(m)}}\cdot d_u(f_i|_{X^{(m)}})\big]_{\mathcal{H}^{(m)}}$, that is,
\begin{equation}\label{E:equivclassm}
b^{(m)}:=\sum_i g_i|_{X^{(m)}}\cdot \partial^{(m)}(f_i|_{X^{(m)}}).
\end{equation}

The following convergence result may be seen as a partial generalization of (\ref{E:weakconvenergy}). It is immediate from  Corollary \ref{C:jumpsvanish} and bilinear extension. 

\begin{corollary}\label{C:compatibleb}
Suppose $b$ and $b^{(m)}$ are as in (\ref{E:equivclass}) and (\ref{E:equivclassm}) and $g\in C(X)$. Then we have 
\begin{equation}\label{E:lim}
\lim_m \left\|g|_{X^{(m)}}\cdot b^{(m)}\right\|_{\mathcal{H}^{(m)}}= \left\|g\cdot b\right\|_{\mathcal{H}}.
\end{equation}
\end{corollary}

\begin{remark}\label{R:restriction}
One might argue that an analog of Corollary \ref{C:bmtob} in terms of a simple restriction of vector fields $b\in\mathcal{H}$ to $X^{(m)}$ would be more convincing than Corollary \ref{C:compatibleb}. However, as $\mathcal{H}$ and $\mathcal{H}^{(m)}$ are obtained by different factorizations, it is not obvious how to correctly define a restriction operation on all of $\mathcal{H}$.
Using the finitely ramified cell structure one can introduce restrictions $b|_{X^{(m)}}$ to $X^{(m)}$ of certain types of vector fields $b\in \mathcal{H}$ and obtain an counterpart of (\ref{E:lim}) with these restrictions $b|_{X^{(m)}}$ in place of the $b^{(m)}$'s. This auxiliary observation is discussed in Section \ref{S:restriction}, it is not needed for our main results.
\end{remark}

\subsection{Accumulation points}\label{SS:KSaccu}

Let $a\in C(X)$ satisfy (\ref{E:triviala}) with $0<\lambda<\Lambda$, suppose $M>0$ is large enough so that $\lambda_0:=\lambda/2-1/M>0$ and that $b^{(m)},\hat{b}^{(m)}\in\mathcal{H}^{(m)}$ satisfy
\begin{equation}\label{E:bmssup}
\sup_m \left\|b^{(m)}\right\|_{\mathcal{H}^{(m)}}^2<+\infty\quad \text{and}\quad  \sup_m \vert|\hat{b}^{(m)}\vert|_{\mathcal{H}^{(m)}}^2<+\infty.
\end{equation}
Let $\gamma_M$ be as in (\ref{E:KSgammam}) and $\hat{\gamma}_M$ similarly but with the $\hat{b}^{(m)}$ in place of $b^{(m)}$ and suppose that $c\in C(X)$ satisfies (\ref{E:conseqforfieldssup}). Then for each $m$ the form $(\mathcal{Q}^{(m)},\mathcal{F}^{(m)})$ as in (\ref{E:KSQm})
is a closed form on $L^2(X^{(m)},\mu^{(m)})$, and (\ref{E:posdef}) holds with $\delta(b^{(m)})=\delta(\hat{b}^{(m)})=1/M$ 
and with $\gamma_M$, $\hat{\gamma}_M$ in place of $\gamma(b)$, $\gamma(\hat{b})$ in (\ref{E:upperconst}). There is a constant $K>0$ such that for each $m$ the generator $(\mathcal{L}^{\mathcal{Q}^{(m)}},\mathcal{D}(\mathcal{L}^{\mathcal{Q}^{(m)}}))$ of $(\mathcal{Q}^{(m)},\mathcal{F}^{(m)})$ obeys the sector condition (\ref{E:sectorcond}) with sector constant $K$. As a consequence, we can observe the following uniform energy bounds on solutions to elliptic and parabolic equations similar to Proposition \ref{P:accusingle}.

\begin{proposition}\label{P:KSaccu}
Let $a$, $b^{(m)}$, $\hat{b}^{(m)}$ and $c$ be as above such that (\ref{E:bmssup}) and (\ref{E:conseqforfieldssup}) hold. 
\begin{enumerate}
\item[(i)] If $f\in L^2(X,\mu)$, and $u_m$ is the unique weak solution to (\ref{E:ellipticeq}) with $\mathcal{L}^{\mathcal{Q}^{(m)}}$ in place of $\mathcal{L}$ and $f_m = \Phi_m f$ in place of $f$ then we have $\sup_m \mathcal{Q}_1^{(m)}(u_m)<+\infty$. 
\item[(ii)] If $\mathring{u}\in L^2(X,\mu)$, and $u_m$ is the unique solution to (\ref{E:paraboliceq}) in $L^2(X^{(m)}, \mu^{(m)})$ with $\mathcal{L}^{\mathcal{Q}^{(m)}}$ in place of $\mathcal{L}$ and with initial condition $\mathring{u}_m=\Phi_m \mathring{u}$ then for any $t>0$ we have $\sup_m \mathcal{Q}_1^{(m)}(u_m(t))<+\infty$.
\end{enumerate}
\end{proposition}

\begin{proof}
Since (\ref{E:conseqforfieldssup}) and (\ref{E:sectorcond}) hold with the same constants $c_0$ and $K$ for all $m$, Corollaries \ref{C:solutionelliptic} and \ref{C:solutionparabolic} together with \eqref{E:uninormbound} yield that  
$\sup_m \mathcal{Q}_1^{(m)}(u_m)\leq \left(\frac{2}{c_0}+\frac{4}{c_0^2}\right)\left\|f\right\|_{L^2(X,\mu)}$ and $\sup_m\mathcal{Q}_1^{(m)}(u_m(t))\leq \left(\frac{C_K}{t}+1\right)\left\|\mathring{u}\right\|_{L^2(X,\mu)}^2$
and the results follow.  
\end{proof}

\begin{remark}
Proposition \ref{P:KSaccu} needs only Assumption \ref{A:connection} (i) and (ii). Assumption \ref{A:convergence}, Assumption \ref{A:decay} and Assumption \ref{A:connection} (iii) and (iv) are not needed.
\end{remark}

\begin{remark}
The hypotheses of Proposition \ref{P:KSaccu} imply that $((\mathcal{Q}^{(m)},\mathcal{F}^{(m)}))_m$ is an equi-elliptic family in the sense of \cite[Definition 2.1]{MugnoloNittkaPost13}.
\end{remark}

By the compactness of $X$ we can find accumulation points in $C(X)$ for extensions to $X$ of linearizations of solutions. The next corollary may be seen as an analog of Corollary \ref{C:KigamiAA}. Recall the definitions of the projections $H_m^{(m)}$ and the extension operators $\ext_m$.

\begin{corollary}\label{C:KSaccu}
Let $a$, $b^{(m)}$, $\hat{b}^{(m)}$ and $c$ be as above such that (\ref{E:bmssup}) and (\ref{E:conseqforfieldssup}) hold. 
\begin{enumerate}
\item[(i)] If $f\in L^2(X,\mu)$, and $u_m$ is the unique weak solution to (\ref{E:ellipticeq}) with $\mathcal{L}^{\mathcal{Q}^{(m)}}$ in place of $\mathcal{L}$ and $f_m = \Phi_m f$ in place of $f$ then each subsequence $(u_{m_k})_k$ of $(u_m)_m$ has a further subsequence $(u_{m_{k_j}})_j$ such that $(\ext_{m_{k_j}}H_{m_{k_j}}^{(m_{k_j})}u_{m_{k_j}})_j$ converges to a limit $\widetilde{u}\in C(X)$ uniformly on $X$.
\item[(ii)] If $\mathring{u}\in L^2(X,\mu)$, and $u_m$ is the unique solution to (\ref{E:paraboliceq}) in $L^2(X^{(m)}, \mu^{(m)})$ with $\mathcal{L}^{\mathcal{Q}^{(m)}}$ in place of $\mathcal{L}$ and with initial condition $\mathring{u}_m=\Phi_m \mathring{u} $ then for any $t>0$ each subsequence $(u_{m_k}(t))_k$ of $(u_m(t))_m$ has a further subsequence $(u_{m_{k_j}}(t))_j$ such that $(\ext_{m_{k_j}}H_{m_{k_j}}^{(m_{k_j})}u_{m_{k_j}}(t))_j$ converges to a limit $\widetilde{u}_t\in C(X)$ uniformly on $X$.
\end{enumerate}
\end{corollary}

\subsection{Generalized strong resolvent convergence}\label{SS:KS_spectral}

The next result is an analog of Theorem \ref{T:approxsame} on varying spaces, it uses notions of convergence along a sequence of varying Hilbert spaces, \cite{KuwaeShioya03, Toelle06}, see Appendix \ref{S:generalized convergence}. The key ingredient is
Theorem \ref{T:convergence of forms} - a special case of \cite[Theorem 7.15, Corollary 7.16 and Remark 7.17]{Toelle10}, which constitute a natural generalization of \cite[Theorem 3.1]{Hino98} to the framework of varying Hilbert spaces, \cite{KuwaeShioya03}. 

\begin{theorem}\label{T:KS_spectral}
Suppose that 
\begin{equation}\label{E:restrictive}
b=\sum_i g_i\cdot \partial f_i\quad \text{and} \quad \hat{b}=\sum_i \hat{g}_i\cdot \partial \hat{f}_i
\end{equation}
are finite linear combinations with $f_i, \hat{f}_i, g_i, \hat{g}_i\in H_n(X)$ as in (\ref{E:equivclass}) and for any $m$ let
\begin{equation}\label{E:restrictivem}
b^{(m)}:=\sum_i g_i|_{X^{(m)}}\cdot \partial^{(m)}(f_i|_{X^{(m)}})\quad \text{and}\quad \hat{b}^{(m)}:=\sum_i \hat{g}_i|_{X^{(m)}}\cdot \partial^{(m)}(\hat{f}_i|_{X^{(m)}})
\end{equation}
as in (\ref{E:equivclassm}). Let $a\in H_n(X)$ be such that (\ref{E:elliptic}) holds and let $c\in C(X)$. Then $\lim_m \mathcal{L}^{\mathcal{Q}^{(m)}}=\mathcal{L}^{\mathcal{Q}}$ in the KS-generalized resolvent sense, and the following hold.
\begin{enumerate}
\item[(i)] If $f\in L^2(X,\mu)$, $u$ is the unique weak solution to (\ref{E:ellipticeq}) on $X$ and $u_m$ is the unique weak solution to (\ref{E:ellipticeq}) on $X^{(m)}$ with $\mathcal{L}^{\mathcal{Q}^{(m)}}$ and $\Phi_m f$ in place of $\mathcal{L}^{\mathcal{Q}}$ and $f$, then
we have $\lim_m u_m=u$ KS-strongly. Moreover, there is a sequence $(m_k)_k$ with $m_k\uparrow +\infty$ such that $\lim_k \ext_{m_k}H_{m_k}^{(m_k)}u_{m_k} =u$ uniformly on $X$. 
\item[(ii)] If $\mathring{u}\in L^2(X,\mu)$, $u$ is the unique solution to (\ref{E:paraboliceq}) on $X$ and $u_m$ is the unique weak solution to (\ref{E:paraboliceq}) on $X^{(m)}$ with $\mathcal{L}^{\mathcal{Q}^{(m)}}$ and $\Phi_m \mathring{u}$ in place of $\mathcal{L}^{\mathcal{Q}}$ and $\mathring{u}$, then for any $t>0$ we have 
we have $\lim_m u_m=u$ KS-strongly. Moreover, for any $t>0$ there is a sequence $(m_k)_k$ with $m_k\uparrow +\infty$ such that $\lim_k \ext_{m_k}H_{m_k}^{(m_k)} u_{m_k}(t) =u(t)$ uniformly on $X$. 
\end{enumerate}
\end{theorem}

A version of Theorem \ref{T:KS_spectral} for more general coefficients is stated below in Theorem \ref{T:general}. The proof of Theorem \ref{T:KS_spectral} makes use of the following key fact. 

\begin{lemma}\label{L:Arzela-Ascoli & uniqueness of limits}
Suppose $(n_k)_k$ is a sequence with $n_k\uparrow +\infty$ and $(u_k)_k$ is a sequence with $u_k
 \in L^2(X^{(n_k)}, \mu^{(n_k)})$ converging to $u\in L^2(X,\mu)$ KS-weakly and satisfying $\sup_k \E^{(n_k)}_1(u_k)<\infty$. Then we have $u \in \mathcal{F}$, and there is a sequence $(k_j)_j$ with $k_j\uparrow +\infty$ such that 
\begin{enumerate}
\item[(i)] $\lim_j u_{n_{k_j}}=u$ KS-weakly w.r.t. (\ref{E:KSenergy}), and moreover, for any $f\in\mathcal{F}$ and any sequence $(f_j)_j$ such that $f_j\in\mathcal{F}^{(n_{k_j})}$ and $\lim_j f_j=f$ KS-strongly w.r.t. (\ref{E:KSenergy}) along $(n_{k_j})_j$ we have 
\begin{equation}\label{E:energiesonly}
\lim_j \mathcal{E}^{(n_{k_j})}(f_j,u_{n_{k_j}})=\mathcal{E}(f,u). 
\end{equation}
\item[(ii)] $\lim_j \ext_{n_{k_j}}H_{n_{k_j}}^{(n_{k_j})}u_{n_{k_j}}=u$ uniformly on $X$.
\end{enumerate}
\end{lemma}

\begin{proof}
Let $v_k : = \ext_{n_k} H_{n_k}^{(n_k)} u_{k}$. By hypothesis and (\ref{E:basic}) we have  
\begin{equation}\label{E:unienergybound}
\sup_k \E(v_k) = \sup_k \E^{(n_k)}(H_{n_k}^{(n_k)}(u_k)) \leq \sup_k \E^{(n_k)}(u_{n_k}) < + \infty. 
\end{equation}
Since $v_k|_{X^{(n_k)}} = H_{n_k}^{(n_k)} u_{k}$, (\ref{E:unienergybound}), (\ref{E:unitotalmass}) and (\ref{E:bumpdecay}) allow to conclude that 
\begin{equation}\label{E:L2}
\lim_k \left\|v_{k}|_{X^{(n_{k})}}-u_{k}\right\|_{L^2(X^{(n_{k})},\mu^{(n_{k})})}=0,
\end{equation}
what implies that $\lim_k v_k|_{X^{(n_{k})}}=u$ KS-weakly.

We now claim that for any $n$ and any $w\in H_n(X)$ we have 
\begin{equation}\label{E:claim1}
\lim_k \left\langle w, v_k\right\rangle_{L^2(X,\mu)}=\left\langle w,u\right\rangle_{L^2(X,\mu)}. 
\end{equation} 
We clearly have $\lim_k \Phi_{n_k} w=w$ KS-strongly. Therefore 
\[\left\langle w,u\right\rangle_{L^2(X,\mu)}=\lim_k\left\langle \Phi_{n_k} w, v_k|_{X^{(n_{k})}}\right\rangle_{L^2(X^{n_k},\mu^{(n_k)})},\]
and using  (\ref{E:connection}) and (\ref{E:unienergybound}) this limit is seen to equal 
\[\lim_k \left\langle \Phi_{n_k}w, \Phi_{n_k}v_k\right\rangle_{L^2(X^{n_k},\mu^{(n_k)})}=\lim_k \left\langle\Phi_{n_k}^\ast \Phi_{n_k}w, v_k\right\rangle_{L^2(X,\mu)}.\]
Applying (\ref{E:adjoint}) we arrive at (\ref{E:claim1}). By (\ref{E:unienergybound}), and since (\ref{E:L2consist}) implies
$\sup_k \left\|v_k\right\|_{L^2(X,\mu)}<+\infty$, we can find a sequence $(k_j)_j$ with $\lim_j k_j=+\infty$ such that 
$(u_{k_j})_j$ converges KS-weakly w.r.t. (\ref{E:KSenergy}) to a limit $u_\mathcal{E}\in\mathcal{F}$ and $(v_{k_j})_j$ converges weakly in $L^2(X,\mu)$ to a limit $\overline{u}_\mathcal{E}\in\mathcal{F}$. Since $\bigcup_{n\geq 0} H_n(X)$ is dense in $L^2(X,\mu)$ we have $\overline{u}_{\mathcal{E}}=u$ by (\ref{E:claim1}), what shows that $u\in\mathcal{F}$. We now verify that 
\begin{equation}\label{E:claim2}
\overline{u}_{\mathcal{E}}=u_\mathcal{E}.
\end{equation}
For any $w\in H_n(X)$ the equalities
\begin{align}
\mathcal{E}_1(w,\overline{u}_{\mathcal{E}})&=\lim_j \left\lbrace \mathcal{E}(w,v_{k_j})+\left\langle w, v_{k_j}\right\rangle_{L^2(X,\mu)}\right\rbrace\notag\\
&=\lim_j \left\lbrace \mathcal{E}(w,v_{k_j})  -\big\langle \Phi_{n_{k_j}}^\ast \Phi_{n_{k_j}} w, v_{k_j}\big\rangle_{L^2(X,\mu)} \right\rbrace\notag\\
&=\lim_j \left\lbrace \mathcal{E}^{(n_{k_j})}(w|_{X^{(n_{k_j})}},v_{k_j}|_{X^{(n_{k_j})}})  -\big\langle \Phi_{n_{k_j}} w, \Phi_{n_{k_j}} v_{k_j}\big\rangle_{L^2(X^{(n_{k_j})},\mu^{(n_{k_j})})} \right\rbrace\notag
\end{align}
hold, the second and third equality due to (\ref{E:adjoint}) and (\ref{E:basic}), respectively. Using (\ref{E:connection}) twice on the second summands in the last line, the above limit is seen to equal
\[\lim_j \left\lbrace \mathcal{E}^{(n_{k_j})}(w|_{X^{(n_{k_j})}},v_{k_j}|_{X^{(n_{k_j})}})  -\big\langle w|_{X^{(n_{k_j})}},  v_{k_j}|_{X^{(n_{k_j})}}\big\rangle_{L^2(X^{(n_{k_j})},\mu^{(n_{k_j})})} \right\rbrace.\]
For $j$ so large that $n_{k_j}\geq n$ the function $w|_{X^{(n_{k_j})}}$ is an element of $H_{n_{k_j}}(X^{(n_{k_j})})$, so that by orthogonality in $\mathcal{F}^{(n_{k_j})}$ we can replace $v_{k_j}|_{X^{(n_{k_j})}}=H_{n_{k_j}}^{(n_{k_j})}u_{k_j}$ in the first summand by $u_{k_j}$. In the second term we can make the same replacement by (\ref{E:bumpdecay}) and (\ref{E:unitotalmass}), so that the above can be rewritten 
\begin{align}
\lim_j \left\lbrace \mathcal{E}^{(n_{k_j})}(w|_{X^{(n_{k_j})}},u_{k_j})  -\big\langle w|_{X^{(n_{k_j})}},  u_{k_j}\big\rangle_{L^2(X^{(n_{k_j})},\mu^{(n_{k_j})})} \right\rbrace & =\lim_j \mathcal{E}_1^{(n_{k_j})}(w|_{X^{(n_{k_j})}},u_{k_j})\notag\\
&=\mathcal{E}_1(w,u_{\mathcal{E}}),\notag
\end{align}
because $\lim_j w|_{X^{(n_{k_j})}}=w$ KS-strongly w.r.t. (\ref{E:KSenergy}). Since $\bigcup_{n\geq 0} H_n(X)$ is dense in $\mathcal{F}$, this implies (\ref{E:claim2}) and therefore the first statement of (i), so far for the sequence $(u_{k_j})_j$. The statement on the limit (\ref{E:energiesonly}) in (i) follows by Corollary \ref{C:energyimpliesL2}.

To save notation in the proof of (ii) we now write $(u_k)_k$ for the sequence $(u_{k_j})_j$ extracted in (i). Let $x_0\in V_0$. Then (\ref{E:resistanceest}) implies that $(v_k-v_k(x_0))_k$ is an equicontinuous and equibounded sequence of functions on $X$, so that by Arzel\`a-Ascoli we can find a subsequence $(v_{k_j}-v_{k_j}(x_0))_j$ which converges uniformly on $X$ to a function $w_{x_0}\in C(X)$. Since $\mu$ is finite, this implies $\lim_j  v_{k_j}-v_{k_j}(x_0)=w_{x_0}$ in $L^2(X,\mu)$.
By (\ref{E:connection}) and (\ref{E:unienergybound}) we also have 
\[\lim_j \left\|v_{k_j}|_{X^{(n_{k_j})}}-v_{k_j}(x_0)-\Phi_{n_{k_j}}(v_{k_j}-v_{k_j}(x_0))\right\|_{L^2(X^{(n_{k_j})},\mu^{(n_{k_j})})}=0,\] 
so that combining, we see that $\lim_j (v_{k_j}|_{X^{(n_{k_j})}}-u_{k_j}|_{X^{(n_{k_j})}}(x_0))=w_{x_0}$ KS-strongly and therefore also KS-weakly. Since $\lim_k v_k|_{X^{(n_{k})}}=u$ KS-weakly by (\ref{E:L2}), we may conclude that $\lim_k v_k|_{X^{(n_{k})}}(x_0)=u-w_{x_0}$ KS-weakly. In particular, by \cite[Lemma 2.3]{KuwaeShioya03}, 
\[\sup_j|v_{k_j}|_{X^{(n_{k_j})}}(x_0)| \mu(X^{(n_{k_j})})^{1/2}=\sup_j \left\|v_{k_j}|_{X^{(n_{k_j})}}(x_0)\right\|_{L^2(X^{(n_{k_j})},\mu^{(n_{k_j})})}<+\infty.\] 
Since $\lim_m \mu^{(m)}(X^{(m)})=\mu(X)>0$ it follows that $v_{k_j}|_{X^{(n_{k_j})}}(x_0)$ is a bounded sequence of real numbers and therefore has a subsequence converging to some limit $z\in\mathbb{R}$. Keeping the same notation for this subsequence, we can use (\ref{E:connection}) and (\ref{E:unitotalmass}) to conclude that $\lim_j \big\|v_{k_j}|_{X^{(n_{k_j})}}(x_0)-\Phi_{n_{k_j}}z\big\|_{L^2(X^{(n_{k_j})},\mu^{(n_{k_j})})}=0$, hence $\lim_j v_{k_j}|_{X^{(n_{k_j})}}(x_0)=z$ KS-weakly and therefore necessarily $z=u-w_{x_0}$. This implies that $\lim_j v_{k_j}=\lim_j  (v_{k_j}-v_{k_j}(x_0)) + \lim_j v_{k_j}(x_0)=u$ uniformly on $X$ as stated in (ii). Clearly the statements in (i) remain true for this subsequence.
\end{proof}

We prove Theorem \ref{T:KS_spectral}. 
\begin{proof}
Since the operators $\mathcal{L}^{\mathcal{Q}^{(m)}}$ obey the sector condition (\ref{E:sectorcond}) with the same sector constant, Theorem \ref{T:convergence of forms} will imply the desired convergence, provided that the forms $\mathcal{Q}^{(m)}$ and $\mathcal{Q}$ satisfy the conditions in Definition \ref{D:generalized convergence}. Corollary \ref{C:KSaccu} then takes care of the claimed uniform convergences.

Without loss of generality we may (and do) assume that the function $c\in C(X)$ satisfies condition $\eqref{E:conseqforfieldssup}$. Otherwise we use the same shifting argument as in the proof of Theorem \ref{T:approxsame}, the statements on uniform convergence then follow using Corollary \ref{C:KSaccu}.

By (\ref{E:smallfields}), (\ref{E:cisnice}), (\ref{E:posdef}) and (\ref{E:upperconst}) together with Proposition \ref{P:Luke} and Corollaries \ref{C:KSuniformHardy} and \ref{C:compatibleb} we can find a constant $C>0$ such that for any sufficiently large $m$ we have
\begin{equation}\label{E:KScompareforms}
C\E_1^{(m)}(f) \leq \mathcal{Q}^{(m)}_{1} (f)\leq C^{-1}\E_1^{(m)}(f), \quad f \in \mathcal{F}^{(m)}.
\end{equation}

To check condition (i) in Definition \ref{D:generalized convergence}, suppose that $(u_m)_m$ is a sequence with $u_m\in L^2(X^{(m)},\mu^{(m)})$ converging KS-weakly to a function $u\in L^2(X,\mu)$ and such that  $\varliminf_m \mathcal{Q}^{(m)}_1(u_m)<+\infty$. It has a subsequence $(u_{m_k})_k$ which by (\ref{E:KScompareforms}) satisfies $\sup_k \mathcal{E}^{(m_k)}(u_{m_k})<+\infty$, and by Lemma \ref{L:Arzela-Ascoli & uniqueness of limits} we then know that $u\in\mathcal{F}$, what implies the condition.

To verify condition (ii), suppose that $u\in\mathcal{F}$, $(m_k)_k$ is a sequence with $m_k\uparrow +\infty$ and that $u_k\in L^2(X^{(m_k)},\mu^{(m_k)})$ are such that $\lim_k u_k=u$ KS-weakly and $\sup_k \mathcal{Q}_1^{(m_k)}(u_k)<+\infty$. By (\ref{E:KScompareforms}) we have $\sup_k \mathcal{E}^{(m_k)}_1(u_k)<+\infty$. Now let $w\in H_n(X)$. Clearly $\lim_m w|_{X^{(m)}}=w$ KS-strongly. By Lemma \ref{L:productsKSweakly} we may assume that along $(m_k)_k$ we also have 
$\lim_k a|_{X^{(m_k)}}u_k=au$ and $\lim_k (w\hat{g}_i)|_{X^{(m_k)}}u_k=w\hat{g}_iu$ KS-weakly for all $i$, otherwise we pass to a suitable subsequence. By (\ref{E:pointwisemult}) also $\sup_k \mathcal{E}^{(m_k)}_1(a|_{X^{(m_k)}}u_k)<+\infty$ and $\sup_k \mathcal{E}^{(m_k)}_1((w\hat{g}_i)|_{X^{(m_k)}}u_k)<+\infty$. By Lemma \ref{L:Arzela-Ascoli & uniqueness of limits} we can therefore find a sequence $(k_j)_j$ as stated so that (i) and (ii) in Lemma \ref{L:Arzela-Ascoli & uniqueness of limits} hold simultaneously for the sequences $(u_{k_j})_j$, $(a|_{X^{(m_{k_j})}}u_{k_j})_j$ and  $((w\hat{g}_i)|_{X^{(m_{k_j})}}u_{k_j})_j$ with limits $u$, $au$ and $w\hat{g}_iu$, respectively. Our first claim is that
\begin{equation}\label{E:diffusionterm}
\lim_j \left\langle \partial^{(m_{k_j})}(w|_{X^{(m_{k_j})}}), a|_{X^{(m_{k_j})}}\cdot \partial^{(m_{k_j})} u_{k_j}\right\rangle_{\mathcal{H}^{(m_{k_j})}}=\left\langle \partial w, a\cdot \partial u\right\rangle_{\mathcal{H}}.
\end{equation}
To see this note first that by the Leibniz rule for $\partial^{(m_{k_j})}$ each element of the sequence on the left hand side equals
\[\left\langle \partial^{(m_{k_j})}(w|_{X^{(m_{k_j})}}), \partial^{(m_{k_j})}(a|_{X^{(m_{k_j})}} u_{k_j})\right\rangle_{\mathcal{H}^{(m_{k_j})}} - \left\langle \partial^{(m_{k_j})}(w|_{X^{(m_{k_j})}}), u_{k_j}\cdot \partial^{(m_{k_j})} (a|_{X^{(m_{k_j})}}) \right\rangle_{\mathcal{H}^{(m_{k_j})}}.\]
The first term converges to $\left\langle \partial w, \partial(a u)\right\rangle_\mathcal{H}$ by (\ref{E:energiesonly}). In the second summand we can replace $u_{k_j}$ by $H^{(m_{k_j})}_{m_{k_j}}u_{k_j}$, note that by (\ref{E:boundedaction}) and (\ref{E:bumpdecay}) we have 
\[\lim_j \left\|(u_{m_{k_j}}-H^{(m_{k_j})}_{m_{k_j}}u_{k_j})\cdot \partial^{(m_{k_j})}(a|_{X^{(m_{k_j})}})\right\|_{\mathcal{H}^{(m_{k_j})}}=0.\]
By Lemma \ref{L:Arzela-Ascoli & uniqueness of limits} (ii) we also have 
\[\lim_j \left\|(H^{(m_{k_j})}_{m_{k_j}}u_{k_j}-u|_{X^{(m_{k_j})}})\cdot \partial^{(m_{k_j})}(a|_{X^{(m_{k_j})}})\right\|_{\mathcal{H}^{(m_{k_j})}}=0, \]
so that 
\begin{align}
\lim_j &\left\langle \partial^{(m_{k_j})}(w|_{X^{(m_{k_j})}}), u_{k_j}\cdot \partial^{(m_{k_j})} (a|_{X^{(m_{k_j})}}) \right\rangle_{\mathcal{H}^{(m_{k_j})}}\notag\\
&=\lim_j \left\langle \partial^{(m_{k_j})}(w|_{X^{(m_{k_j})}}), u\cdot \partial^{(m_{k_j})} (a|_{X^{(m_{k_j})}}) \right\rangle_{\mathcal{H}^{(m_{k_j})}}\notag\\
&=\left\langle \partial w, u\cdot \partial a\right\rangle_{\mathcal{H}}\notag
\end{align}
by Corollary \ref{C:compatibleb} and polarization. Using the Leibniz rule for $\partial$ we arrive at (\ref{E:diffusionterm}). We next claim that 
\begin{equation}\label{E:thirdterm}
\lim_j \left\langle w|_{X^{(m_{k_j})}}\cdot \hat{b}^{(m_{k_j})},\partial^{(m_{k_j})}u_{m_{k_j}}\right\rangle_{\mathcal{H}^{(m_{k_j})}}=\left\langle w\cdot \hat{b},\partial u\right\rangle_{\mathcal{H}}.
\end{equation}
Each element of the sequence on the left hand side is a finite linear combination with summands
\begin{multline}
\left\langle \partial^{(m_{k_j})}(\hat{f}_i|_{X^{(m_{k_j})}}),\partial^{(m_{k_j})}((w\hat{g}_i)|_{X^{(m_{k_j})}}u_{k_j})\right\rangle_{\mathcal{H}^{(m_{k_j})}}\notag\\
-\left\langle \partial^{(m_{k_j})}(\hat{f}_i|_{X^{(m_{k_j})}}),u_{k_j}\cdot\partial^{(m_{k_j})}((w\hat{g}_i)|_{X^{(m_{k_j})}})\right\rangle_{\mathcal{H}^{(m_{k_j})}}.
\end{multline}
The first term converges to $\left\langle \partial\hat{f}_i, \partial (w\hat{g}_iu)\right\rangle_\mathcal{H}$ by (\ref{E:energiesonly}). To see that 
\begin{equation}\label{E:intermed}
\lim_j\left\langle \partial^{(m_{k_j})}(\hat{f}_i|_{X^{(m_{k_j})}}),u_{k_j}\cdot\partial^{(m_{k_j})}((w\hat{g}_i)|_{X^{(m_{k_j})}})\right\rangle_{\mathcal{H}^{(m_{k_j})}}=\left\langle \partial\hat{f}_i,  u\cdot\partial(w\hat{g}_i)\right\rangle_\mathcal{H}
\end{equation}
let $\varepsilon>0$ and choose $n'$ so that by (\ref{E:approxbyPH}) we have
\begin{equation}\label{E:closeonX}
\mathcal{E}(H_{n'}(w\hat{g}_i)-w\hat{g}_i)^{1/2}<\varepsilon \left\|u\right\|_{\sup}^{-1}\mathcal{E}(\hat{f}_i)^{-1/2}.
\end{equation}
For any $j$ so that $m_{k_j}\geq n'$ we have 
\[H_{m_{k_j}}^{(m_{k_j})}((w\hat{g}_i)|_{X^{(m_{k_j})}})=H_{m_{k_j}}(w\hat{g}_i)|_{X^{(m_{k_j})}}=H_{n'}(w\hat{g}_i)|_{X^{(m_{k_j})}}\]
and by (\ref{E:bumpdecayproducts}) therefore 
\begin{equation}\label{E:closeonXm}
\mathcal{E}^{(m_{k_j})}(H_{n'}(w\hat{g}_i)|_{X^{(m_{k_j})}}-(w\hat{g}_i)|_{X^{(m_{k_j})}})^{1/2}<\varepsilon \mathcal{E}(f_i)^{-1/2}\mathcal{E}(u)^{-1/2}
\end{equation}
for large enough $j$. Since as before we can replace $u_{k_j}$ by $u|_{X^{(m_{k_j})}}$, (\ref{E:closeonXm}) shows that 
\begin{multline}
\lim_j |\left\langle \partial^{(m_{k_j})}(\hat{f}_i|_{X^{(m_{k_j})}}),u_{k_j}\cdot\partial^{(m_{k_j})}((w\hat{g}_i)|_{X^{(m_{k_j})}})\right\rangle_{\mathcal{H}^{(m_{k_j})}} \notag\\
- \left\langle\partial^{(m_{k_j})}(\hat{f}_i|_{X^{(m_{k_j})}}),u\cdot\partial^{(m_{k_j})}(H_{n'}(w\hat{g}_i)|_{X^{(m_{k_j})}})\right\rangle_{\mathcal{H}^{(m_{k_j})}}|<\frac{\varepsilon}{2}.
\end{multline} 
By Corollary \ref{C:compatibleb} and (\ref{E:closeonX}) we have 
\[\lim_j |\left\langle\partial^{(m_{k_j})}(\hat{f}_i|_{X^{(m_{k_j})}}),u\cdot\partial^{(m_{k_j})}(H_{n'}(w\hat{g}_i)|_{X^{(m_{k_j})}})\right\rangle_{\mathcal{H}^{(m_{k_j})}}-\left\langle \partial\hat{f}_i,u\cdot\partial(w\hat{g}_i)\right\rangle_{\mathcal{H}}|<\frac{\varepsilon}{2}.\]
Since $\varepsilon$ was arbitrary, we can combine these two estimates to conclude (\ref{E:intermed}) and therefore (\ref{E:thirdterm}). The identity
\begin{equation}\label{E:secondterm}
\lim_j \left\langle u_{k_j}\cdot b^{(m_{k_j})},\partial^{(m_{k_j})}(w|_{X^{(m_{k_j})}})\right\rangle_{\mathcal{H}^{(m_{k_j})}}=\left\langle u\cdot b,\partial w\right\rangle_\mathcal{H}
\end{equation}
follows by linearity from the fact that by Lemma \ref{L:Arzela-Ascoli & uniqueness of limits} (ii) and Corollary \ref{C:compatibleb} we have
\begin{align}
\lim_j &\left\langle (u_{k_j}g_i|_{X^{(m_{k_j})}})\cdot \partial^{(m_{k_j})}f_i,\partial^{(m_{k_j})}(w|_{X^{(m_{k_j})}})\right\rangle_{\mathcal{H}^{(m_{k_j})}}\notag\\
&=\lim_j\left\langle (ug_i|_{X^{(m_{k_j})}})\cdot \partial^{(m_{k_j})}f_i,\partial^{(m_{k_j})}(w|_{X^{(m_{k_j})}})\right\rangle_{\mathcal{H}^{(m_{k_j})}}\notag\\
&= \left\langle (ug_i)\cdot \partial f_i,\partial w\right\rangle_{\mathcal{H}}. \notag
\end{align}
Together with the obvious identity
\[\lim_j \left\langle (cw)|_{X^{(m_{k_j})}}, u_{k_j}\right\rangle_{L^2(X^{(m_{k_j})},\:\mu^{(m_{k_j})})}=\left\langle cw,u\right\rangle_{L^2(X,\mu)},\]
formulas (\ref{E:diffusionterm}), (\ref{E:thirdterm}) and (\ref{E:secondterm}) imply 
\[\lim_j \mathcal{Q}^{(m_{k_j})}(w|_{X^{(m_{k_j})}}, u_{k_j})=\mathcal{Q}(w,u),\]
what shows condition (ii) in Definition \ref{D:generalized convergence}.
\end{proof}

Theorems \ref{T:approxsame} and \ref{T:KS_spectral} together allow an approximation result involving more general coefficients.  

\begin{theorem}\label{T:general}
Let $a\in\mathcal{F}$ be such that (\ref{E:elliptic}) holds with $0<\lambda<\Lambda$. Let $b,\hat{b}\in\mathcal{H}$ and let $c\in C(X)$. Then we can find $a_n^{(m)}\in\mathcal{F}^{(m)}$ and $b_n^{(m)}, \hat{b}_n^{(m)}\in\mathcal{H}^{(m)}$ such that for any $n$ and $m$ the forms
\begin{equation}\label{E:Qgeneral}
\begin{split}
\mathcal{Q}^{(n,m)}(f,g)&=\left\langle a_n|_{X^{(m)}}\cdot \partial f, \partial g\right\rangle_{\mathcal{H}^{(m)}}-\big\langle g\cdot b_n^{(m)},\partial f\big\rangle_{\mathcal{H}^{(m)}} \\
&-\big\langle f\cdot \hat{b}_n^{(m)}, \partial g\big\rangle_{\mathcal{H}^{(m)}}-\left\langle c|_{X^{(m)}}f,g\right\rangle_{L^2(X^{(m)},\mu^{(m)})}, \quad f,g\in \mathcal{F}^{(m)}
\end{split}
\end{equation}
are sectorial closed forms on $L^2(X^{(m)},\mu^{(m)})$, respectively. Moreover, writing $(\mathcal{L}^{\mathcal{Q}^{(n,m)}},\mathcal{D}(\mathcal{L}^{\mathcal{Q}^{(n,m)}}))$ for the generator of the form $(\mathcal{Q}^{(n,m)},\mathcal{D}(\mathcal{Q}^{(n,m)}))$, we can observe the following.
\begin{enumerate}
\item[(i)]  If $f\in L^2(X,\mu)$, $u$ is the unique weak solution to (\ref{E:ellipticeq}) on $X$ and $u^{(m)}_n$ is the unique weak solution to (\ref{E:ellipticeq}) on $X^{(m)}$ with $\mathcal{L}^{\mathcal{Q}^{(n,m)}}$ and $\Phi_m f$ in place of $\mathcal{L}^{\mathcal{Q}}$ and $f$, then there are sequences $(m_k)_k$ and $(n_l)_l$ with $m_k\uparrow +\infty$ and $n_l\uparrow +\infty$ so that 
\[\lim_l\varlimsup_k \big\|\ext_{m_k}H_{m_k}^{(m_k)}u^{(m_k)}_{n_l} -u\big\|_{\sup}=0.\] 
\item[(ii)] If $\mathring{u}\in L^2(X,\mu)$, $u$ is the unique solution to (\ref{E:paraboliceq}) on $X$ and $u_n^{(m)}$ is the unique weak solution to (\ref{E:paraboliceq}) on $X^{(m)}$ with $\mathcal{L}^{\mathcal{Q}^{(n,m)}}$ and $\Phi_m \mathring{u}$ in place of $\mathcal{L}^{\mathcal{Q}}$ and $\mathring{u}$, then for any $t>0$ there are sequences $(m_k)_k$ and $(n_l)_l$ with $m_k\uparrow +\infty$ and $n_l\uparrow +\infty$ so that 
\[\lim_l\varlimsup_k \big\|\ext_{m_k}H_{m_k}^{(m_k)}u^{(m_k)}_{n_l}(t) -u(t)\big\|_{\sup}=0.\] 
\end{enumerate}
\end{theorem}

\begin{remark}
By \cite[Corollary 1.16]{Attouch} we can find a sequence $(l_k)_k$ with $l_k\uparrow +\infty$ such that 
\[\varlimsup_k \big\|\ext_{m_k}H_{m_k}^{(m_k)}u^{(m_k)}_{n_{l_k}} -u\big\|_{\sup}=0\] 
in the situation of Theorem \ref{T:general} (i) and similarly for (ii).
\end{remark}

The following is a straightforward consequence of the density of $H_\ast(X)$ in $\mathcal{F}$, we omit its short proof.
\begin{lemma}\label{L:help}
The space of finite linear combinations $\sum_i g_i\partial f_i$ with $g_i,f_i\in H_\ast(X)$ is dense in $\mathcal{H}$.
\end{lemma}

We prove Theorem \ref{T:general}.

\begin{proof}
Given $a\in\mathcal{F}$, let $\left(a_n\right)_n\subset H_\ast(X)$ be a sequence approximating $a$ uniformly on $X$ and such that all $a_n$ satisfy \eqref{E:elliptic} with the same constants $0<\lambda<\Lambda$ as $a$.  Let $M>0$ be large enough such that $\lambda_0 := \lambda/2- 1/M>0$. By Lemma \ref{L:help} there exist
\[ b_n := \sum_i g_{n,i} \partial f_{n,i} \qquad\text{and} \qquad \hat{b}_n := \sum_i \hat{g}_{n,i} \partial \hat{f}_{n,i}\]
with $f_{n,i},\hat{f}_{n,i}, g_{n,i},\hat{g}_{n,i} \in  H_\ast(X)$ that approximate $b$ and $\hat{b}$ in $\mathcal{H}$, respectively. For each $n$ we can proceed as in (\ref{E:equivclassm}) and consider the elements
\[ b_n^{(m)}:=\sum_i g_{n,i}|_{X^{(m)}}\cdot \partial^{(m)}(f_{n,i}|_{X^{(m)}})\quad \text{and}\quad \hat{b}_n^{(m)}:=\sum_i \hat{g}_{n,i}|_{X^{(m)}}\cdot \partial^{(m)}(\hat{f}_{n,i}|_{X^{(m)}})\]
of $\mathcal{H}^{(m)}$. With $\gamma_M$ and $\hat{\gamma}_M$ as in \eqref{E:KSgammam} and assuming that, without loss of generality, $c\in C(X)$ satisfies \eqref{E:conseqforfieldssup}, we can conclude that for each $n$ and each sufficiently large $m$ the forms $(\mathcal{Q}^{(n,m)},\mathcal{D}(\mathcal{Q}^{(n,m)}))$ as in \eqref{E:Qgeneral} with $\mathcal{D}(\mathcal{Q}^{(n,m)})=\mathcal{F}^{(m)}$ are closed forms in $L^2(X^{(m)}, \mu^{(m)})$.

To prove (i), suppose that $f\in L^2(X,\mu)$ and $u$ is the unique weak solution to (\ref{E:ellipticeq}) on $X$.
Let $u^{(m)}_n$ be the unique weak solution to (\ref{E:ellipticeq}) on $X^{(m)}$ with $\mathcal{L}^{\mathcal{Q}^{(n,m)}}$ and $\Phi_m(f)$ in place of $\mathcal{L}^{\mathcal{Q}}$ and $f$. By Theorem \ref{T:KS_spectral} we can find a sequence $(m_k)_k$ with $m_k\uparrow +\infty$ so that $\lim_{k\rightarrow\infty} \ext_{m_k}H_{m_k}^{(m_k)}u^{(m_k)}_1 =u_1$ 
uniformly on $X$. Repeated applications of Theorem \ref{T:KS_spectral} allow to thin out $(m_k)_k$ further so that for any $n$ we have   
\[\big\|\ext_{m_k}H_{m_k}^{(m_k)}u^{(m_k)}_j-u_j\big\|_{\sup}<2^{-n}, \quad j\leq n,\]
provided that $k$ is greater than some integer $k_n$ depending on $n$. On the other hand Theorem \ref{T:approxsame} allows to find a sequence $(n_l)_l$ with $n_l\uparrow +\infty$ such that $\lim_{l\rightarrow\infty} u_{n_l}=u$ uniformly on $X$, and combining these facts, we obtain (i). Statement (ii) is proved in the same manner.
\end{proof}

\section{Discrete and metric graph approximations}\label{S:examples}

\subsection{Discrete approximations}\label{SS:discretegraphs}

We describe approximations in terms of discrete Dirichlet forms, our notation follows that of Subsection \ref{SS:setup}. Let $(\mathcal{E},\mathcal{F})$ be a local regular resistance form on the compact space $(X,R)$, obtained under Assumption \ref{A:basic} as in Section \ref{SS:setup}, and suppose that also Assumption \ref{A:gotozero} is satisfied. Let $X^{(m)}=V_m$, $\mathcal{E}^{(m)}=\mathcal{E}_{V_m}$ and $\mathcal{F}^{(m)}=\ell(V_m)$ be the discrete energy forms on the finite subsets $V_m$ as in (\ref{E:limitform}). Clearly Assumption \ref{A:convergence} is satisfied, 
note that for every $u \in H_m(X)$ we have $\E_{V_m}( u|_{V_m}) = \E(u)$ and that (\ref{E:weakconvenergy}) is immediate from (\ref{E:energymeasureapprox}). Since every element of $\ell(V_m)$ is the pointwise restriction of a function in $H_m(X)$, the operator $H_m^{(m)}$ is the identity operator $\id_{\mathcal{F}^{(m)}}$, so that Assumption \ref{A:decay} is trivially satisfied, as pointed out in Remark \ref{R:trivial}. 

Now let $\mu$ be a finite Borel measure on $X$ such that for any $m$ the value $V(m):=\inf_{\alpha\in\mathcal{A}_m}\mu(X_\alpha)$ is strictly positive. Following \cite{PostSimmer17} we define, for each $m$, a measure $\mu^{(m)}$ on $V_m$ by
\[\mu^{(m)}(\{p\}):=\int_X \psi_{p,m}(x)d\mu(x), \quad p\in V_m,\]
where $\psi_{p,m}\in H_m(X)$ is the (unique) harmonic extension to $X$ of the function $\mathbf{1}_{\{p\}}$ on $V_m$. Since $X_\alpha^{(m)}=V_\alpha$ and $\sum_{p\in V_\alpha}\psi_{p,m}(x)=1$ for all $m$, $\alpha\in\mathcal{A}_m$ and $x\in X_\alpha$, we have 
\[\mu^{(m)}(X_\alpha^{(m)})=\sum_{p\in V_\alpha}\mu^{(m)}(\{p\})=\int_X \sum_{p\in V_\alpha}\psi_{p,m}(x)\mu(dx)\geq \mu(X_\alpha)\geq V(m)\]
for all $m$ and $\alpha\in\mathcal{A}_m$, so that Assumption \ref{A:connection} (i) is seen to be satisfied. 

For each $m$ let $\Phi_m$ be a linear operator $\Phi_m:L^2(X,\mu)\to \ell^2(V_m, \mu^{(m)})$ defined by 
\[\Phi_m f(p):= \frac{1}{\mu^{(m)}(\{p\})}\left\langle f,\psi_{p,m}\right\rangle_{L^2(X,\mu)}, \quad p\in V_m,\quad f\in L^2(X,\mu).\]
In \cite[proof of Theorem 1.1]{PostSimmer17} it was shown that for each $m$ the adjoint $\Phi_m^\ast$ of $\Phi_m$ equals the harmonic extension operator $\ext_m:\ell^2(V_m, \mu^{(m)})\to H_m(X)$,
\[\ext_m v =\sum_{p\in V_m} v(p)\psi_{p,m},\quad v\in \ell^2(V_m,\mu^{(m)})\]
which satisfies $\Vert \ext_m f \Vert_{L^2(X, \mu)} \leq \Vert f \Vert_{\ell^2(V_m, \mu^{(m)})}$ for all $f \in \ell^2(V_m, \mu^{(m)})$. Consequently \eqref{E:uninormbound} is fulfilled, and also (\ref{E:L2consist}) holds. The function $\psi_{p,m}$ is supported on the union of all $X_\alpha$, $\alpha\in \mathcal{A}_m$, which contain the point $p$. By Assumption \ref{A:gotozero} we therefore have 
\begin{equation}\label{E:supportstozero}
\lim_{m\rightarrow\infty} \sup_{p \in V_m} \diam_R \left( \supp \psi_{p,m}\right) = 0.
\end{equation}
If a sequence $(u_m)_m\subset \mathcal{F}$ is such that $\sup_m \E(u_m)<\infty$ then by (\ref{E:resistanceest}) it is equicontinuous, and combined with (\ref{E:supportstozero}) it follows that given $\varepsilon>0$ we have 
\[\sup_{p\in V_m}\sup_{x\in\psi_{p,m}}|u_m(p)-u_m(x)|<\varepsilon\]
whenever $m$ is large enough, and consequently 
\[\Vert \Phi_m u_m -  u_m|_{V_m} \Vert_{\ell^2(V_m, \mu^{(m)})}^2 \leq \sum_{p \in V_m} \frac{1}{\mu^{(m)}(\{p\})} \left(\int_X \vert u_m(x) -u_m(p) \vert \psi_{p,m}(x) d \mu(x)\right)^2 <\varepsilon^2\]
for such $m$, note that $\sum_{p\in V_m} \psi_{p,m}(x)=1$ for all $m$ and $x\in X$. This shows \eqref{E:connection}.
For every $u \in \mathcal{F}$ it follows that
\[\lim_{m \rightarrow \infty} \Vert u|_{V_m} \Vert^2_{\ell^2(V_m, \mu^{(m)})} = \lim_{m\rightarrow\infty}\sum_{p \in V_m}\int_X \left[(u(p) - u(x))(u(p)+u(x)) +  u^2(x)\right]\psi_{p,m}(x) d \mu(x) = \Vert u \Vert^2_{L^2(X,\mu)},\]
since $u$ is bounded and $\lim_m \sum_{p \in V_m}\int_X (u(p) - u(x)) \psi_{p,m}(x) d \mu(x)=0$ by (\ref{E:supportstozero}) as above, proving \eqref{E:KSconvHilbert}. To verify the remaining condition \eqref{E:adjoint} note that for $u \in H_n(X)$ we have
\[\Vert \Phi^\ast_m \Phi_m u - u \Vert_{L^2(X,\mu)} \leq \Vert \Phi_m^\ast\Vert_{\ell^2(V_m, \mu^{(m)}) \rightarrow L^2(X, \mu)} \Vert \Phi_m u -  u|_{V_m} \Vert_{\ell^2(V_m, \mu^{(m)})} + \Vert \Phi_m^\ast (u|_{V_m} ) - u \Vert_{L^2(X,\mu)}, \]
and since $\Phi_m^\ast (u|_{V_m}) = H_m u$ the last summand is bounded by $\diam_R(X)^{1/2}\E(H_m u - u)^{1/2} \mu(X)^{1/2}$. Using (\ref{E:approxbyPH}), \eqref{E:uninormbound} and \eqref{E:connection} condition \eqref{E:adjoint} now follows.

\begin{examples}\label{Ex:pcf}
It is well known that p.c.f. self-similar structures form a subclass of finitely ramified sets. Because of its importance, and since we will discuss metric graph approximations for this subclass in the next section, we provide some details. Let $(K,S,\lbrace F_j\rbrace_{j \in S})$ be a connected \emph{post-critically finite (p.c.f.) self-similar structure}, see \cite[Definitions 1.3.1, 1.3.4 and 1.3.13]{Ki01}. The set of finite words $w=w_1w_2...w_m$ of length $|w|=m$ over the alphabet $S$ is denoted by $W_m:=S^m$, and we write $W_\ast=\bigcup_{m\geq 0} W_m$. Given a word $w\in W_m$ we write $F_w=F_{w_1}\circ F_{w_2}\circ ...\circ F_{w_m}$ and use the abbreviations $K_w:=F_w(K)$ and $V_w :=F_w(V_0)$. Then $(K,\{K_w\}_{w\in W_\ast}, \{V_w\}_{w\in W_\ast})$ is a finitely ramified cell structure in the sense of Definition \ref{D:finitelyramified}. We consider the discrete sets $V_m:=\cup_{|w|=m} V_w$, $m\geq 0$, and assume that $((\mathcal{E}_{V_m},\ell(V_m)))_m$ is a sequence of Dirichlet forms associated with a \emph{regular harmonic structure} on $K$, \cite[Definitions 3.1.1 and 3.1.2]{Ki01}, that is, there exist  constants $r_j \in (0,1)$, $j\in S$, a Dirichlet form $\mathcal{E}_{V_0}(u)=\frac12\sum_{p\in V_0}\sum_{q\in V_0}c(0;p,q)(u(p)-u(q))^2$ on $\ell(V_0)$, for all $m\geq 1$ we have 
\begin{equation}\label{E:simp_reg_structure}
\mathcal{E}_{V_m}(u,v)= \sum_{w\in W_m}r_w^{-1}\mathcal{E}_0(u\circ F_w, v\circ F_w),\quad u,v\in \ell(V_m),
\end{equation}
where $r_w:=r_{w_1}\dots r_{w_m}$ for $w=w_1...w_m$, and $(\mathcal{E}_{V_{m+1}})_{V_m}=\mathcal{E}_{V_m}$ for all $m\geq 0$. The regularity of the harmonic structure implies in particular that $\Omega = K$, \cite[Theorem 3.3.4]{Ki01}, and the limit (\ref{E:limitform}) defines a (self-similar) local regular resistance form $(\mathcal{E},\mathcal{F})$ on $K$. Assumptions \ref{A:basic} and \ref{A:gotozero} are clear from general theory, \cite{Ki01}.
\end{examples}

\begin{examples}
Further examples which fit into the above scheme are for instance non-self-similar resistance forms on Sierpinski gaskets associated with regular harmonic structures, \cite{Meyers}, certain energy forms on random Sierpinski gaskets, \cite{Hambly92, Hambly97}, finitely ramified graph-directed sets with a regular harmonic structure, \cite[Section 4, in particular p. 18]{HamblyNyberg03}, or basilica Julia sets with a regular harmonic structure, \cite[Theorem 3.9]{RogersTeplyaev10}. 
\end{examples}

\subsection{Metric graph approximations}\label{SS:metricgraphs}

We describe approximations in terms of local Dirichlet forms on metric graphs (also called 'cable-systems' in \cite{BBK06}). We follow the method in \cite{HinzMeinert19+} and therefore specify to the case where $X$ is a post-critically finite self-similar set $K$. Let the setup and notation be as in Examples \ref{Ex:pcf}.

For each $m\geq 0$ we consider $V_m$ as the vertex set of a finite simple (unoriented) graph $G_m=(V_m,E_m)$ with two vertices $p,q\in V_m$ being the endpoints of the same edge $e\in E_m$ if there is a word $w$ of length $|w|=m$ such that $F_w^{-1}p, F_w^{-1}q\in V_0$ and $c(0;F_w^{-1},F_w^{-1}q)>0$. For each $m$ and $e\in E_m$ let $l_e$ be a positive number and identify the edge $e$ with an oriented copy of the interval $(0,l_e)$ of length $l_e$, we write $i(e)$ and $j(e)$ for the initial and the terminal vertex of  $e$, respectively. This yields a sequence $(\Gamma_m)_{m\geq 0}$ of metric graphs $\Gamma_m$, and for each $m$ the set $X_{\Gamma_{m}}=V_m\cup \bigcup_{e\in E_m} e$, endowed with the natural length metric, becomes a compact metric space See \cite{HinzMeinert19+} for details and further references. By construction we have $X_{\Gamma_m}\subset X_{\Gamma_{m+1}}$ and $X_{\Gamma_m}\subset K$ for each $m$.

On the space $X_{\Gamma_m}$ we consider the bilinear form $(\mathcal{E}_{\Gamma_m}, \dot{W}^{1,2}(X_{\Gamma_m}))$, where
\[\mathcal{E}_{\Gamma_m}(f):=\sum_{w\in W_m}r_w^{-1}\sum_{e\in E_m,\, e\subset K_w} l_e\mathcal{E}_e(f_e)\quad \text{and}\quad \mathcal{E}_e(f_e)=\int_0^{l_e} (f_e'(t))^2 dt\]
and 
\[\dot{W}^{1,2}(X_{\Gamma_m})=\{f=(f_e)_{e\in E_m}\in C(X_{\Gamma_m}): f_e\in  \dot{W}^{1,2}(e), \, \mathcal{E}_{\Gamma_m}(f)<+\infty \}.  \]
Here $f_e$ is the restriction of $f$ to $e\in E_m$ and $\dot{W}^{1,2}(e)$ is the homogeneous Sobolev space consisting of locally Lebesgue integrable functions $g$ on the edge $e$ such that 
\[\mathcal{E}_e(g):=\int_0^{l_e} (g'(s))^2\:ds<+\infty,\]
where the derivative $g'$ of $g$ is understood in the distributional sense. Each form $\mathcal{E}_e$, $e\in E_m$, satisfies 
\begin{equation}\label{E:resistonedge}
(f_e(s)-f_e(s'))^2\leq l_e\mathcal{E}_{e}(f_e)
\end{equation}
for any $f\in \dot{W}^{1,2}(X_{\Gamma_m})$ and any $s,s'\in e$. See \cite{HinzMeinert19+} for further details. We approximate $K$, endowed with $(\mathcal{E},\mathcal{F})$ as in Examples \ref{Ex:pcf}, by the spaces $X^{(m)}=X_{\Gamma_m}$ carrying the resistance forms $\mathcal{E}^{(m)}=\mathcal{E}_{\Gamma_m}$ with domains $\mathcal{F}^{(m)}=\dot{W}^{1,2}(X_{\Gamma_m})$.

To a function $f\in \dot{W}^{1,2}(X_{\Gamma_m})$ which is linear on each edge $e\in E_m$ we refer as \emph{edge-wise linear} function, and we denote the closed linear subspace of $\dot{W}^{1,2}(X_{\Gamma_m})$ of
such functions by $EL_m$. If $f\in EL_m$, then its derivative on a fixed edge $e$ is the constant function
$f_e'=l_e^{-1}(f(j(e))-f(i(e)))$, so that 
\begin{equation}\label{E:harmoniccase}
\mathcal{E}_e(f_e)=\int_0^{l_e}(f'_e(t))^2dt=\frac{1}{l_e}(f(j(e))-f(i(e)))^2
\end{equation}
on each $e\in E_m$. For a general function $f\in \dot{W}^{1,2}(X_{\Gamma_m})$ formula (\ref{E:harmoniccase}) becomes an inequality in which the left hand side dominates the right hand side. Given a function $g\in \ell(V_m)$ it has a unique extension $h$ to $X_{\Gamma_m}$ which is edge-wise linear, $h\in EL_m$. In particular, if $f\in H_m(K)$ is an $m$-piecewise harmonic function on the p.c.f. self-similar set $K$ then its pointwise restriction $f|_{X_{\Gamma_m}}$ to $X_{\Gamma_m}$ is a member of $EL_m$, and $\mathcal{E}_{\Gamma_m}(f|_{X_{\Gamma_m}})=\mathcal{E}(f)$. Since any such $f\in H_m(K)$ is uniquely determined by its values on $V_m\subset X_{\Gamma_m}$, this restriction map is injective, and Assumption \ref{A:convergence} (i) is seen to be satisfied. Assumption \ref{A:convergence} (ii) is verified in the following lemma. By $\nu^{(m)}_f$ we denote the energy measures associated with the form $(\mathcal{E}_{\Gamma_m}, \dot{W}^{1,2}(X_{\Gamma_m}))$.

\begin{lemma}
For any $f \in \mathcal{F}$ we have $\nu_f=\lim_{m\to\infty} \nu^{(m)}_{H_m(f)|_{X_{\Gamma_m}}}$ weakly on $K$.
\end{lemma}

\begin{proof}
For $f\in\mathcal{F}$ and nonnegative $g\in C(K)$ we have 
\[ \left\vert \left( \int_K g d \nu_f\right)^\frac{1}{2} - \left(\int_K g d \nu_{H_m(f)}\right)^\frac{1}{2} \right\vert \leq  \Vert g \Vert_{\sup} \E(f-H_m(f)),\]
see \cite[Section 3.2]{FOT94}. This implies the relation $\int_K g\:d\nu_f=\lim_m \int_K g\:d\nu_{H_m(f)}$, which by the standard decomposition $g=g^+-g^-$ remains true for arbitrary $g\in C(K)$. For any $m$ we have 
\[\int_K g\:d\nu_{H_m(f)}=\sum_{w\in W_m}r_w^{-1}\sum_{e\in E_m, e\subset K_w} l_e^2(H_m(f)_e')^2g_e(i(e))\]
by (\ref{E:approxbydiscreteforms}), here $H_m(f)_e'\in \mathbb{R}$ denotes the slope of the restriction $H_m(f)_e$ of $H_m(f)$ to $e$. On the other hand,
\[\int_K g\:d\nu_{H_m(f)|_{X_{\Gamma_m}}}^{(m)}=\sum_{w\in W_m}r_w^{-1}\sum_{e\in E_m, e\subset K_w} l_e(H_m(f)_e')^2\int_0^{l_e}g_e(t)dt,\]
and given $\varepsilon>0$ we have $\sup_{e \in E_{m}} \sup_{s,t \in e} \vert g(s) -g(t) \vert <\varepsilon$ whenever $m$ is large enough, and in this case,
\[\left|\int_K g\:d\nu_{H_m(f)}-\int_K g\:d\nu_{H_m(f)|_{X_{\Gamma_m}}}^{(m)}\right|\leq \varepsilon\: \sum_{w\in W_m}r_w^{-1}\sum_{e\in E_m, e\subset K_w} l_e^2(H_m(f)_e')^2=\varepsilon\:\mathcal{E}_{\Gamma_m}(H_m(f)|_{X_{\Gamma_m}})\leq \varepsilon\:\mathcal{E}(f).\]
Combining, it follows that $\lim_m \int_K g\:d\nu_f=\lim_m\int_K g\:d\nu_{H_m(f)|_{X_{\Gamma_m}}}$.
\end{proof}

We verify condition \eqref{E:bumpdecay} in Assumption \ref{A:decay} in the present setup. It states that the small oscillations on the interior of individual edges in $X_{\Gamma_m}$ subside uniformly for sequences of functions with a uniform energy bound. 

\begin{lemma}\label{L:decay1}
Let $(f_m)_m$ be a sequence of functions $f_m \in \dot{W}^{1,2}(X_{\Gamma_m})$ such that $\sup_m \E_{\Gamma_m}(f_m)<+ \infty$ and   $f_m|_{V_m} = 0$ for all $m$. Then $\lim_m \Vert f_m \Vert_{\sup, X_{\Gamma_m}}=0$.
\end{lemma}
\begin{proof}
By \eqref{E:resistonedge} we have   
\[\sup_{t\in e}\vert \left(f_m\right)_e(t) \vert^2 \leq l_e\:\int_0^{l_e} \left(\left(f_m'\right)_e(t)\right)^2 dt\]
on each $e\in E_m$ and consequently
\begin{align*}
\Vert f_m \Vert^2_{\sup,X_{\Gamma_m}} \leq \sum_{e \in E_m} \sup_{t\in e} \vert \left(f_m\right)_e(t)\vert^2 \leq (\max_{j\in S}r_j)^m \sup_n \E_{\Gamma_n} (f_n) .
\end{align*}
\end{proof}

By $H_{\Gamma_m}$ we denote the orthogonal projection in 
$\dot{W}^{1,2}(X_{\Gamma_m})$ onto $EL_m$. Given $f_m\in \dot{W}^{1,2}(X_{\Gamma_m})$ it clearly follows that $f_m - H_{\Gamma_m}f_m \in \dot{W}^{1,2}(X_{\Gamma_m})$, we have $\left(f_m - H_{\Gamma_m}f_m\right)|_{V_m} = 0$ and $\E_{\Gamma_m} (f_m - H_{\Gamma_m}f_m)\leq  \E_{\Gamma_m}(f_m)$. We verify \eqref{E:bumpdecayproducts} in Assumption \ref{A:decay}.

\begin{lemma}\label{L:decay2}
Given $f,g \in H_n(X)$, we have 
\[\lim_{m\rightarrow \infty} \E_{\Gamma_m} \left( f|_{X_{\Gamma_m}}g|_{X_{\Gamma_m}}-H_{\Gamma_m}\left(f|_{X_{\Gamma_m}}g|_{X_{\Gamma_m}}\right) \right) = 0.\]
\end{lemma}
\begin{proof}
We first note that for any $m\geq n$ the functions $f_e$ and $g_e$ are linear on any fixed $e \in E_m$, 
\begin{align*}
f_e(t) = f_e(0) + f_e'\cdot t \qquad \text{and} \qquad g_e(t) = g_e(0) + g_e'\cdot t, \quad t \in [0,l_e],
\end{align*}
with slopes $f_e'\in \mathbbm{R}$ and $g_e' \in \mathbb{R}$, respectively. Therefore $\E_e(f_e)= l_e \left(f_e'\right)^2$ for each such $e$ and 
\begin{align}\label{E:preparation}
l_e^2\left(f_e'\right)^2 \leq \sum_{\vert w \vert = m } \sum_{e \in E_m, e \in K_w} l_e^2\left(f_e'\right)^2 \leq (\max_i r_i)^m \sup_{m\geq n}\E_{\Gamma_m}\left(f|_{X_{\Gamma_m}}\right) = (\max_i r_i)^m \E(f), 
\end{align}
similarly for the function $g$. Since \[\left(fg\right)_e(t) = f_e(t) g_e(t)= f_e(0)g_e(0) + g_e(0)f_e'\cdot t +  f_e(0)g_e'\cdot t + f_e'g_e'\cdot t^2 \]
and therefore in particular
\begin{align*}
H_{\Gamma_m} \left((fg)|_e\right)(t) &= f_e(0)g_e(0) + \frac{t}{l_e}\left( f_e(l_e)g_e(l_e) - f_e(0)g_e(0) \right) \\ 
&= f_e(0)g_e(0) + \frac{t}{l_e}\left( f_e'g_e' l_e^2 + \left(f_e(0) g_e' + g_e(0) f_e'\right)l_e \right)
\end{align*}
we obtain 
\[ \left( \left(fg\right)_e - H_{\Gamma_m} \left((fg)|_e\right) \right)(t) = f_e'g_e't^2 - f_e' g_e' l_e t,\quad t\in [0,l_e].\]
This implies that for any edge $e \in E_m$ we have 
\[\E_e \left(\left( \left(fg\right)_e - H_{\Gamma_m} \left((fg)|_e\right) \right)(t)\right) = \left(f_e'g_e'\right)^2\int_0^{l_e}\left(2t-l_e\right)^2 dt = \frac{1}{3}\left(f_e'g_e'\right)^2 l_e^3, \quad t\in [0,l_e]. \]   
Summing up over $e \in E_m$ and using \eqref{E:preparation}, we see that 
\begin{align*}
\E_{\Gamma_m}( f|_{X_{\Gamma_m}}g|_{X_{\Gamma_m}} &- H_{\Gamma_m}\left(f|_{X_{\Gamma_m}}g|_{X_{\Gamma_m}}\right) )\\ 
&= \sum_{|w|=m}r_w^{-1}\sum_{e\in E_m, e\subset K_w} l_e \mathcal{E}_e( f|_{X_{\Gamma_m}}g|_{X_{\Gamma_m}} - H_{\Gamma_m}\left(f|_{X_{\Gamma_m}}g|_{X_{\Gamma_m}}\right))\\
&\leq \frac{1}{3} \sum_{|w|=m}r_w^{-1}\sum_{e\in E_m, e\subset K_w} l_e^4 (f_e')^2(g_e')^2 \\ 
&\leq \frac{1}{3}(\max_i r_i)^m \E(f) \sum_{|w|=m}r_w^{-1}\sum_{e\in E_m, e\subset K_w} l_e^2 g_e'^2 \\
&= \frac{1}{3}(\max_i r_i)^m \E(f)\E(g).
\end{align*}
\end{proof}

In what follows let $\mu$ be a finite Borel measure on $K$ so that $V(m):=\inf_{|w|=m}\mu(K_w)>0$ for each $m$. Given an edge $e\in E_m$ we set
\begin{equation}\label{E:definepsi1}
\psi_{e,m}(x):=\frac{1}{\deg_m(i(e))}\psi_{i(e),m}(x)+\frac{1}{\deg_m(j(e))}\psi_{j(e),m}(x),\quad x\in K,
\end{equation}
to obtain a function $\psi_{e,m}$ which satisfies 
\begin{equation}\label{E:sumpsiem1}
\sum_{e\in E_m}\left\langle \psi_{e,m},\mathbf{1}\right\rangle_{L^2(K,\mu)}=\sum_{p\in V_m}\psi_{p,m}(x)= 1,\quad x\in K.
\end{equation}
We endow the space $X_{\Gamma_m}$ with the measure $\mu^{(m)}:=\mu_{\Gamma_m}$ which on each individual edge $e\in E_m$ equals 
\[\frac{1}{l_e}\left(\int_K \psi_{e,m}(x)\mu(dx)\right)\:\lambda^1|_{e},\]
here $\lambda^1$ denotes the one-dimensional Lebesgue measure. Writing $X_w^{(m)}$ for $X_{\Gamma_m}\cap K_w=V_w=\bigcup_{e\in E_m, e\subset K_w} e$, we see that
\[\mu_{\Gamma_m}(X_w^{(m)})=\sum_{e\in E_m, e\subset K_w} \int_K\psi_{e,m}(x) \mu(dx)\geq \int_{K_w}\psi_{e,m}(x) \mu(dx)=\mu(K_w)\geq V(m),\]
so part (i) of Assumption \ref{A:connection} is satisfied. The remaining conditions in Assumption \ref{A:connection} (ii)-(iv) now follow from results in \cite{HinzMeinert19+}: If for each $m$ we consider the linear operator $\Phi_m:L^2(X_{\Gamma_m}, \mu_{\Gamma_m})\rightarrow L^2(K, \mu)$ defined by 
\[\Phi_m u(t)=\sum_{e\in E_m}\:\mathbf{1}_e(t)\:\frac{\left\langle u,\psi_{e,m}\right\rangle_{L^2(K,\mu)}}{\left(\int_K \psi_{e,m} d\mu\right)},\quad u\in L^2(K,\mu),\]
then \eqref{E:uninormbound} and \eqref{E:KSconvHilbert} are satisfied by \cite[Prop. 4.1]{HinzMeinert19+} (there the operator $\Phi_m$ is denoted by $J_{0,m}^\ast$), and a proof of \eqref{E:adjoint} is provided in \cite[Lemma C.3]{HinzMeinert19+}. 
Condition \eqref{E:connection} follows from Lemma \cite[Lemma C.2]{HinzMeinert19+} (there the pointwise restriction of $m$-harmonic functions to $X_{\Gamma_m}$ is denoted by $\widetilde{J}_{1,m}$). 
For the operators $\ext_m H_{\Gamma_m}: \dot{W}^{1,2}(X_{\Gamma_m})\rightarrow H_m(K)$ (denoted by $J_{1,m}$ in Lemma \cite[Lemma C.2]{HinzMeinert19+}) we can use \cite[Lemma C.2]{HinzMeinert19+} and \cite[Prop. 4.1]{HinzMeinert19+} to see that  if $(f_m)_m$ is a sequence of functions $f_m \in \dot{W}^{1,2}(X_{\Gamma_m})$ with $\sup_m \E_{\Gamma_m}(f_m)<\infty$ then 
\[\Vert \ext_m  H_{\Gamma_m}f_m \Vert_{L^2(K, \mu)}\leq \Vert f_m \Vert_{L^2(X_{\Gamma_m}, \mu_{\Gamma_m})} +C(\max_i r_i)^{m/2}\sup_m \mathcal{E}_{\Gamma_m}(f_m)^{1/2}\]
with a positive constant $C$ depending only on $N$. Consequently also \eqref{E:L2consist} is satisfied.

\subsection{Short remarks on possible generalizations}\label{SS:generalizations}

Although not covered by the above results, we conjecture that under suitable additional conditions one can produce similar results for p.c.f. self-similar sets with non-regular harmonic structures, diamond lattice fractals, \cite{AkkermansDunneTeplyaev, Alonso, HamblyKumagai}, Laaks\o~ spaces, \cite{Steinhurst}, and compact fractafolds, \cite{Str03}. 
Well-known general results, \cite[Proposition 2.10 and Theorem 2.14]{Ki03}, motivate the question how to implement discrete or metric graph approximations for the Sierpinski carpet, endowed with its standard energy form. Another question is how to establish approximations by graph-like manifolds, \cite{PostSimmer18}, for non-symmetric forms of type (\ref{E:Q}), and a transparent discussion of drift and divergence terms should be quite interesting. A further open question is how to establish approximations in energy norm. This would most likely have to involve second order splines as for instance discussed in \cite{StrU00} for the case of the Sierpinski gasket endowed with its standard energy form and the self-similar Hausdorff measure. Several tools used in the present paper rely heavily on the use of linear and harmonic functions, and second order versions are not so straightforward to see. A question in a different direction, particularly interesting in connection with probability, \cite{CHT18}, is how to approximate equations involving nonlinear first order terms. There are results on the convergence of certain non-linear operators along varying spaces, \cite{Toelle10}, but they do not cover these cases.

\section{Restrictions of vector fields}\label{S:restriction}

As mentioned in Remark \ref{R:restriction}, a finitely ramified cell structure also permits a restriction
operation for specific vector fields. As discussed in \cite{HinzMeinert19+} the spaces $\im \partial$ and $\mathcal{F}/\sim$ are isometric as Hilbert spaces, and similarly for $\im \partial^{(m)}$ and $\mathcal{F}^{(m)}/\sim$. Recall also that for each $m$ the pointwise restriction $u\mapsto u|_{X^{(m)}}$ is an isometry from $H_m(X)/\sim$ onto $H_m(X^{(m)})/\sim$. Therefore (\ref{E:equivclass}) and (\ref{E:equivclassm}) give rise to a well defined restriction of gradients of $n$-harmonic functions: Given $f\in H_n(X)$ and $m\geq n$ we can define the restriction of $\partial f$ to $X^{(m)}$ by 
\begin{equation}\label{E:restrictgradients}
(\partial f)|_{X^{(m)}}:=\partial^{(m)}(f|_{X^{(m)}}), 
\end{equation}
and this operation is an isometry from $\partial (H_m(X))$ onto $\partial^{(m)}(H_m(X^{(m)}))$, see for instance \cite[Subsection 4.4]{HinzMeinert19+}. In the sequel we assume, in addition to the assumptions made in Section \ref{S:varying}, that for each $m$ and each $\alpha\in \mathcal{A}_m$ the form $\mathcal{E}_{\alpha}(u)=\frac12\sum_{p\in V_\alpha}\sum_{q\in V_\alpha}c(m;p,q)(u(p)-u(q))^2$, $u\in\mathcal{F}$, is irreducible on $V_\alpha$. Following \cite{IRT12} we define subspaces $\mathcal{H}_m$ of $\mathcal{H}$ by
\[\mathcal{H}_m:=\bigg\lbrace\sum_{\alpha\in\mathcal{A}_m} \mathbf{1}_{X_\alpha}\partial h_\alpha:\quad \text{$h_\alpha\in H_m(X)$ for all $\alpha\in \mathcal{A}_m$}\bigg\rbrace.\]
Then $\mathcal{H}_m\subset\mathcal{H}_{m+1}$ for all $m$, \cite[Lemma 5.3]{IRT12}, and $\bigcup_{m\geq 0}\mathcal{H}_m$ is dense in $\mathcal{H}$, \cite[Theorem 5.6]{IRT12}. To generalize (\ref{E:restrictgradients}) we now define a pointwise restriction of elements of $\mathcal{H}_m$ to $X^{(m)}$ by
\begin{equation}\label{E:restrictfields}
\bigg(\sum_{\alpha\in\mathcal{A}_m}\mathbf{1}_{X_\alpha}\partial h_\alpha\bigg)|_{X^{(m)}}:=\sum_{\alpha\in\mathcal{A}_m}\mathbf{1}_{X_\alpha^{(m)}}\partial^{(m)}(h_\alpha|_{X^{(m)}}),
\end{equation}
and clearly this restriction operation maps $\mathcal{H}_m$ into $\mathcal{H}^{(m)}$. Thanks to the finitely ramified cell structure of $X$ it is straightforward to see that this definition is correct.  The following auxiliary result is parallel to Corollary \ref{C:compatibleb}.

\begin{lemma}\label{L:restrictb}
For any $b\in \mathcal{H}_n$ and any $g\in C(X)$ we have 
\begin{equation}\label{E:limfr}
\lim_m \left\|g|_{X^{(m)}}\cdot b|_{X^{(m)}}\right\|_{\mathcal{H}^{(m)}}= \left\|g\cdot b\right\|_{\mathcal{H}}.
\end{equation}
\end{lemma}
%

\begin{proof}
Let $\varepsilon>0$. Choose $n_g\geq n$ sufficiently large such that 
\[\sup_{\beta\in\mathcal{A}_{n_g}}\sup_{x,y\in X_\beta}|g(x)^2-g(y)^2|<\frac{\varepsilon}{5\sum_{\alpha\in\mathcal{A}_n}\mathcal{E}(h_\alpha)}.\]
For all $\beta\in\mathcal{A}_{n_g}$ choose $x_\beta\in X_\beta\setminus V_{n_g}$ and define $\widetilde{g}(x):= g(x_\beta)$ if $x\in X_\beta\setminus V_{n_g}$ and $\widetilde{g}(x):=0$ if $x \in V_{n_g}$. Then we we have 
\[\sup_{\beta\in\mathcal{A}_{n_g}}\sup_{x\in X_\beta\setminus V_{n_g}} |g(x)^2-\widetilde{g}(x)^2|<\frac{\varepsilon}{5\sum_{\alpha\in\mathcal{A}_n}\mathcal{E}(h_\alpha)}\]
and therefore
\begin{equation}\label{E:approx1}
\bigg|\sum_{\alpha\in\mathcal{A}_n}\int_{X_\alpha\setminus V_{n_g}} g|_{X^{(m)}}^2 d\nu_{h_\alpha|_{X^{(m)}}}^{(m)}-\int_{X_\alpha\setminus V_{n_g}} \widetilde{g}|_{X^{(m)}}^2 d\nu_{h_\alpha|_{X^{(m)}}}^{(m)}\bigg|<\frac{\varepsilon}{5}
\end{equation}
for all $m$ and also
\begin{equation}\label{E:approx2}
\bigg|\sum_{\alpha\in\mathcal{A}_n}\int_{X_\alpha\setminus V_{n_g}} g^2 d\nu_{h_\alpha}-\int_{X_\alpha\setminus V_{n_g}} \widetilde{g}^2 d\nu_{h_\alpha}\bigg|<\frac{\varepsilon}{5}.
\end{equation}
The energy measures $\nu_{h_\alpha}$ are nonatomic, hence by (\ref{E:weakconvenergy}) and the Portmanteau lemma we can find a positive integer $m_\varepsilon\geq n_g$ so that for all $m\geq m_\varepsilon$ and all $\alpha\in\mathcal{A}_n$ we have 
\begin{equation}\label{E:Vngsmall}
\nu_{h_\alpha|_{X^{(m)}}}^{(m)}(V_{n_g})<\frac{\varepsilon}{2|\mathcal{A}_n|^2\left\|g\right\|_{\sup}^2}
\end{equation}
and 
\begin{equation}\label{E:measuresclose}
\bigg|\nu_{h_\alpha|_{X^{(m)}}}^{(m)}(X_\beta\setminus V_{n_g})-\nu_{h_\alpha}(X_\beta\setminus V_{n_g})\bigg|<\frac{\varepsilon}{2|\mathcal{A}_n|\left\|g\right\|_{\sup}^2}.
\end{equation}
Since (\ref{E:measuresclose}) implies 
\begin{align}
\bigg|\sum_{\alpha\in \mathcal{A}_n}\sum_{\beta\in \mathcal{A}_{n_g}}g(x_\beta)^2 & \nu_{h_\alpha|_{X^{(m)}}}^{(m)}(X_\alpha\cap X_\beta\cap V_{n_g}^c)-\sum_{\alpha\in \mathcal{A}_n}\sum_{\beta\in \mathcal{A}_{n_g}}g(x_\beta)^2 \nu_{h_\alpha}(X_\alpha\cap X_\beta\cap V_{n_g}^c)\bigg|\notag\\
&\leq \left\|g\right\|_{\sup}^2 \sum_{\beta\in \mathcal{A}_{n_g}}\bigg|\nu_{h_\alpha|_{X^{(m)}}}^{(m)}(X_\beta\setminus V_{n_g})-\nu_{h_\alpha}(X_\beta\setminus V_{n_g})\bigg|\notag\\
&\leq\varepsilon,\notag
\end{align}
we can use (\ref{E:approx1}) and (\ref{E:approx2}) to obtain
\begin{equation}\label{E:integralsclose}
\bigg|\sum_{\alpha\in\mathcal{A}_n}\int_{X_\alpha\setminus V_{n_g}} g|_{X^{(m)}}^2 d\nu_{h_\alpha|_{X^{(m)}}}^{(m)}- \sum_{\alpha\in\mathcal{A}_n}\int_{X_\alpha\setminus V_{n_g}} g^2 d\nu_{h_\alpha}\bigg|<\frac{3\varepsilon}{5}.
\end{equation}
On the other hand, we have 
\[\left\|g|_{X^{(m)}}\cdot b|_{X^{(m)}}\right\|_{\mathcal{H}^{(m)}}^2=\sum_{\alpha\in\mathcal{A}_n} \int_{X_\alpha}
g|_{X^{(m)}}^2 d\nu^{(m)}_{h_\alpha|_{X^{(m)}}}+\sum_{\alpha,\alpha'\in\mathcal{A}_n, \alpha'\neq \alpha}\int_{X_\alpha \cap X_{\alpha'}} g|_{X^{(m)}}^2 d\nu^{(m)}_{h_\alpha|_{X^{(m)}},h_{\alpha'}|_{X^{(m)}}}.\]
By (\ref{E:Vngsmall}), the Cauchy-Schwarz inequality for energy measures and Definition \ref{D:finitelyramified} (vi) we see that the second summand on the right hand side is bounded by 
\[\left(\sum_{\alpha\in\mathcal{A}_n} \nu^{(m)}_{h_\alpha|_{X^{(m)}}}(V_{n_g})^{1/2}\right)^2<\frac{\varepsilon}{5},\]
and using (\ref{E:Vngsmall}) once more, we obtain
\begin{equation}\label{E:closetonorm}
\bigg|\left\|g|_{X^{(m)}}\cdot b|_{X^{(m)}}\right\|_{\mathcal{H}^{(m)}}^2-\sum_{\alpha\in\mathcal{A}_n}\int_{X_\alpha\setminus V_{n_g}} g|_{X^{(m)}}^2 d\nu^{(m)}_{h_\alpha|_{X^{(m)}}}\bigg|<\frac{2\varepsilon}{5}.
\end{equation}
Combining (\ref{E:integralsclose}), (\ref{E:closetonorm}) and the fact that $\left\|g\cdot b\right\|_{\mathcal{H}}^2=\sum_{\alpha\in\mathcal{A}_n}\int_{X_\alpha} g^2 d\nu_{h_\alpha}$, we arrive at (\ref{E:limfr}).
\end{proof}

\appendix

\section{Generalized strong resolvent convergence}\label{S:generalized convergence}

The notation in this section is different from that in the main text. We review a special case of the notion of convergence for bilinear forms as studied in \cite{Toelle06} (and, among more general results, also in \cite{Toelle10}). It covers in particular the case of coercive closed forms, \cite{MaRoeckner92}. The results in \cite{Toelle06} are generalization of results in \cite[Section 3]{Hino98} to the framework of varying Hilbert spaces in \cite{KuwaeShioya03}. 

In \cite[Subsections 2.2 - 2.7]{KuwaeShioya03} a concept of convergence $H_m\to H$ of  Hilbert spaces $H_m$ to a Hilbert space $H$ was introduced, including a suitable notion of generalized strong resolvent convergence for self-ajoint operators, cf. \cite[Definition 2.1]{KuwaeShioya03}. A basic tool of the method in \cite{KuwaeShioya03} is a family of identification operators $\Phi_m$ defined on a dense subspace $\mathcal{C}$ of the limit space $H$, each mapping $\mathcal{C}$  into one of the spaces $H_m$.  Let $H$, $H_1$, $H_2$, ... be separable Hilbert spaces. The sequence $(H_m)_m$ is said to \emph{converge to $H$ in KS-sense}, $\lim_m H_m=H$, if there are a dense subspace $\mathcal{C}$ of $H$ and operators 
\begin{equation}\label{E:Phi}
\Phi_m:\mathcal{C}\to H_m
\end{equation}
such that 
\begin{equation}\label{E:HmtoH}
\lim_m \left\|\Phi_m w\right\|_{H_m}=\left\|w\right\|_H,\quad w\in\mathcal{C}.
\end{equation}
We recall \cite[Definitions 2.4, 2.5 and 2.6]{KuwaeShioya03}.

\begin{definition}\label{D:KS}\mbox{}
\begin{enumerate}
\item[(i)] A sequence $(u_m)_m$ with $u_m\in H_m$ is said to \emph{converge KS-strongly to $u\in H$} if there is a sequence $(\widetilde{u}_m)_m\subset \mathcal{C}$ such that 
\begin{equation}\label{E:strongconv}
\lim_{n\to\infty}\varlimsup_{m\to\infty}\left\|\Phi_m\widetilde{u}_n-u_m\right\|_{H_m}=0\quad \text{and}\quad \lim_{n\to\infty} \left\|\widetilde{u}_n-u\right\|_H=0.
\end{equation}
\item[(ii)] A sequence $(u_m)_m$ with $u_m\in H_m$ is said to \emph{converge KS-weakly to $u\in H$} if $\lim_m \left\langle u_m, v_m\right\rangle_{H_m} = \left\langle u, v\right\rangle_H$ for every sequence $(v_m)_m$ KS-strongly convergent to $v$.
\item[(iii)] A sequence $(B_m)_m$ of bounded linear operators $B_m:H_m\to H_m$ is said to \emph{converge KS-strongly} to a bounded linear operator $B:H\to H$ if for any sequence $(u_m)_m$ with $u_m\in H_m$ converging KS-strongly to $u\in H$ the sequence $(B_m u_m)_m$  converges KS-strongly to $Bu$. 
\end{enumerate}
\end{definition}

\begin{remark}\label{R:classicalKS}
In the classical case where $H_m\equiv H$ and $\Phi_m\equiv\id_H$ for all $m$ the strong convergence of bounded linear operators $B_m$ defined in (iii) differs from the classical definition of strong convergence of bounded linear operators on Hilbert spaces, as pointed out in \cite[Section 2.3]{KuwaeShioya03}. However, a sequence $(B_m)_m$ of bounded linear operators $B_m:H\to H$ admitting a uniform bound in operator norm $\sup_m \left\|B_m\right\|<+\infty$ converges KS-strongly to a bounded linear operator $B:H\to H$ if and only if it converges strongly to $B$ in the usual sense,  \cite[Lemma 2.8 (1)]{KuwaeShioya03}. 
\end{remark}

Now suppose that $(A_m)_m$ is a sequence of linear operators $A_m:H_m\to H_m$ each of which generates a $C_0$-semigroup and also
$A:H\to H$ is the generator of a $C_0$-semigroup. Suppose that there exist constants $\omega\in\mathbb{R}$ and $M>0$ such that the resolvent sets of each $A_m$ and of $A$ contain $(\omega,+\infty)$ and for any positive integer $n$ and any $\lambda>\omega$ we have $\sup_m \left\|(\lambda-A_m)^{-n}\right\|\leq M(\lambda-\omega)^{-n}$ and $\left\|(\lambda-A)^{-n}\right\|\leq M(\lambda-\omega)^{-n}$. In this situation we say that the $A_m$ \emph{converge} to $A$ \emph{in KS-generalized strong resolvent sense} if for some (hence all) $\lambda>\omega$ the $\lambda$-resolvent operators $R_{\lambda}^{A_m}=(\lambda-A_m)^{-1}$ of the $A_m$ converge KS-strongly to the $\lambda$-resolvent operator $R_\lambda^A=(\lambda-A)^{-1}$ of $A$.

\begin{remark}
For any $\lambda>\omega$ the sequence $(R_{\lambda}^{A_m})_m$ satisfies $\sup_m \big\|R_{\lambda}^{A_m}\big\|<M(\lambda-\omega)^{-1}$. In the classical case where $H_m\equiv H$  and $\Phi_m\equiv \id_H$ for all $m$ we therefore observe that the sequence of operators $(A_m)_m$ as in (iv) converges to $A$ as in (iv) in the KS-generalized strong resolvent sense if and only if it converges to $A$ in the usual strong resolvent sense, see \cite[Section 8.1]{Kato80} (or \cite[Section VIII.7]{RS80} for the self-adjoint case).
\end{remark}

One can also introduce a generalization of Mosco convergence for coercive closed forms (not necessarily symmetric). The following definition is a shorted version for coercive closed forms, \cite{MaRoeckner92}, of \cite[Definition 7.14]{Toelle10} (see also \cite[Definition 2.43]{Toelle06}) sufficient for our purposes. We use notation (\ref{E:symmetricpart}) to denote the symmetric part of a bilinear form.

\begin{definition}\label{D:generalized convergence}
A sequence $((\mathcal{Q}^{(m)}, \mathcal{D}(\mathcal{Q}^{(m)})))_m$ of coercive closed forms $(\mathcal{Q}^{(m)}, \mathcal{D}(\mathcal{Q}^{(m)}))$ on $H_m$, respectively, with uniformly bounded sector constants, $\sup_m K_m<+\infty$, is said to \emph{converge in the KS-generalized Mosco sense} to a coercive closed form $(\mathcal{Q}, \mathcal{D}(\mathcal{Q}))$ on $H$ if there exists a subset $\mathcal{C} \subset \mathcal{D}(\mathcal{Q})$, dense in $\mathcal{D}(\mathcal{Q})$, and the following two conditions hold:
\begin{itemize}
\item[(i)] If $(u_m)_m$ KS-weakly converges to $u$ in $H$ and satisfies $\varliminf_m \widetilde{\mathcal{Q}}^{(m)}_1(u_m) <\infty$, then $u \in \mathcal{D}(\mathcal{Q})$. 
\item[(ii)] For any sequence $(m_k)_k$ with $m_k \uparrow \infty$, any $w \in \mathcal{C}$, any $u \in \mathcal{D} (\mathcal{Q})$ and any sequence $(u_k)_k$, $u_k \in H_{m_k}$, converging KS-weakly to $u$ and such that $\sup_k \widetilde{\mathcal{Q}}_1^{(m_k)}(u_k)<\infty$, there exists a sequence $(w_k)_k$, $w_k \in H_{m_k}$, converging KS-strongly to $w$ and such that 
\[\varliminf_k \mathcal{Q}^{(m_k)}(w_k, u_k)\leq \mathcal{Q}(w,u).\]
\end{itemize}
\end{definition}

In \cite{Hino98, Toelle06, Toelle10} one can find further details. The next Theorem is a special case of \cite[Theorem 7.15, Corollary 7.16 and Remark 7.17]{Toelle10} (see also \cite[Theorem 2.4.1 and Corollary 2.4.1]{Toelle06}), which generalize \cite[Theorem 3.1]{Hino98}.

\begin{theorem}\label{T:convergence of forms}
For each $m$ let $(\mathcal{Q}^{(m)}, \mathcal{D}(\mathcal{Q}^{(m)}))$  be a coercive closed form on $H_m$ and assume that the corresponding sector constants are uniformly bounded, $\sup_m K_m<+\infty$. Let $\big(G^{\mathcal{Q}^{(m)}}_\alpha\big)_{\alpha>0}$, $\big(T_t^{\mathcal{Q}^{(m)}}\big)_{t>0}$ and $(\mathcal{L}^{\mathcal{Q}^{(m)}},\mathcal{D}(\mathcal{L}^{\mathcal{Q}^{(m)}}))$ be the associated resolvent, semigroup and generator on $H_m$. Suppose 
that $(\mathcal{Q}, \mathcal{D}(\mathcal{Q}))$ is a coercive closed form on $H$ with resolvent $\big(G^{\mathcal{Q}}_\alpha\big)_{\alpha>0}$, semigroup $\big(T_t^{\mathcal{Q}}\big)_{t>0}$ and generator $(\mathcal{L}^{\mathcal{Q}},\mathcal{D}(\mathcal{L}^{\mathcal{Q}}))$. Then the following are equivalent:
\begin{itemize}
\item[(1)] The sequence of forms $(\mathcal{Q}^{(m)}, \mathcal{D}(\mathcal{Q}^{(m)}))_m$ converges to $(\mathcal{Q}, \mathcal{D}(\mathcal{Q}))$ in the KS-generalized Mosco sense.
\item[(2)] The sequence of operators $\big(G^{\mathcal{Q}^{(m)}}_\alpha\big)_{m}$ converges to $G^{\mathcal{Q}}_\alpha$ KS-strongly for any $\alpha>0$.
\item[(3)] The sequence of operators $\big(T_t^{\mathcal{Q}^{(m)}}\big)_{m}$ converges to $T_t^{\mathcal{Q}}$ KS-strongly for any $t>0$.
\item[(4)] The sequence of operators $(\mathcal{L}^{\mathcal{Q}^{(m)}},\mathcal{D}(\mathcal{L}^{\mathcal{Q}^{(m)}}))$ converges to $(\mathcal{L}^{\mathcal{Q}},\mathcal{D}(\mathcal{L}^{\mathcal{Q}}))$ in the KS-generalized strong resolvent sense.
\end{itemize}
\end{theorem}

\begin{remark}\mbox{}\label{R:simplify}
Theorem \ref{T:convergence of forms} and Definition \ref{D:generalized convergence} provide a characterization of convergence in the (KS-generalized) strong resolvent sense in terms of the associated bilinear forms. In the case of symmetric forms these conditions differ from those originally used in \cite[Definition 2.1.1 and Theorem 2.4.1]{Mosco94} and \cite[Definition 2.11 and Theorem 2.4]{KuwaeShioya03}, see \cite[Remark 3.4]{Hino98}
\end{remark}



\begin{thebibliography}
\normalsize

\bibitem{AkkermansDunneTeplyaev}
E. Akkermans, G. Dunne, A. Teplyaev, \emph{Physical consequences of complex dimensions of fractals}, Europhys.
Lett. {\bf 88} (2009) 40007.

\bibitem{Allain75}
G. Allain, \emph{Sur la repr\'esentation des formes de Dirichlet},
Ann. Inst. Fourier {\bf 25} (1975), 1--10.

\bibitem{Alonso}
P. Alonso Ruiz, \emph{Explicit Formulas for Heat Kernels on Diamond Fractals}, 
Commun. Math. Phys. {\bf 364} (2018), 1305--1326.

\bibitem{AmbrosioHonda17}
L. Ambrosio, S. Honda, \emph{New stability results for sequences of metric measure spaces with uniform Ricci bounds from below}, in: \emph{Measure Theory in Non-smooth Spaces}, ed. N. Gigli, deGruyter, New York, 2017, pp. 1--51.

\bibitem{AmbrosioStraTrevisan17}
L. Ambrosio, F. Stra, D. Trevisan, \emph{Weak and strong convergence of derivations and stability of flows with respect to MGH convergence}, J. Funct. Anal. {\bf 272} (3) (2017), 1182--1229.

\bibitem{Attouch}
H. Attouch, \emph{Variational Convergence for Functions and Operators}, Pitman, London, 1984.

\bibitem{Azzam}
J. Azzam, M. Hall, R.S. Strichartz, \emph{Conformal energy, conformal Laplacian, and energy measures on the Sierpinski gasket},
Trans. Amer. Math. Soc. {\bf 360} (2008), 2089--2130.

\bibitem{Ba} 
M. T. Barlow, \emph{Diffusions on fractals}, Lectures on Probability Theory and Statistics (Saint-Flour, 1995), 1--121,
Lecture Notes in Math. \textbf{1690}, Springer, Berlin, 1998.


\bibitem{BB89}
M. Barlow, R. Bass, \emph{The construction of Brownian motion on the Sierpinski carpet}, Ann. Inst. H. Poincar\'e {\bf 25} (1989) 225--257.

\bibitem{BBK06}
M. Barlow, R. F. Bass, T. Kumagai, \emph{Stability of parabolic Harnack inequalities on metric measure spaces},
J. Math. Soc. Japan {\bf 58}(2), 485--519.

\bibitem{BBKT10}
M. Barlow, R. F. Bass, T. Kumagai, A. Teplyaev, \emph{Uniqueness of Brownian motion on Sierpi\`nski carpets},
J. Eur. Math. Soc. (JEMS) 12 (2010), no. 3, 655--701.


\bibitem{BarlowPerkins}
M.T. Barlow, E.A. Perkins, \emph{Brownian motion on the Sierpinski gasket}, 
Probab. Theory Relat. Fields {\bf 79} (1988), 543--624.

\bibitem{BK16}
F. Baudoin, D. J. Kelleher, \emph{Differential forms on Dirichlet spaces and Bakry- \'{E}mery estimates on metric graphs}, Trans. Amer. Math. Soc. {\bf 371}(5) (2019), 3145--3178.

\bibitem{BBST99}
O. Ben-Bassat, R.S. Strichartz, A. Teplyaev, \emph{What is not in the domain of the Laplacian on a Sierpinski gasket type fractal}, J. Funct. Anal. {\bf 166} (1999), 197--217.

\bibitem{BH91}
N. Bouleau, F. Hirsch, \emph{Dirichlet Forms and Analysis on Wiener Space},
deGruyter Studies in Math. 14, deGruyter, Berlin, 1991.

\bibitem{CHT18}
J.P. Chen, M. Hinz, A. Teplyaev, \emph{From non-symmetric particle systems to non-linear PDEs on fractals}, in 'Stochastic Partial Differential Equations and Related Fields - In Honor of Michael R\"ockner, SPDERF, Bielefeld, Germany, October 2016', Springer Proc. in Math. and Stat. {\bf 229}, Springer Intl. Publ., 2018, pp. 503--513.


\bibitem{CS03}
F. Cipriani, J.-L. Sauvageot, \emph{Derivations as square roots of Dirichlet forms},
J. Funct. Anal. {\bf201} (2003), 78--120.

\bibitem{CS09}
F. Cipriani, J.-L. Sauvageot, \emph{Fredholm modules on p.c.f. self-similar fractals and their conformal geometry},
Comm. Math. Phys. {\bf286} (2009), 541--558.

\bibitem{Croydon}
D. Croydon, \emph{Scaling limits of stochastic processes associated with resistance forms}, Ann. Inst. H. Poincaré Probab. Statist. {\bf 54}(4) (2018), 1939--1968.

\bibitem{Eb99}
A. Eberle, \emph{Uniqueness and non-uniqueness of semigroups generated by singular diffusion operators}, Lect. Notes Math. {\bf 1718}, Springer, New York, 1999.

\bibitem{EngelNagel}
K.-J. Engel, R. Nagel, \emph{One-Parameter Semigroups for Evolution Equations},
Graduate Texts in Math. 194, Springer, New York, 2000.

\bibitem{Evans}
L.C. Evans, \emph{Partial Differential Equations}, Graduate Studies in Mathematics, Vol. 19, Amer. Math. Soc., Providence, Rhode Island, 1998.

\bibitem{ExnerPost09}
P. Exner, O. Post, \emph{Approximation of quantum graph vertex couplings by scaled Schr\"odinger operators on thin branched manifolds}, J. Phys. A {\bf 42} (2009), 415305, 22.


\bibitem{FaHu99}
K.J. Falconer, J. Hu, \emph{Non-linear elliptical equations on the Sierpi\'nski gasket}, J. Math. Anal. Appl. {\bf 240} (2) (1999), 552--573. 

\bibitem{FaHu01}
K.J. Falconer, J. Hu, \emph{Nonlinear diffusion equations on unbounded fractal domains}, J. Math. Anal. Appl. {\bf 256} (2), 606--624.

\bibitem{FiKu04}
P.J. Fitzsimmons, K. Kuwae, \emph{Non-symmetric perturbations of symmetric
Dirichlet forms}, J. Funct. Anal. {\bf 208} (2004), 140--162.

\bibitem{FOT94}
M. Fukushima, Y. Oshima and M. Takeda, \emph{Dirichlet forms and symmetric Markov processes},
deGruyter, Berlin, New York, 1994.

\bibitem{FKW07}
S.A. Fulling, P. Kuchment, J.H. Wilson, \emph{Index theorems for quantum graphs}, J. Phys. A: Math. Theor. {\bf 40} (2007), 14165--14180.

\bibitem{Gigli15}
N. Gigli, \emph{On the differential structure of metric measure spaces and applications}, Mem. Amer. Math. Soc. {\bf 236} (1113) (2015).

\bibitem{Gigli17}
N. Gigli, \emph{Non-smooth differential geometry - an approach tailored for spaces with Ricci curvature bounded from below}, Mem. Amer. Math. Soc. {\bf 251} (11) (2017). 


\bibitem{GT10}
D. Gilbarg, N. Trudinger, \emph{Partial Differential Equations of Second Order},
Springer, New York, 1998.

\bibitem{Grigoryan}
A. Grigor'yan, \emph{Introduction to Analysis on Graphs}, University Lecture Series 71,  Amer. Math. Soc., 2018.

\bibitem{Hambly92}
B.M. Hambly, \emph{Brownian motion on a homogeneous random fractal}, Probab. Theory Rel. Fields {\bf 94} (1992), 1--38.

\bibitem{Hambly97}
B.M. Hambly, \emph{Brownian motion on a random recursive Sierpinski gasket}, Ann. Probab. {\bf 25} (3) (1997), 1059--1102.

\bibitem{HamblyKumagai}
B.M. Hambly, T. Kumagai, \emph{Diffusion on the scaling limit of the critical percolation cluster in the diamond hierarchical lattice}, Commun. Math. Phys. {\bf 295}(1) (2010), 29--69.

\bibitem{HamblyNyberg03}
B. Hambly, S. Nyberg, \emph{Finitely ramified graph-directed fractals, spectral asymptotics and the multidimensional
renewal theorem}, Proc. Edinb. Math. Soc. {\bf 46}(2) (2003), 1--34.

\bibitem{Hei01} 
J. Heinonen,  \emph{Lectures on analysis on metric spaces.}\, Universitext. Springer-Verlag, New York, 2001.

\bibitem{Hino98}
M. Hino, \emph{Convergence of non-symmetric forms}, J. Math. Kyoto Univ. {\bf 38}(2) (1998), 329--341.


\bibitem{Hino03}
M. Hino, \emph{On singularity of energy measures on self-similar sets},
Probab. Theory Relat. Fields {\bf132} (2005), 265--290.


\bibitem{Hino08}
M. Hino, \emph{Martingale dimension for fractals},
Ann. of Probab. {\bf 36}(3) (2008), 971--991

\bibitem{Hino10}
M. Hino, \emph{Energy measures and indices of Dirichlet forms, with applications to derivatives on some fractals},
Proc. London Math. Soc. {\bf 100}  (2010), 269--302.


\bibitem{HinoKumagai06}
M. Hino, T. Kumagai, \emph{A trace theorem for Dirichlet forms on fractals}, J. Funct. Anal. {\bf 238} (2006), 578--611.

\bibitem{Hino05}
M. Hino, K. Nakahara, \emph{On singularity of energy measures on self-similar sets II},
Bull. London Math. Soc. {\bf 38} (2006), 1019--1032.

\bibitem{Hinz09}
M. Hinz, \emph{Approximation of jump processes on fractals}, Osaka J. Math. {\bf 46} (2009), 141--171.

\bibitem{Hinz15}
M. Hinz, \emph{Magnetic energies and Feynman-Kac-It\^o formulas for symmetric Markov processes}, Stoch. Anal. Appl. {\bf 33}(6) (2015), 1020--1049. 

\bibitem{Hinz16}
M. Hinz, \emph{Sup-norm closable bilinear forms and Lagrangians}, 
Ann. Mat. Pura Appl. {\bf 195}(4) (2016), 1021--1054.


\bibitem{HinzKochMeinert}
M. Hinz, D. Koch, M. Meinert, \emph{Sobolev spaces and calculus of variations on fractals}, in: Analysis, Probability and Mathematical Physics on Fractals, World Scientific, 2020, pp. 419--450.

\bibitem{HinzMeinert19+}
M. Hinz, M. Meinert, \emph{On the viscous Burgers equation on metric graphs and fractals}, J. Fractal Geometry {\bf 7}(2) (2020), 137--182. 

\bibitem{HR16}
M. Hinz, L. Rogers, \emph{Magnetic fields on resistance spaces},  J. Fractal Geometry {\bf 3} (2016),  75--93.

\bibitem{HRT13}
M. Hinz, M. R\"ockner, A. Teplyaev, \emph{Vector analysis for  Dirichlet forms and quasilinear PDE and SPDE on metric measure spaces}, Stoch. Proc. Appl. {\bf 123}(12) (2013), 4373--4406.

\bibitem{HT13}
M. Hinz, A. Teplyaev, \emph{Dirac and magnetic Schr\"odinger operators on fractals},
J. Funct. Anal. {\bf 265}(11) (2013), 2830--2854.

\bibitem{HTams}
M. Hinz,   A. Teplyaev, \emph{Local Dirichlet forms, Hodge theory, and
the Navier-Stokes equations on topologically one-dimensional
fractals}, Trans. Amer. Math. Soc. {\bf 367}  (2015),
1347--1380, Corrigendum in Trans. Amer. Math. Soc. {\bf 369} (2017), 6777--6778.

\bibitem{HT-fgs5}
M. Hinz, A. Teplyaev, \emph{Finite energy coordinates and vector analysis on fractals}, In: Fractal Geometry and Stochastics V, Progress in Probability, vol. {\bf 70},  Birkh\"auser, 2015, pp. 209--227.



\bibitem{HinzTeplyaev15}
M. Hinz,  A. Teplyaev, \emph{Closability, regularity, and  approximation  by  graphs  for  separable bilinear forms}, Zap. Nauchn. Sem. S.-Peterburg. Otdel. Mat. Inst. Steklov. (POMI) {\bf 441} (2015), 299--317. Reprint: J. Math. Sci. {\bf 219}(5) (2016), 807--820.


\bibitem{IRT12}
M. Ionescu, L. Rogers, A. Teplyaev, \emph{Derivations, Dirichlet forms and spectral analysis},
J. Funct. Anal. {\bf 263}(8) (2012), 2141--2169. 


\bibitem{Ka12}
N. Kajino, \emph{Heat kernel asymptotics for the measurable Riemannian structure on the Sierpinski gasket},
Pot. Anal. {\bf 36} (2012), 67--115.


\bibitem{Kasue02}
A. Kasue, \emph{Convergence of Riemannian manifolds and Laplace operators I}, Ann. Inst. Fourier {\bf 52} (2002), 1219--1257.

\bibitem{Kasue06}
A. Kasue, \emph{Convergence of Riemannian manifolds and Laplace operators II}, Pot. Anal. {\bf 24} (2006), 137--194.

\bibitem{Kato80}
T. Kato, \emph{Perturbation Theory for Linear Operators}, Springer, New York, 1980.

\bibitem{KLWbook}
M. Keller, D. Lenz, R. K. Wojciechowski, \emph{Graphs and discrete Dirichlet spaces}, Springer, to appear.

\bibitem{Ki89}
J. Kigami, \emph{A harmonic calculus on the Sierpinski space}, Japan J. Appl.
Math. {\bf 6} (1989), 259–290.

\bibitem{Ki93}
J. Kigami, \emph{Harmonic calculus on p.c.f. self-similar sets}, Trans. Amer. Math. Soc. {\bf 335} (1993), 721--755.

\bibitem{Ki93b} J. Kigami,
\emph{Harmonic metric and Dirichlet form on the Sierpi\'nski gasket},
Asymptotic problems in probability theory: stochastic models and
diffusions on fractals (Sanda/Kyoto, 1990), 201--218,
Pitman Res. Notes Math. Ser., \textbf{283},
Longman Sci. Tech., Harlow, 1993.

\bibitem{Ki01}
J. Kigami, \emph{Analysis on Fractals}, Cambridge Univ. Press, Cambridge, 2001.

\bibitem{Ki03}
J. Kigami, \emph{Harmonic analysis for resistance forms}, J. Funct. Anal. {\bf 204} (2003), 525--544.


\bibitem{Ki08}
J. Kigami, \emph{Measurable Riemannian geometry on the Sierpinski gasket: the Kusuoka measure and the Gaussian heat kernel estimate}, Math. Ann. {\bf 340} (2008), 781--804.


\bibitem{Ki12}
J. Kigami, \emph{Resistance forms, quasisymmetric maps and heat kernel estimates},
Mem. Amer. Math. Soc. {\bf 216} (2012), no. 1015.


\bibitem{Kolesnikov05}
A.V. Kolesnikov, \emph{Convergence of Dirichlet forms with changing speed measures on $\mathbb{R}^d$}, Forum Math. {\bf 17}(2) (2005), 225--259.

\bibitem{Kolesnikov06}
A.V. Kolesnikov, \emph{Mosco convergence of Dirichlet forms in infinite dimensions with
changing reference measures}, J. Func. Anal. {\bf 230}(2) (2006), 382--418.


\bibitem{KS99}
V. Kostrykin, R. Schrader, \emph{Kirchhoff's rule for quantum wires}, 
J. Phys. A {\bf 32} (1999), 595--630.

\bibitem{KS00}
V. Kostrykin, R. Schrader, \emph{Kirchhoff's rule for quantum wires. II: The inverse problem with possible applications to quantum computers}, 
Fortschr. Phys. {\bf 48} (2000), 703--716.

\bibitem{Ku04}
P. Kuchment, \emph{Quantum graphs I. Some basic structures}, Waves in Random media {\bf 14} (2004), 107--128.

\bibitem{Ku05}
P. Kuchment, \emph{Quantum graphs II. Some spectral properties of quantum and combinatorial graphs},
J. Phys. A: Math. Theory {\bf 38} (2005), 4887--4900.


\bibitem{KumagaiSturm05}
T. Kumagai, K.T. Sturm: \emph{Construction of diffusion processes on fractals, $d$-sets and general metric measure spaces}, J. Math. Kyoto Univ. {\bf 45} (2) (2005), 307--327.


\bibitem{Ku89}
S. Kusuoka, \emph{Dirichlet forms on fractals and products of random matrices}, Publ. Res. Inst. Math. Sci. {\bf 25} (1989), 659--680.

\bibitem{Ku93} S. Kusuoka,
\emph{Lecture on diffusion process on nested fractals.}\
Lecture Notes in Math. \textbf{1567} 39--98,
Springer-Verlag, Berlin, 1993.



\bibitem{KuwaeShioya03}
K. Kuwae and T. Shioya: \textit{Convergence of spectral structures},
Comm. Anal. Geom. \bf11 \normalfont (4) (2003), 599--673. 

\bibitem{LiuQian17}
X. Liu, Zh. Qian, \emph{Parabolic type equations associated with the Dirichlet form on the Sierpinski gasket}, Probab. Theory Relat. Fields {\bf 175} (2019), 1063--1098. 

\bibitem{MaRoeckner92}
Z.-M. Ma and M. R{\"o}ckner, \textit{Introduction to the Theory of (Nonsymmetric) Dirichlet Forms}, Universitext.
Springer, Berlin (1992).

\bibitem{Meyers}
R. Meyers, R.S. Strichartz, A. Teplyaev, \emph{Dirichlet forms on the Sierpinski gasket}, Pacific
J. Math. {\bf 217}(2004), 149--174.


\bibitem{Mosco94}
U. Mosco, \emph{Composite media and asymptotic Dirichlet forms},
J. Funct. Analysis, {\bf 123} (1994), 368--421.


\bibitem{Mugnolo14}
D. Mugnolo, \emph{Semigroup Methods for Evolution Equations on Networks}, Understanding Complex Systems, 
Springer International Publishing Switzerland, 2014.


\bibitem{MugnoloNittkaPost13}
D. Mugnolo, R. Nittka, O. Post, \emph{Norm convergence of sectorial operators on varying Hilbert spaces},
Operators and Matr. {\bf 7}(4) (2013), 955--995.


\bibitem{P83}
A. Pazy, \emph{Semigroups of Linear Operators and Applications to Partial Differential Equations}, Springer, New York, 1983.

\bibitem{Post06}
O. Post, \emph{Spectral convergence of quasi-one-dimensional spaces}, Ann. Henri Poincar\'e {\bf 7}(2006), 933--973.

\bibitem{Post12}
O. Post, \emph{Spectral Analysis on Graph-like Spaces}, Lect. Notes Math. {\bf 2039}, Springer, Berlin, 2012.

\bibitem{PostSimmer17}
O. Post, J. Simmer, \emph{Approximation of fractals by discrete graphs: norm resolvent and spectral convergence},
J. Integr. Equ. Oper. Theory (2018) 90:68.  

\bibitem{PostSimmer18}
O. Post, J. Simmer, \emph{Approximation of fractals by manifolds and other graph-like spaces},
preprint (2018), arXiv:1802.02998.

\bibitem{RS80}
M. Reed, B. Simon, \emph{Methods of Modern Mathematical Physics I: Functional Analysis}, Academic Press, San Diego, 1980.

\bibitem{RogersTeplyaev10}
L. Rogers, A. Teplyaev, \emph{Laplacians on the basilica Julia set}, Commun. Pure Appl. Anal. {\bf 9}(2010), 211--231. 

\bibitem{Steinhurst}
B. Steinhurst, \emph{Uniqueness of locally symmetric Brownian motion on Laakso spaces}, Pot. Anal. {\bf 38}(1) (2013), 281--298.

\bibitem{Str03}
R.S. Strichartz, \emph{Fractafolds based on the Sierpinski gasket and their spectra}, Trans. Amer. Math. Soc. {\bf 355} (2003),
4019--4043.


\bibitem{Str06}
R.S. Strichartz, \emph{Differential Equations on Fractals: A Tutorial}, Princeton Univ. Press, Princeton 2006.

\bibitem{Str16}
R.S. Strichartz, \emph{'Graph paper' trace characterizations of functions of finite energy}, J. d'Anal. Math. {\bf 128} (2016), 239--260.

\bibitem{StrU00}
R.S. Strichartz, M. Usher, \emph{Splines on fractals}, Math. Proc. Camb. Phil. Soc. {\bf 129} (2000), 331--360.


\bibitem{Suzuki18}
K. Suzuki, \emph{Convergence of non-symmetric diffusion processes on RCD spaces},
Calc. Var. (2018) 57:120.


\bibitem{T08}
A. Teplyaev, \emph{Harmonic coordinates on fractals with finitely ramified cell structure}, Canad. J. Math. {\bf 60} (2008), 457--480.


\bibitem{Toelle06}
J. T\"olle, \emph{Convergence of non-symmetric forms with changing reference measures}, Diploma thesis, Bielefeld University, 2006.

\bibitem{Toelle10}
J. T\"olle, \emph{Variational convergence of nonlinear partial differential operators on varying Banach spaces},
Ph.D. thesis, Bielefeld University, 2010.



\bibitem{W00}
N. Weaver, \emph{Lipschitz algebras and derivations II. Exterior differentiation}, J. Funct. Anal. {\bf 178} (2000), 64-112.

\end{thebibliography}
\end{document}